\documentclass[a4paper, 11pt]{amsart}

\usepackage[leqno]{amsmath}
\usepackage{amssymb,amsthm,MnSymbol}
\usepackage{mathrsfs}
\usepackage[all]{xy}
\usepackage[utf8]{inputenc}
\usepackage{dsfont}
\usepackage{enumitem}
\usepackage{hyperref}
\usepackage{tikz}
\usepackage{tikz-cd}
\usepackage{cleveref}
\usepackage[divide={2.8cm,*,2.8cm}]{geometry}

\numberwithin{equation}{section}

\theoremstyle{plain}
\newtheorem{thm}{Theorem}[section]
\newtheorem{thmx}{Theorem}
\newtheorem{prop}[thm]{Proposition}
\newtheorem{lem}[thm]{Lemma}
\newtheorem{cor}[thm]{Corollary}

\theoremstyle{definition}
\newtheorem{defn}[thm]{Definition}
\newtheorem{set}[thm]{Setup}
\newtheorem{rem}[thm]{Remark}
\newtheorem{ex}[thm]{Example}

\newcommand{\A}{{\mathcal{A}}}

\newcommand{\B}{{\mathcal{B}}}
\newcommand{\C}{{\mathcal{C}}}

\newcommand{\D}{{\mathcal{D}}}
\newcommand{\E}{{\mathcal{E}}}
\newcommand{\M}{{\mathcal{M}}}

\newcommand{\He}{{\mathcal{H}}}
\newcommand{\Hea}{{\mathcal{H}_\aff}}
\newcommand{\F}{{\mathsf{F}}}
\newcommand{\Fc}{{\mathcal{F}}}
\newcommand{\G}{{\mathsf{G}}}
\newcommand{\R}{{\mathsf{R}}}
\newcommand{\Ll}{{\mathsf{L}}}
\newcommand{\LM}{{\mathcal{L}^G_M}}
\newcommand{\RM}{{\mathcal{R}^G_M}}
\newcommand{\Ss}{{\mathcal{S}}}
\newcommand{\m}{{\mathfrak{m}}}
\newcommand{\T}{{\mathcal{T}}}
\newcommand{\W}{{\mathcal{W}}}
\newcommand{\XX}{{\mathbb{X}}}
\newcommand{\Z}{{\mathbb{Z}}}

\newcommand{\Coind}{{\operatorname{Coind}}}
\newcommand{\coker}{{\operatorname{coker}}}
\newcommand{\diag}{{\operatorname{diag}}}
\newcommand{\End}{{\operatorname{End}}}
\newcommand{\Ext}{{\operatorname{Ext}}}

\newcommand{\Ho}{{\operatorname{Ho}}}
\newcommand{\Hom}{{\operatorname{Hom}}}
\newcommand{\id}{{\operatorname{id}}}

\newcommand{\ind}{{\operatorname{ind}}}
\newcommand{\Ind}{{\operatorname{Ind}}}
\newcommand{\Ker}{{\operatorname{Ker}}}

\newcommand{\Loc}{{\operatorname{Loc}}}
\newcommand{\Mod}{{\operatorname{Mod}}}

\newcommand{\pd}{{\operatorname{pd}}}

\newcommand{\Rep}{{\operatorname{Rep}}}
\newcommand{\Res}{{\operatorname{Res}}}

\newcommand{\val}{{\operatorname{val}}}
\newcommand{\aff}{{\operatorname{aff}}}

\newcommand{\uu}[1]{\underline{\underline{#1}}}
\newcommand{\abs}[1]{\left| #1 \right|}

\renewcommand{\Im}{{\operatorname{Im}}}

\newenvironment{psmallmatrix}
  {\left(\begin{smallmatrix}}
  {\end{smallmatrix}\right)}

\begin{document}


\title{Parabolic induction in the homotopy category of pro-$p$ Iwahori-Hecke modules}
\date{}
\author{Nicolas Dupr\'e}
\address{Universit\"at Duisburg-Essen\\
Fakult\"at f\"ur Mathematik\\
Thea-Leymann-Stra{\ss}e 9\\
D--45127 Essen, Germany}
\email{nicolas.dupre@uni-due.de}


\begin{abstract}
Let $G$ be the group of rational points of a split connected reductive group over a non-archimedean local field of residue characteristic $p$, and let $\He$ denote the pro-$p$ Iwahori-Hecke algebra of $G$ over a field of characteristic $p$. We study the parabolic induction functor for $\He$-modules in terms of the Gorenstein projective model structures introduced by Hovey. Let $\Ho(\He)$ denote the associated homotopy category of this model structure. We show that $\Ho(\He)$ and its thick subcategory generated by the essential images of finitely many parabolic induction functors are related via a recollement of triangulated categories. We then investigate the isomorphism classes of simple $\He$-modules in $\Ho(\He)$ and give a complete classification for simple supersingulars when $G=\mathrm{GL}_n$.
\end{abstract}

\maketitle
\thispagestyle{empty}
\footnotetext{{\it 2020 Mathematics Subject Classification}. Primary 20C08, 18N40, 18N55.}

\setcounter{tocdepth}{1}
\tableofcontents


\section{Introduction}

\subsection{Background and motivation} Let $p$ be a prime and let $\mathfrak{F}$ be a non-archimedean local field of residue characteristic $p$. Let $G=\mathbb{G}(\mathfrak{F})$ be the group of $\mathfrak{F}$-rational points of a split connected reductive group $\mathbb{G}$ over $\mathfrak{F}$. Let $k$ be a field of characteristic $p$ and denote by $\Rep_k^\infty(G)$ the category of $k$-linear smooth representations of $G$, i.e. those representations such that the stabiliser of each vector is open. This category lies on one side of the mod-$p$ local Langlands program.

Let $I$ be a pro-$p$ Iwahori subgroup of $G$ and let $\He=k[I\backslash G/I]$ be the corresponding \emph{pro-$p$ Iwahori-Hecke algebra} of $G$, i.e.\ the convolution algebra of $I$-biinvariant compactly supported functions on $G$. Write $\Mod(\He)$ for the category of right $\He$-modules. Then taking $I$-invariants defines a functor $U:\Rep_k^\infty(G)\to\Mod(\He)$ which sends non-zero representations to non-zero modules. This functor is expected to provide a strong relation between the two categories, however its behaviour is not sufficiently understood beyond $G=\mathrm{GL}_2(\mathbb{Q}_p)$ and a few related cases (cf.\ the work of Ollivier \cite{Oll09}). In particular, one of the major open problems in the theory is to determine the relationship between the isomorphism classes of irreducible admissible smooth representations of $G$ and of simple $\He$-modules.

In previous work \cite{KD21} with Kohlhaase, we showed that $U$ restricts to an equivalence between an explicit full subcategory of $\Rep_k^\infty(G)$ and the full subcategory $\mathrm{GProj}(\He)$ of \emph{Gorenstein projective} $\He$-modules. These modules are a key component in the definition of a model structure on $\Mod(\He)$ due to Hovey \cite{Hov2}. We used this in \cite{KD21} to construct and study a model structure on $\Rep_k^\infty(G)$ and the derived functor $\mathbf{R}U$ between the associated homotopy categories. We generally obtained that that the formal properties of this functor are improved at the homotopy level, for instance $\mathbf{R}U$ is essentially surjective. This suggests that one should instead investigate the relationship between the \emph{homotopy classes} of irreducible admissible smooth representations of $G$ and of simple $\He$-modules.

This paper is concerned with the Hecke side of that question, i.e.\ we wish to study the simple $\He$-modules viewed as objects of the homotopy category $\Ho(\He)$ of $\Mod(\He)$. A first step towards this was given by Koziol in \cite{Koz} where the simple $\He$-modules of finite projective dimension were classified. When rephrased in our language, this gives a classification of the simple modules which are isomorphic to zero in $\Ho(\He)$. In particular, it was shown that the simple supersingular $\He$-modules are generically never zero in $\Ho(\He)$.

Our main goal in this work is to treat the harder problem of determining whether two nonisomorphic simple $\He$-modules of infinite projective dimension can become isomorphic in $\Ho(\He)$. If $M$ denotes a standard Levi subgroup of $G$ with associated Hecke algebra $\He_M$ then a key ingredient in Abe's construction of the simple modules in \cite{Abe19} is the \emph{parabolic induction} functor $\Ind^{\He}_{\He_M}:\Mod(\He_M)\to\Mod(\He)$. We therefore first aim to study the properties of $\Ind^{\He}_{\He_M}$ in the language of Hovey's Gorenstein projective model structure. This allows us to essentially reduce our question to the case of supersingular modules. In many situations of interest, we are then able to obtain a rather complete understanding of the situation (cf.\ \Cref{second_main_thm} \& \Cref{third_main_thm} below for details).

\subsection{Main results} We first show that the parabolic induction functor induces decompositions of $\Ho(\He)$ via various recollements. To state this more precisely, we note that $\Ho(\He)$ has the same class of objects as $\Mod(\He)$, being a localisation of it, and that it was proved by Hovey to be canonically a triangulated category. Fixing a finite collection $\M$ of proper standard Levi subgroups of $G$, we may therefore let $\Ss^G_\M$ be the thick subcategory of $\Ho(\He)$ generated by the essential images of $\Ind^{\He}_{\He_{M}}$ for all $M\in\M$. When $\M$ equals the set of all proper standard Levi subgroups of $G$, we simply write $\Ss^G_{\mathrm{all}}$ for this thick subcategory.

Also, associated to $G$ is a group $\Omega$ which is equal to the length zero elements in the Iwahori-Weyl group $W=N_G(T)/(T\cap I')$ (where $I'$ is the \emph{Iwahori} subgroup of $G$ with pro-$p$ radical $I$). This group $\Omega$ is known to be a finitely generated abelian group, and we denote by $\Omega_{\text{tor}}$ its (finite) torsion-subgroup. With this notation established, our first main result then says:

\begin{thmx}\label{first_main_thm}
\begin{enumerate}
\item The inclusion $i^G_\M:\Ss^G_\M\hookrightarrow \Ho(\He)$ is part of a recollement
\[
\xymatrix@C=0.8cm{
\Ss^G_\M \ar[rrr]^{i^G_\M} &&& \Ho(\He) \ar[rrr]^{\pi^G_\M} \ar @/_1.5pc/[lll]_{\mathcal{L}^G_\M } \ar
 @/^1.5pc/[lll]^{\mathcal{R}^G_\M} &&& \Ho(\He)_\M
\ar @/_1.5pc/[lll]_{\ell^G_\M} \ar @/^1.5pc/[lll]^{r^G_\M}
 }
\]
of triangulated categories, where $\Ho(\He)_\M=\Ho(\He)/\Ss^G_\M$.
\item Assume that $k$ is algebraically closed, that the root system of $\mathbb{G}$ is irreducible and that $p\nmid\abs{\Omega_{\text{tor}}}$. Suppose that $\m\in\Mod(\He)$ has finite length. Then the following are equivalent:
\begin{itemize}
    \item[(a)] $\m$ lies in the essential image of $\ell^G_{\mathrm{all}}$
    \item[(b)] $\m$ lies in the essential image of $r^G_{\mathrm{all}}$
    \item[(c)] the inclusion $\m_{\mathrm{ss}}\subseteq \m$ is an isomorphism in $\Ho(\He)$, where $\m_{\mathrm{ss}}$ denotes the largest supersingular submodule of $\m$.
\end{itemize}
\end{enumerate}
\end{thmx}

This is proved in \Cref{many_Levi_rec} and \Cref{finite_dim_mod_ss}. Roughly speaking, part (i) of \Cref{first_main_thm} says that the triangulated category $\Ho(\He)$ can be decomposed into two orthogonal thick subcategories, one equal to $\Ss^G_\M$ and the other one equal to $\Ho(\He)_\M$, where the latter category is embedded fully faithfully into $\Ho(\He)$ via either one of the functors $\ell^G_\M$ or $r^G_\M$ (cf.\ \Cref{recollement_defn} and the discussion following it for details). We point out that while the proof of (i) is mostly formal, it however relies on the fact that parabolic induction is fully faithful (cf.\ \Cref{fully_faithful}), which appears to be a new result.

We also point out that the recollement in \Cref{first_main_thm}(i) can be entirely reconstructed from various new model structures on $\Mod(\He)$. This is proved in the more general context of a recollement in the homotopy category of an abelian model category, see \Cref{recollement_model}, \Cref{apply_Becker}, \Cref{other_proj_inj} and \Cref{last_piece} for the details.

The proof of part (ii) of \Cref{first_main_thm} relies on Abe's classification of the simple $\He$-modules in \cite{Abe19}, on Koziol's results on the projective dimension of simple modules from \cite{Koz}, and on a lift to the homotopy level of a computation of Abe of the image of the simple modules under the adjoints of the parabolic induction functor (cf.\ \Cref{Abe_formula_derived}). As a consequence of the latter computation, we also obtain a new proof of a result of Koziol and reduce the study of isomorphism classes of simple $\He$-modules in $\Ho(\He)$ to the case of simple supersingular $\He$-modules, cf. \Cref{half_Koziol}.

Having established this, the next step is to determine when two simple supersingular $\He$-modules can become isomorphic in $\Ho(\He)$. Here the main difficulty is that computing the Hom spaces in $\Ho(\He)$ a priori requires to compute explicit Gorenstein projective replacements in $\Mod(\He)$. Instead, inspired by the arguments of Koziol in \cite{Koz} and by the functorial resolutions of Ollivier-Schneider \cite{OS14}, we use Bruhat-Tits theory. If $F$ denotes a face of the Bruhat-Tits building contained in the closure of a chosen chamber $C$, then there is an associated subalgebra $\He_F$ of $\He$. These subalgebras together generate a Gorenstein subalgebra $\Hea$ of $\He$ called the \emph{affine Hecke algebra}. The simple supersingular Hecke modules were classified by Vign\'eras, cf.\ \cite{Vign3}, and a key input in this classification are the so-called supersingular characters of $\Hea$. Our main insight is to consider the `diagonal' functor
$$
\Delta=(\Res^\Hea_{\He_F})_F:\Mod(\Hea)\to\prod_{F\subseteq \overline{C}}\Mod(\He_F)
$$
which is right Quillen and in fact determines the model structure on $\Mod(\Hea)$ via a right transfer property, cf.\ \Cref{facts_H_F}. This indicates that we should in principle be able to do computations in the product category $\prod_{F\subseteq \overline{C}}\Ho(\He_F)$ instead of $\Ho(\Hea)$. The advantage here is that there is no need to compute replacements in order to compute Hom spaces in $\Ho(\He_F)$ for any $F$, because $\He_F$ is known to be self-injective and this implies that every $\He_F$-module is Gorenstein projective.

More specifically, we have that $\Delta$ lifts to a functor $\Ho(\Delta):\Ho(\Hea)\to\prod_{F\subseteq \overline{C}}\Ho(\He_F)$. One major technical input in this paper is the computation of the image under $\Ho(\Delta)$ of the Hom space in $\Ho(\Hea)$ between two supersingular characters. Roughly, one gets that this image is always zero unless the root system is of type $\Phi_1\times \Phi_2$, where $\Phi_1$ is irreducible of rank 2 and $\Phi_2$ is a product of $A_1$'s, and the two characters are of a very explicit form which we will denote here by $\chi_1^{\text{special}}$ and $\chi_2^{\text{special}}$, see \Cref{characters_Haff_Ho} for the details. In order to go from supersingular characters of $\Hea$ to simple supersingular modules, we then analyse the actions of the group $\Omega$ on these characters and on the Hom spaces between these characters. In particular, given a supersingular character $\chi$ we denote by $\Omega_\chi$ its stabiliser in $\Omega$. We also note that $\Omega$ acts on the faces $F$ contained in $\overline{C}$, and we denote the stabiliser of a face $F$ by $\Omega_F$. By combining this analysis of the $\Omega$-action on Hom spaces in $\Ho(\Hea)$, cf.\ \Cref{Ho_hom_invariants2}, with the above computation of the image of Hom spaces under $\Ho(\Delta)$, we obtain the following:

\begin{thmx}\label{second_main_thm}
Assume that $k$ is algebraically closed and let $\m$ and $\m'$ be two simple supersingular $\He$-modules of infinite projective dimension.
\begin{enumerate}
    \item Suppose that $\mathbb{G}$ is semisimple and simply-connected. Then $\m\cong\m'$ in $\Ho(\He)$ if and only if $\m\cong\m'$ in $\Mod(\He)$.
    \item Suppose that if one of $\Res^\He_\Hea(\m)$ and $\Res^\He_\Hea(\m')$ contains $\chi_1^{\text{special}}$ as a submodule then the other one does not contain $\chi_2^{\text{special}}$. Furthermore, if we fix a supersingular character $\chi$ occuring in $\Res^\He_\Hea(\m)$, assume that there exists a face $F\subseteq\overline{C}$ such that $\Omega_\chi\subseteq \Omega_F$ and $\Res^\Hea_{\He_F}(\chi)$ is not projective. Then $\m\cong\m'$ in $\Ho(\He)$ if and only if $\m\cong\m'$ in $\Mod(\He)$.
\end{enumerate}
\end{thmx}

The above is a combination of \Cref{classification_affine} and \Cref{imperfect_theorem}. We then use these results to give a complete classification of the isomorphism classes of simple supersingular modules in the situation where $G$ is a product of $\mathrm{GL}_n$'s. Roughly, we then get that two nonisomorphic simple supersingular $\He$-modules almost never become isomorphic in $\Ho(\He)$ and when they do it is only in a restricted family of cases. In what follows, we denote by $\widetilde{\Omega}$ the lift of $\Omega$ to $\widetilde{W}=N_G(T)/(T\cap I)$. This group embeds canonically as a subgroup of $\He^\times$.

\begin{thmx}[{\Cref{last_thm}}]\label{third_main_thm}
Assume that $k$ is algebraically closed and suppose that $G$ is a product of $\mathrm{GL}_{n}$'s. Let $\m$ and $\m'$ be two simple supersingular $\He$-modules of infinite projective dimension.
\begin{enumerate}
    \item If $G$ does not have root system of type $A_2\times A_1\times\cdots\times A_1$ (possibly empty product of $A_1$'s), then $\m\cong\m'$ in $\Ho(\He)$ if and only if $\m\cong\m'$ in $\Mod(\He)$.
    \item If $G$ does have root system of type $A_2\times A_1\times\cdots\times A_1$ (possibly empty product of $A_1$'s), then $\m\cong\m'$ in $\Ho(\He)$ if and only if either $\m\cong\m'$ in $\Mod(\He)$ or, up to swapping $\m$ and $\m'$, we have that $\chi_1^{\text{special}}$ and $\chi_2^{\text{special}}$ occur as submodules of $\Res^\He_\Hea(\m)$ and $\Res^\He_\Hea(\m')$ respectively and $\m\cong \m'$ as representations of $\widetilde{\Omega}$.
\end{enumerate}
\end{thmx}

In the case $G=\mathrm{GL}_n$, we note in particular that \Cref{third_main_thm} says that two simple supersingular $\He$-modules are isomorphic in $\Ho(\He)$ if and only if they are isomorphic in $\Mod(\He)$, except if $n=3$. As the $n$-dimensional simple supersingular $\He$-modules are in bijection with certain irreducible $n$-dimensional Galois representations by the work of Grosse-Kl\"onne \cite{GK16}, our results show that the modules appearing on the Hecke side of this bijection are generically still non-isomorphic in $\Ho(\He)$.

\subsection{Organisation of the paper} In \S2.1-2.3 we recall background material on pro-$p$ Iwahori-Hecke algebras, parabolic induction and simple $\He$-modules, and we prove that parabolic induction is fully faithful. In \S2.4-2.5 we recall the notion of an abelian model category, left/right transfers of model structures and the definition of Hovey's Gorenstein projective model structure, and we prove some general preparatory lemmas. 

In \S3.1-3.2, we study the properties of the functor of parabolic induction in the language of model categories, culminating with the proof of \Cref{first_main_thm}(i). We only use the fact that parabolic induction is fully faithful, has both a left and a right adjoint, and its left adjoint is exact. Consequently, \S3.1 is written in a higher level of generality (cf.\ \Cref{setup1}). In \S3.3, we study the relation between the recollement of \Cref{first_main_thm}(i) and supersingular modules and prove \Cref{first_main_thm}(ii).

In \S4.1, we investigate recollements in the homotopy category of an abelian model category $\A$. We show that if $\A$ is equipped with both a projective and injective model structure with the same class of trivial objects, then any recollement in $\Ho(\A)$ can be reconstructed from new model structures on $\A$. In the case of Hecke modules, we show that two of these new model structures can be right transferred to $\Rep_k^\infty(G)$. In \S4.2, we show that these new model structures are Bousfield localisations of the original model structures on $\A$.

Finally, in \S5 we compute the isomorphism classes of simple supersingular $\He$-modules. After recalling the necessary background on the affine Hecke algebra and supersingular characters in \S5.1, we study the image of the Hom space between two supersingular characters under the functor $\Ho(\Delta)$ in \S5.2, proving \Cref{second_main_thm}(i). In \S5.3, we study the action of $\Omega$ on the Hom spaces in $\Ho(\Hea)$, proving \Cref{second_main_thm}(ii). We finish in \S5.4 by analysing the case where $G$ is a product of $\mathrm{GL}_n$'s and prove \Cref{third_main_thm}.

\subsection{Acknowledgments} We would like to thank Jan Kohlhaase and Claudius Heyer for helpful conversations and for giving comments on earlier versions of this paper. The author was supported by the {\it Deutsche Forschungsgemeinschaft} (DFG, German Research Foundation) within the project {\it Smooth modular representation theory of $p$-adic reductive groups}, project number 435414187. The author is also a member of the DFG Research Training Group 2553 {\it `Symmetries and Classifying Spaces -- Analytic, Arithmetic and Derived'}. The financial support of the DFG is gratefully acknowledged. We thank the anonymous referees for providing detailed feedback and for making several suggestions which helped improve the paper.

\subsection{Notation and conventions} A class of objects of a category $\C$ will usually be identified with the corresponding full subcategory. If we denote an adjunction by $F:\C\rightleftarrows\D:U$ then $F$ is always assumed to be left adjoint to $U$. The unit (resp.\ the counit) of an adjunction will always be denoted by $\eta$ (resp.\ $\varepsilon$). When such an adjunction is a Quillen adjunction between model categories, we denote by $\mathbf{L}F:\Ho(\C)\rightleftarrows\Ho(\D):\mathbf{R}U$ the total derived adjunction between the homotopy categories. Given a functor $F:\C\to\D$, we denote by $\Im(F)$ its essential image. We also say that $F$ {\it preserves} (resp.\ {\it reflects}) a property (P) if $F*$ (resp.\ $*$) has property (P) whenever $*$ (resp.\ $F*$) does.

For any unital ring $R$ we denote by $\Mod(R)$ the category of right $R$-modules. Given $\m,\mathfrak{n}\in\Mod(R)$ we write $\Hom_R(\m,\mathfrak{n})$ for the set of $R$-linear maps from $\m$ to $\mathfrak{n}$, and we denote by $\pd_R(\m)$ and $\id_R(\m)$ the projective and injective dimensions of $\m$, respectively. For any topological group $J$ and field $k$, we denote by $\Rep_k^\infty(J)$ the category of $k$-linear smooth representations of $J$, i.e.\ the category of all $k$-vector spaces $V$ carrying a $k$-linear action of $J$ such that the stabiliser of any $v\in V$ is open in $J$. The representation $V=k$ equipped with the trivial $J$-action will be denoted by $\mathbf{1}\in \Rep_k^\infty(J)$.  Finally, for any $V\in \Rep_k^\infty(J)$ and any subgroup $J'\leq J$ we write $V^{J'}$ for the vector subspace of $J'$-invariants, i.e.\ of $v\in V$ such that $j\cdot v=v$ for all $j\in J'$.

The conventions for pro-$p$ Iwahori-Hecke algebras across the literature can differ depending on the authors. Our previous work \cite{KD21} as well as \cite{OS14} define $\He=\End_G(\ind_I^G(\mathbf{1}))^{\text{op}}$ and work with \emph{left} $\He$-modules, while the literature on parabolic induction for Hecke modules defines $\He=\End_G(\ind_I^G(\mathbf{1}))$ and works with \emph{right} $\He$-modules. Since this paper relies on many existing computations with parabolic induction, and in order to avoid confusions with the notions of positive/negative elements, we will stick with the latter convention and work with right $\He$-modules.

\section{Reminders on pro-$p$ Iwahori-Hecke algebras and model structures}

\subsection{The Bruhat-Tits building and Weyl groups} Let $p$ be a prime and let $\mathfrak{F}$ be a nonarchimedean local field of residue field $\mathbb{F}_q$ of characteristic $p$. Let $G=\mathbb{G}(\mathfrak{F})$ denote the group of $\mathfrak{F}$-rational points of a split connected reductive algebraic group defined over $\mathfrak{F}$. We fix a maximal split $\mathfrak{F}$-torus $\mathbb{T}$ of $\mathbb{G}$ and let $\mathbb{C}$ denote the connected component of the centre of $\mathbb{G}$. We denote by $T=\mathbb{T}(\mathfrak{F})$ and $Z=\mathbb{C}(\mathfrak{F})$ the corresponding groups of $\mathfrak{F}$-rational points. Let $r_{ss}$, resp.\ $r_Z$, denote the semisimple rank of $G$, resp.\ the rank of $Z$.

We let $\mathscr{X}$ denote the semisimple Bruhat-Tits building of $G$. In the standard apartment $\mathscr{A}:=X_*(\mathbb{T}/\mathbb{C})\otimes_{\Z}\mathbb{R}$ of $\mathscr{X}$, we fix a chamber $C$ and a hyperspecial vertex $x_0$ such that $x_0\in \overline{C}$. Given a face $F$ in $\mathscr{X}$, we let $P_{F}$ denote the parahoric subgroup of $G$ associated to $F$ and $P_F^\dagger$ the stabiliser in $G$ of the face $F$. The group $P_F$ has a pro-$p$ radical which we denote by $I_F$. In particular, we set $I=I_C$ and $I'=P_C$ which are called a pro-$p$ Iwahori subgroup and Iwahori subgroup of $G$, respectively.

The chamber $C$ determines a set $\Phi^+$ of positive roots of the root system $\Phi\subseteq X^*(\mathbb{T}/\mathbb{C})$ of $(\mathbb{G},\mathbb{T})$. We will view elements of $\Phi$ as characters of $T$. If $\Phi'$ is an irreducible component of $\Phi$, we write $\mathrm{rk}(\Phi')$ to denote its rank. We let $\Pi$ denote the basis of $\Phi$ defined by $\Phi^+$, and $B=T\ltimes U$ the Borel subgroup containing $T$ defined by $\Phi^+$, where $U$ is the unipotent radical of $B$. A \emph{standard parabolic subgroup} $P=M\ltimes N$ is any parabolic subgroup of $G$ containing $B$. The corresponding Levi subgroup $M$ will also be called \emph{standard}. Given a standard parabolic subgroup $P=MN$, we let $\Pi_M$ (resp.\ $\Phi_M$, resp.\ $\Phi_M^+$) be the corresponding simple roots (resp.\ root system, resp.\ positive roots) defined by $M$. We will sometimes also write $\Pi_P$ to denote this same set of simple roots.

Recall that the affine functions on $\mathscr{A}$ given by
$$
(\alpha, m):=\alpha(\bullet)+m\quad \text{for $\alpha\in\Phi$ and $m\in\Z$}
$$
are called the affine roots (note that we are implicitly using splitness of our group, see \cite[\S 4.2]{OS14} for the general definition). Identifying $\Phi$ with the affine roots of the form $(\alpha, 0)$, we will view $\Phi$ as a subset of the set $\Phi_\aff$ of affine roots. The chamber $C$ then also determines a set of positive affine roots $\Phi^+_\aff$ defined as those which take non-negative value on $C$. There is a partial order on $\Phi$ given by $\alpha\leq \beta$ if and only if $\beta-\alpha$ is a non-negative integral linear combination of elements of $\Pi$. We let $\Phi^{\min}\subset \Phi$ denote the set of minimal elements with respect to this partial order, and let $\Pi_\aff:=\Pi\cup\{(\alpha, 1): \alpha\in\Phi^{\min}\}$.

Now let $W_0=N_G(T)/T$ be the finite Weyl group of $G$, with length function $\ell: W_0\to\Z_{\geq 0}$ defined with respect to $\Pi$. Moreover we denote by
$$
W=N_G(T)/(T\cap I')\cong \Lambda\rtimes W_0
$$
the extended Weyl group of $(G, T)$, where $\Lambda:=T/(T\cap I')$. The action of $G$ on $\mathscr{X}$ restricts to an action of $W$ on $\mathscr{A}$ by affine automorphisms. We let
$$
\langle -,-\rangle:X^*(\mathbb{T}/\mathbb{C})\times X_*(\mathbb{T}/\mathbb{C})\to\Z
$$
denote the natural perfect pairing. The action of $T$ on the standard apartment is then given by translation via $\nu$, where $\nu:T\to X_*(\mathbb{T}/\mathbb{C})$ is the group homomorphism characterised by
$$
\langle \alpha,\nu(t)\rangle=-\val(\alpha(t))
$$
for all $\alpha\in\Phi$ and all $t\in T$. Here $\val: \mathfrak{F}^\times\to \Z$ denotes the normalised valuation.

We now set
$$
\widetilde{\Lambda}:=T/(T\cap I) \quad\text{and}\quad \widetilde{W}:=N_G(T)/(T\cap I).
$$
Note that the map $\nu$ descends to $\Lambda$ and $\widetilde{\Lambda}$. The quotient $(T\cap I')/(T\cap I)$ identifies with the group of $\mathbb{F}_q$-rational points of $\mathbb{T}$, which we denote by $T(\mathbb{F}_q)$. So we see that $\widetilde{\Lambda}$ is an extension of $\Lambda$ by $T(\mathbb{F}_q)$. For any standard Levi subgroup $M$, we let $W_{0,M}\subset W_0$ be the corresponding finite Weyl group. We then denote by $W_M$ the preimage of $W_{0, M}$ under the natural projection $W\twoheadrightarrow W_0$. Given a subset $X\subseteq W$ we write $\widetilde{X}$ for its pre-image under the projection $\widetilde{W}\twoheadrightarrow W$. Also, given any $w\in W_0$ we fix a lift $\hat{w}\in \widetilde{W}$. Note that the length function $\ell$ extends to both $W$ and $\widetilde{W}$.

By the identification $W\cong \Lambda\rtimes W_0$, we may (and indeed always will) view $W_0$ as a subgroup of $W$. Moreover, any element $\tilde{w}\in \widetilde{W}$ may be written as $\lambda \hat{\overline{w}}$, where $\lambda\in\widetilde{\Lambda}$ and $\overline{w}$ is the image of $\tilde{w}$ in $W_0$ under the projection. We also define
$$
W_0^M:=\{w\in W_0 \mid w(\Phi^+_M)\subset\Phi^+\}.
$$
By the Bruhat decomposition,
\begin{equation}\label{Bruhat}
G=\coprod_{\tilde{w}\in \widetilde{W}} I\tilde{w}I, \quad M=\coprod_{\tilde{w}\in \widetilde{W}_M} I_M\tilde{w}I_M
\end{equation}
where $I_M:=I\cap M$ is a pro-$p$ Iwahori subgroup of $M$.

Let $\Omega=\{w\in W \mid \ell(w)=0\}$. Then we have a decomposition $W=W_{\text{aff}}\rtimes\Omega$ where $W_{\text{aff}}$ denotes the affine Weyl group, generated by the set $S$ of simple affine reflections fixing the walls of $C$. The elements of $S$ are all of the form $s_\alpha$ for $\alpha\in \Pi_\aff$ (cf.\ \cite[\S 4.3]{OS14}). Given any affine root $\alpha=(\beta, m)$ with $\beta\in\Phi$, $m\in\Z$, we define $\alpha^\vee:=\beta^\vee$. Now, given an $\overline{\mathbb{F}_p}$-valued character $\xi$ of $T(\mathbb{F}_q)$, we let
\begin{equation}\label{S_xi}
S_\xi:=\{s_\alpha\in S\mid \xi(\alpha^\vee(x))=1 \text{ for all }x\in \mathbb{F}_q^\times\}.
\end{equation}
The group $\Omega$ is a finitely generated abelian group and we write $\Omega_{\text{tor}}$ for its (finite) torsion subgroup.

To each face $F\subseteq\overline{C}$ we associate the subset $S_F\subseteq S$ of those reflections fixing $F$ pointwise and we let $W_F$ denote the subgroup of $W$ generated by $S_F$. Note that $S_C=\emptyset$. If $\Phi=\sqcup_{i=1}^r \Phi_i$ is the decomposition of $\Phi$ into irreducible components and $\Pi=\sqcup_{i=1}^r\Pi_i$ is the corresponding decomposition of the basis, we also have an associated decomposition $S=\sqcup_{i=1}^r S_i$. Here $S_i=\{s_\beta\mid \beta\in\Pi_i\}\cup\{s_{(-\alpha_i,1)}\}$, where $\alpha_1, \ldots, \alpha_r$ denote the highest roots in each irreducible component. Then the map
\begin{align}\label{bijection_faces}
\{\text{faces $F$ with $F\subseteq \overline{C}$}\}&\longrightarrow \{S'\subseteq S\mid \text{$S'\cap S_i\neq S_i$ for every $1\leq i\leq r$}\}\\
F&\longmapsto S_F \notag
\end{align}
is a bijection (cf.\ \cite[\S 1.3.5]{BruTit72}). Moreover, this bijection is order-reversing in the sense that $F'\subseteq \overline{F}$ if and only if $S_F\subseteq S_{F'}$.

\subsection{Pro-$p$ Iwahori-Hecke algebras}

Let $k$ denote a field of characteristic $p$, and set
$$
\XX:=\ind_I^G(\mathbf{1})=\{f:I\backslash G\to k \mid f \text{ has finite support}\}\in\Rep_k^\infty(G)
$$
with $G$-action given by right translation. The representation $\XX$ is a left module over its endomorphism ring
$$
\He:=\End_G(\XX),
$$
the so-called pro-$p$ Iwahori-Hecke algebra of $G$ with respect to $I$ over $k$. By Frobenius reciprocity, $\He$ can be identified with the convolution $k$-algebra
$$
\He=\Hom_G(\XX, \XX)\cong \XX^I=\ind_I^G(\mathbf{1})^I=k[I\backslash G/I]
$$
of $I$-biinvariant compactly supported maps $G\to k$. This construction has analogues for various subgroups of $G$, which we briefly recall now.

For any face $F$ of $\mathscr{X}$, we set
$$
\XX_F=\ind_{I_{C(F)}}^{P_F}(\mathbf{1})\quad\text{and}\quad \XX_F^\dagger=\ind_{I_{C(F)}}^{P_F^\dagger}(\mathbf{1}),
$$
where $C(F)$ is the chamber defined in \cite[Lemma 1.3]{Koh}, and the corresponding Hecke algebras are then defined analogously as
$$
\He_F=\End_{P_F}(\XX_F)\quad\text{and}\quad \He_F^\dagger=\End_{P_F^\dagger}(\XX_F^\dagger).
$$
When $F\subseteq \overline{C}$ then $C(F)=C$, $I_{C(F)}=I$ and so extending functions by zero gives embeddings of $\XX_F$ and $\XX_F^\dagger$ into $\XX$ which are $P_F$-equivariant and $P_F^\dagger$-equivariant respectively. The algebras $\He_F$ and $\He_F^\dagger$ can then be naturally identified with subalgebras of $\He$ by viewing them as $I$-biinvariant functions on $P_F$ and $P_F^\dagger$ respectively and extending functions by zero. We then have that $\He$ is free over $\He_F^\dagger$, which itself is also free over $\He_F$ (cf. \cite[Lemma 4.20 \& Proposition 4.21]{OS14}).

We also let $G_\aff$ be the subgroup of $G$ generated by all the parahoric subgroups $P_F$ for $F$ varying among all the faces of $\mathscr{X}$. We can then also construct a $G_\aff$-representation $\XX_\aff=\ind_{I}^{G_\aff}(\mathbf{1})$ and an associated Hecke algebra $\Hea=\End_{G_\aff}(\XX_\aff)$. Via extension by zero, we can once again identify $\XX_\aff$ with a $G_\aff$-subrepresentation of $\XX$ and $\Hea$ with a $k$-subalgebra of $\He$. Viewing all the above Hecke algebras as subalgebras of $\He$, we note that $\Hea$ contains all the $\He_F$ for $F\subseteq \overline{C}$ and is in fact generated by them as an algebra (cf.\ the proof of \cite[Lemma 1.4]{Koh}). Furthermore, for any $F\subseteq \overline{C}$ we have that $\Hea$ is free as both a left and right $\He_F$-module.

For any standard Levi subgroup $M$, we also let $\XX_M:=\ind_{I_M}^M(\mathbf{1})$ and $\He_M:=\End_M(\XX_M)$. By Frobenius reciprocity, $\He_M$ can again be identified with the convolution $k$-algebra of $I_M$-biinvariant compactly supported maps $M\to k$. Note however that $\He_M$ is \emph{not} a subalgebra of $\He$ in general. By Frobenius reciprocity again, we have a functor $U_M:\Rep_k^\infty(M)\to \Mod(\He_M)$ defined by $U_M(V):=\Hom_M(\XX_M,V)\cong V^{I_M}$ for any $V\in\Rep_k^\infty(M)$.

For any $\tilde{w}\in \widetilde{W}_M$, we set $T^M_{\tilde{w}}:=\mathds{1}_{I_M\tilde{w}I_M}\in k[I_M\backslash M/I_M]=\He_M$, the characteristic function of the double coset $I_M\tilde{w}I_M$. When $M=G$, we drop the superscript and write $T_{\tilde{w}}$. By the Bruhat decomposition \eqref{Bruhat}, the set $\{T^M_{\tilde{w}}\mid \tilde{w}\in\widetilde{W}\}$ is a $k$-basis of $\He_M$. There are also certain dual elements $T^{M,*}_{\tilde{w}}$ which form another $k$-basis of $\He_M$, cf. \cite[\S 4.8, eq.\ (4.11)]{OS14}. In addition, the subalgebra $\Hea$ also has a $k$-basis given by the set $\{T_{\tilde{w}}\mid \tilde{w}\in\widetilde{W}_\aff\}$. The ring structure of $\He$ is determined by the braid relation
$$
T_{\tilde{w}\tilde{w}'}=T_{\tilde{w}}T_{\tilde{w}'}\quad\text{if}\quad \ell(\tilde{w}\tilde{w}')=\ell(\tilde{w})+\ell(\tilde{w}'),
$$
and a certain quadratic relation which we now describe. Fix the set $\{\hat{s}\mid s\in S\}\subseteq \widetilde{W}_\aff$ of lifts of the simple affine reflections defined as in \cite[\S 4.8]{OS14}. Then the quadratic relations for $\He$ are given by
\begin{equation}\label{quadratic_rel}
T_{\hat{s}}^2=T_{\hat{s}}\cdot\left(|\mu_{\alpha}|\sum_{t\in \alpha^\vee(\mathbb{F}_q^\times)} T_t\right)
\end{equation}
for all $s=s_\alpha\in S$, where $\mu_{\alpha}\cong \ker(\alpha^\vee|_{\mathbb{F}_q^\times})$ has order one or two, cf.\ \cite[\S 2.2, eqs.\ (27)-(28), \& \S 2.1.6, eq.\ (13) \& Remark 2.7]{OS19}.

Recall from \cite[\S 4.1]{Abe19} that for $w\in W_{0,M}$ and $\lambda\in\widetilde{\Lambda}$, the element $\lambda \hat{w}\in \widetilde{W}_M$ is called \emph{$M$-positive} (resp.\ \emph{$M$-negative}) if $\langle \alpha, \nu(\lambda)\rangle\leq 0$ (resp.\ $\langle \alpha, \nu(\lambda)\rangle\geq 0$) for every $\alpha\in\Phi^+\setminus \Phi_M^+$. The $k$-span of the $T^M_{\tilde{w}}$ with $M$-positive (resp.\ $M$-negative) $\tilde{w}\in\widetilde{W}$ forms a $k$-subalgebra of $\He_M$ which we denote by $\He_M^+$ (resp.\ $\He_M^-$). By Lemma 4.6 of \textit{loc.\ cit.} (see also \cite[Theorem 1.4(iii)]{Vign5}) we have injective algebra homomorphisms
\begin{equation}\label{embed_plusminus}
j_M^+: \He_M^+\hookrightarrow\He\quad\text{and}\quad j_M^-: \He_M^-\hookrightarrow\He
\end{equation}
given respectively by $j_M^+(T^M_{\tilde{w}})=T_{\tilde{w}}$ and $j_M^-(T^{M,*}_{\tilde{w}})=T^*_{\tilde{w}}$ for all $M$-positive (resp.\ $M$-negative) $\tilde{w}\in\widetilde{W}_M$. Now fix $\lambda_0^+, \lambda_0^-\in\widetilde{\Lambda}$ such that they are in the centre of $\widetilde{W}_M$ and $\langle \alpha, \nu(\lambda_0^+)\rangle< 0$ (resp.\ $\langle \alpha, \nu(\lambda_0^-)\rangle> 0$) for every $\alpha\in\Phi^+\setminus \Phi_M^+$.  We will need the following:

\begin{prop}[{\cite[Theorem 1.4(ii)]{Vign5}}]\label{lambda_pm}
For any $\lambda_0^+, \lambda_0^-\in\widetilde{\Lambda}$ chosen as above,  the elements $T^M_{\lambda_0^\pm}\in\He_M^\pm$ are central and non-zero divisors,  and $\He_M$ is the Ore localisation of $\He_M^\pm$ at $T^M_{\lambda_0^\pm}$.
\end{prop}

We finally need to introduce twists by the longest element $w_0\in W_0$ of the finite Weyl group. Given a standard parabolic $P$ with Levi subgroup $M$, let $w_0^M=w_0w_{0,M}$ where $w_{0,M}$ denotes the longest element in $W_{0,M}$. Then we let $w_0(M)$ denote the standard Levi subgroup determined by $W_{0, w_0(M)}=w_0W_{0,M}w_0$, and we let $w_0(P)$ be the standard parabolic whose Levi subgroup is $w_0(M)$. Lastly, we fix a lift $\hat{w}_0^M\in \widetilde{W}$ of $w_0^M$.

\begin{prop}[{\cite[Proposition 2.20]{Vign5}}]\label{twist} The $k$-linear map
$$
\He_M\to \He_{w_0(M)},\quad\quad T^M_{\tilde{w}}\mapsto T^{w_0(M)}_{\hat{w}_0^M\tilde{w}(\hat{w}_0^M)^{-1}} \quad \text{for } \tilde{w}\in \widetilde{W}_M,
$$
is an isomorphism of $k$-algebras. This induces an exact equivalence of categories
$$
\iota_0^M:\Mod(\He_M)\to \Mod(\He_{w_0(M)})
$$
such that $\iota_0^M(\He_M)\cong \He_{w_0(M)}$.
\end{prop}

The above satisfies a transitivity property as follows. If $P\subseteq P'$ are standard parabolics with Levi subgroups $M\subseteq M'$, define $w_{M'}^M:=w_{0,M'}w_{0,M}$. Then as above we have an equivalence of categories $\iota_{M'}^M:\Mod(\He_M)\to \Mod(\He_{w_{0,M'}(M)})$ induced by the isomorphism $T^M_{\tilde{w}}\mapsto T^{w_{0,M'}(M)}_{\hat{w}_{M'}^M\tilde{w}(\hat{w}_{M'}^M)^{-1}}$. Furthermore, the map $T^{w_{0,M'}(M)}_{\tilde{w}}\mapsto T^{w_0(M)}_{\hat{w}_0^{M'}\tilde{w}(\hat{w}_0^{M'})^{-1}}$ also defines an isomorphism $\He_{w_{0,M'}(M)}\cong \He_{w_0(M)}$ which induces an equivalence $\iota_{0,M'}^M:\Mod(\He_{w_{0,M'}(M)})\to \Mod(\He_{w_0(M)})$. Then the equality
\begin{equation}\label{transitivity_twists}
\iota_0^M=\iota_{0,M'}^M \circ\iota_{M'}^M
\end{equation}
holds, cf.\ \cite[Proposition 2.21]{Vign5}.

\subsection{Parabolic induction and simple modules}

We recall the construction of the parabolic induction and coinduction functors from \cite[\S 3]{Vign5}. We fix a standard Levi subgroup $M$ of $G$. By \eqref{embed_plusminus} we may view the algebra $\He$ as an $(\He_M^+,\He_M^-)$-bimodule via the maps $j_M^+$ and $j_M^-$ respectively. Thus we may define two functors $\Mod(\He_M)\to\Mod(\He)$ by setting
$$
\Ind_{\He_M}^{\He}(\m):=\m \otimes_{\He_M^+}\He\quad\text{and}\quad\Coind_{\He_M}^{\He}(\m):=\Hom_{\He_M^-}(\He,\m)
$$
for any $\m\in\Mod(\He_M)$. The functors $\Ind_{\He_M}^{\He}$ and $\Coind_{\He_M}^{\He}$ are called \emph{parabolic induction} and \emph{parabolic coinduction}, respectively.

By the tensor-hom adjunction, we see that $\Coind_{\He_M}^{\He}$ has left adjoint $\Loc_M:=(-)\otimes_{\He_M^-}\He_M$ and $\Ind_{\He_M}^{\He}$ has right adjoint $\RM:=\Hom_{\He_M^+}(\He_M,-)$. Similarly to the above, we are using $j_M^-$ to view $\He$-modules as $\He_M^-$-modules via restriction in the definition of $\Loc_M$, and similarly $j_M^+$ is used in the definition of $\RM$. Note as well that since $\He_M$ is the Ore localisation of $\He_M^-$ at $T^M_{\lambda_0^-}=T^{M,*}_{\lambda_0^-}$, the functor $\Loc_M$ is the corresponding localisation functor and hence is exact. The counit of the adjoint pair $(\Loc_M,\Coind_{\He_M}^{\He})$ is given as follows. For $\m\in \Mod(\He_M)$, we have a natural map $\Coind_{\He_M}^{\He}(\m)=\Hom_{\He_M^-}(\He,\m)\to \m$ by evaluating at $1$. Localising at $T^M_{\lambda_0^-}$ this gives the counit
\begin{align}\label{counit}
\varepsilon_{\m}:\Coind_{\He_M}^{\He}(\m)\otimes_{\He_M^-}\He_M &\to\m\\
\varphi \otimes (T^M_{\lambda_0^-})^{-n} &\mapsto \varphi(1)\cdot(T^M_{\lambda_0^-})^{-n} \nonumber
\end{align}
The main properties of these functors that we'll need are summarised below:

\begin{thm}[{\cite[1.6-1.10]{Vign5}}]\label{ind_coind_thm} Let $M$ be a standard Levi subgroup of $G$. Then:
\begin{enumerate}
\item there are natural isomorphisms
$$
\Ind_{\He_M}^{\He}\cong \Coind_{\He_{w_0(M)}}^{\He}\circ \iota_0^M \quad \text{and}\quad \Coind_{\He_M}^{\He}\cong \Ind_{\He_{w_0(M)}}^{\He}\circ \iota_0^M,
$$
where $\iota_0^M$ is the equivalence from \Cref{twist}; and
\item parabolic induction and coinduction are transitive and have both a left and a right adjoint. In particular they are exact functors.
\end{enumerate}
\end{thm}

Note that the left (resp.\ right) adjoint of $\Ind_{\He_M}^{\He}$ (resp.\ $\Coind_{\He_M}^{\He}$) can be described explicitly, namely it is equal to the composite $\mathcal{L}_M^G:=\iota_0^{w_0(M)}\circ \Loc_{w_0(M)}$ (resp.\ $\iota_0^{w_0(M)}\circ \mathcal{R}^G_{w_0(M)}$), cf.\ \cite[Theorem 1.9]{Vign5}.

We will use the next general categorical fact several times, so we record it here:

\begin{lem}[{\cite[Proposition 1.5.6]{KaSch}}]\label{ff_counit} Suppose $F:\C\rightleftarrows\D:U$ is an adjunction, and denote by $\eta$ and $\varepsilon$ the unit and counit respectively. Then $F$ (resp.\ $U$) is fully faithful if and only if $\eta$ (resp.\ $\varepsilon$) is an isomorphism.
\end{lem}

The next result is, to the best of our knowledge, not explicitly available anywhere in the literature. Its proof was however inspired by the arguments used in the proof of \cite[Lemma 8.1]{Koz}. We note that it also answers the question asked in the last paragraph of \cite[\S 1]{Vign5}, in the case where the coefficient ring is a field of characteristic $p$.

\begin{prop}\label{fully_faithful} For any standard parabolic subgroup $M$ of $G$, the functors $\Ind_{\He_M}^{\He}$ and $\Coind_{\He_M}^{\He}$ are fully faithful.
\end{prop}

\begin{proof}
It suffices to show the result for coinduction since the two functors differ by a twist. By \Cref{ff_counit} we must show that the counit $\varepsilon_{\m}$ is an isomorphism for all $\m\in\Mod(\He_M)$. By \cite[Lemma 4.10]{Abe19} (see also \cite[Lemma 3.9]{Vign5}), the map $\varphi\mapsto (\varphi(T_{\hat{v}}))_{v\in W_0^M}$ gives a $k$-linear isomorphism $\Coind_{\He_M}^{\He}(\m)\cong \bigoplus_{v\in W_0^M} \m$, where $\hat{v}$ is a fixed lift of $v$ to $\widetilde{W}$ (and by convention we impose for $v=1$ that $\hat{v}=1_{\widetilde{W}}$). Under this identification and from \eqref{counit}, we see that $\varepsilon_{\m}$ is the localisation of the projection onto the $v=1$ summand.

Now, we need to consider the action of $T^*_{\lambda_0^-}=j_M^-(T^M_{\lambda_0^-})$. By \cite[Propositions 2.7(ii) \& 4.12]{Abe19} (whose statements assume $k$ to be algebraically closed but whose proof works for arbitrary fields of characteristic $p$), the action of $T^*_{\lambda_0^-}$ on the summand corresponding to $v\in W_0^M$ is zero if $\hat{v}^{-1}(\lambda_0^-)$ is not $M$-negative, and the action on the $v=1$ component agrees with the action of $T^M_{\lambda_0^-}$ on $\m$. Hence, from the definition of $\varepsilon_{\m}$, we'll be done if we show that $\hat{v}^{-1}(\lambda_0^-)$ is not $M$-negative for $v\neq 1$.

We argue by induction on $\ell(v)$. Pick $\alpha\in \Pi$ such that $\ell(s_\alpha v)<\ell(v)$. First assume that $\alpha\in \Phi^+\setminus\Phi_M^+$, noting that this holds in the base case $\ell(v)=1$ by definition of $W_0^M$. Then, using the fact that $v\in W_0^M$ and that $v(-v^{-1}(\alpha))=-\alpha\in-\Phi^+$, we obtain that $-v^{-1}(\alpha)\in \Phi^+\setminus\Phi_M^+$. Then, we get
$$
\langle -v^{-1}(\alpha), \nu(\hat{v}^{-1}(\lambda_0^-))\rangle=-\langle \alpha, \nu(\lambda_0^-) \rangle<0
$$
by definition of $\lambda_0^-$, and so $\hat{v}^{-1}(\lambda_0^-)$ is not $M$-negative as required.

Assuming instead that $\alpha\in\Phi_M^+$, we then have that $(\hat{s_\alpha} \hat{v})^{-1}(\lambda_0^-)=\hat{v}^{-1}(\lambda_0^-)$ since $\lambda_0^-$ is central in $\widetilde{W}_M$. But since $\ell(s_\alpha v)<\ell(v)$ and $v\in W_0^M$ we also have
$$
\{\beta\in\Phi^+\mid s_\alpha v(\beta)\in-\Phi^+\}\subseteq \{\beta\in\Phi^+\mid v(\beta)\in-\Phi^+\}\subseteq \Phi^+\setminus\Phi_M^+,
$$
and thus $s_\alpha v\in W_0^M$. We are then done by the inductive hypothesis.
\end{proof}

We now recall the classification of simple Hecke modules in terms of supersingulars. The classification was initially stated in terms of parabolic coinduction by Abe \cite[Theorem 4.22]{Abe19} for algebraically closed fields, and was later shown to hold for general $k$ in \cite[Theorem I.8]{HeVi19}. This classification was also reworded in terms of parabolic induction (cf.\ \cite[Corollary 4.30]{AHV18}). We first need to recall the definition of a supersingular $\He$-module as follows. Recall from \cite[\S 5]{Vign1} that for a choice of spherical orientation $o$, i.e. of a closed Weyl chamber of $W_0$, we can associate a basis $\{E_o(\tilde{w})\}_{\tilde{w}\in\widetilde{W}}$ of $\He$ with the property that $E_o(\tilde{w})\in T_{\tilde{w}}+\sum_{\tilde{v}<\tilde{w}}kT_{\tilde{v}}$. Here $\tilde{v}<\tilde{w}$ refers to the Bruhat order on $W_\aff$, suitably extended to $W=W_\aff\rtimes\Omega$ and then inflated to $\widetilde{W}$, cf.\ the discussion preceeding Corollary 5.26 in \textit{loc.\ cit.} Now let $\mathcal{O}$ be a $\widetilde{W}$-conjugacy class  in $\tilde{\Lambda}$, and let
$$
z_\mathcal{O}=\sum_{\tilde{\lambda}\in\mathcal{O}} E_o(\tilde{\lambda}).
$$
Then this element does not depend on the choice of orientation $o$, and defines an element of the centre of $\He$, cf.\ \cite[Theorem 5.1]{Vign3}. The length of $\tilde{\lambda}\in\mathcal{O}$ is independent of the choice of representative $\tilde{\lambda}$, and we denote it by $\ell(\mathcal{O})$.

\begin{defn} An $\He$-module $\m$ is called \emph{supersingular} if for each $\widetilde{W}$-conjugacy class $\mathcal{O}$ in $\widetilde{\Lambda}$ with $\ell(\mathcal{O})>0$ and for every $x\in \m$, there exists some $n\in\mathbb{N}$ such that $xz_\mathcal{O}^n=0$. If the integer $n$ can be chosen independently of $x$ and $\mathcal{O}$, we call $\m$ \emph{uniformly supersingular}.
\end{defn}

Note that the notions of supersingular and uniformly supersingular coincide for $\He$-modules of finite length, or more generally for finitely generated modules, and hence in particular for simple modules.

Next, a \emph{standard simple supersingular triple} for $\He$ is a triple $(P, \mathfrak{n}, Q)$ such that
\begin{itemize}
\item $P$ is a standard parabolic subgroup of $G$ with Levi decomposition $P=MN$;
\item $\mathfrak{n}$ is a simple supersingular module over $\He_M$; and
\item $Q$ is a parabolic subgroup of $G$ with $P\subseteq Q\subseteq P(\mathfrak{n})$, where $P(\mathfrak{n})$ is the standard parabolic subgroup of $G$ corresponding to $\Pi(\mathfrak{n}):=\Pi_M\cup \Pi_{\mathfrak{n}}\subseteq \Pi$ and $\alpha\in \Pi_{\mathfrak{n}}$ if and only if both $\langle \Pi_M, \alpha^\vee\rangle=0$ and $T_\lambda^M$ acts trivially on $\mathfrak{n}$ for all $\lambda\in T'_\alpha$.
\end{itemize}
In the above, for $\alpha\in\Pi$, we set $T'_\alpha=T\cap G'_\alpha$ where $G'_\alpha$ denotes the subgroup of $G$ generated by the root subgroups corresponding to $\pm \alpha$. Two supersingular triples $(P, \mathfrak{n}, Q)$ and $(P', \mathfrak{n}', Q')$ are called equivalent if $P=P'$, $Q=Q'$ and $\mathfrak{n}\cong \mathfrak{n}'$.

Given a supersingular triple $(P, \mathfrak{n}, Q)$, we define $I_\He(P, \mathfrak{n}, Q), CI_\He(P, \mathfrak{n}, Q)\in \Mod(\He)$ as follows. Suppose that $Q$ has Levi decomposition $Q=LV$. Now assume that $Q'$ is any parabolic subgroup of $G$ with $P\subseteq Q'\subseteq P(\mathfrak{n})$ and say it has Levi decomposition $Q'=L'V'$. Then $\mathfrak{n}$ can be extended uniquely to an $\He_{L'}$-module $e_{\He_{L'}}(\mathfrak{n})$ with certain properties, and moreover whenever $P\subseteq Q\subseteq Q'\subseteq P(\mathfrak{n})$ there is a natural injective $\He$-linear map $\Ind_{\He_{L'}}^\He(e_{\He_{L'}}(\mathfrak{n}))\to \Ind_{\He_L}^\He(e_{\He_L}(\mathfrak{n}))$ (cf.\ \cite[Corollary 3.9 \& Proposition 4.5]{AHV18}). There is also an analogous map for coinduction, cf.\ the discussion preceding \cite[Theorem 4.22]{Abe19}. We may thus define the modules
\begin{equation}\label{defn_simple1}
I_\He(P, \mathfrak{n}, Q):=\coker\left( \bigoplus_{Q\subsetneq Q'\subseteq P(\mathfrak{n})} \Ind_{\He_{L'}}^\He(e_{\He_{L'}}(\mathfrak{n}))\to\Ind_{\He_L}^\He(e_{\He_L}(\mathfrak{n}))\right)
\end{equation}
and
\begin{equation}\label{defn_simple2}
CI_\He(P, \mathfrak{n}, Q):=\coker\left( \bigoplus_{Q\subsetneq Q'\subseteq P(\mathfrak{n})} \Coind_{\He_{L'}}^\He(e_{\He_{L'}}(\mathfrak{n}))\to\Coind_{\He_L}^\He(e_{\He_L}(\mathfrak{n}))\right).
\end{equation}
The classification of simple $\He$-modules states that $(P, \mathfrak{n}, Q)\mapsto I_\He(P, \mathfrak{n}, Q)$ and $(P, \mathfrak{n}, Q)\mapsto CI_\He(P, \mathfrak{n}, Q)$ define two bijections between the set of equivalence classes of standard simple supersingular triples and the set of isomorphism classes of simple $\He$-modules, cf.\ \textit{loc.\ cit.\ }and \cite[Corollary 4.30]{AHV18}. Moreover, the two bijections are related by an $\He$-equivariant isomorphism
\begin{equation}\label{compatibility_simples}
CI_\He(P, \mathfrak{n}, Q)\cong I_\He(w_0(P), \iota_0^M(\mathfrak{n}), w_0(Q))
\end{equation}
cf.\ \cite[Corollary 4.24]{AHV18}. Finally, a simple module $CI_\He(P, \mathfrak{n}, Q)$ is supersingular if and only if $P=Q=G$ (cf.\ \cite[Corollary 4.31]{Abe19}), and analogously for $I_\He(P, \mathfrak{n}, Q)$.

\subsection{Recollections on abelian model categories} 

We briefly recall some facts about model structures on abelian categories. For details on the general theory of model categories, we refer the reader to \cite[\S 1]{KD21} and the references therein. Fix a bicomplete abelian category $\A$. We will say that a model structure on $\A$ is an \emph{abelian model structure} if the class of (trivial) fibrations, resp.\ (trivial) cofibrations, coincides with the class of epimorphisms with (trivially) fibrant kernel, resp.\ monomorphisms with (trivially) cofibrant cokernel. An abelian category equipped with such a model structure will be called an \emph{abelian model category}.

By work of Hovey, cf.\ \cite{Hov2}, constructing abelian model structures on $\A$ is equivalent to constructing certain cotorsion pairs. Given a pair $(\C, \Fc)$ of classes of objects of $\A$, consider the pair $(^\perp\Fc, \C^\perp)$ defined by
$$
^\perp\Fc=\{X\in\A \mid \Ext^1_\A(X,Y)=0 \text{ for all }Y\in\Fc\}
$$
and
$$
\C^\perp=\{Y\in\A \mid \Ext^1_\A(X,Y)=0 \text{ for all }X\in\C\}.
$$
Then we say that $(\C, \Fc)$ is a \emph{cotorsion pair} if $\C={^\perp}\Fc$ and $\Fc=\C^\perp$ (cf.\ \cite[Definition 2.3]{Hov2}). It is said to have \emph{enough functorial projectives} if for all $X\in\A$, there is a functorial epimorphism $f:Y\to X$ with $Y\in\C$ such that $\ker(f)\in\Fc$. Dually, it is said to have \emph{enough functorial injectives} if for all $X\in\A$, there is a functorial monomorphism $g:X\to Z$ with $Z\in\Fc$ such that $\coker(g)\in\C$. When a cotorsion pair has both enough functorial projectives and enough functorial injectives, then it is said to be \emph{functorially complete}.

Recall that a class of objects of $\A$ is called \emph{thick} if it is closed under direct summands and satisfies the 2-out-of-3 property with respect to short exact sequences. When $\A$ is equipped with a model structure, we will say that an object $X\in\A$ is \emph{trivial} if the zero map $0\to X$ (or equivalently $X\to 0$) is a weak equivalence. The fundamental result we need is the following.

\begin{thm}[{\cite[Theorem 2.2]{Hov2}}]\label{Hovey} Let $\A$ be a bicomplete abelian category.
\begin{enumerate}
\item If $\A$ is an abelian model category, denote by $\C$ (resp.\ $\Fc$, resp.\ $\W$) the class of cofibrant (resp.\ fibrant, resp.\ trivial) objects. Then $\W$ is thick, and $(\C\cap \W,\Fc)$ and $(\C, \Fc\cap \W)$ are functorially complete cotorsion pairs.
\item Conversely, assume that $(\C, \Fc, \W)$ is a triple of classes of objects of $\A$ such that $\W$ is thick, and $(\C\cap \W,\Fc)$ and $(\C, \Fc\cap \W)$ are functorially complete cotorsion pairs. Then there is a unique abelian model structure on $\A$ such that $\C$, $\Fc$ and $\W$ coincide with the classes of cofibrant, fibrant and trivial objects, respectively.
\end{enumerate}
\end{thm}

In fact, Hovey proved a more general version of the above theorem which applies to a wider class of model structures on $\A$, using the notion of a proper class of short exact sequences, and this has been further generalised to the context of exact categories by Gillespie in \cite{Gil11}. In this paper, we will however only need to work with abelian model structures. We will sometimes denote an abelian model structure on $\A$ by the triple $(\C, \Fc, \W)$ as in \Cref{Hovey}. Note that the homotopy category $\Ho(\A)$ of an abelian model category is canonically equivalent to the quotient $\C\cap \Fc/\Gamma$, where $\Gamma=\C\cap\Fc\cap \W$ (cf.\ \cite[Proposition 4.4(5)]{Gil11}).

Given an abelian model structure $(\C, \Fc, \W)$ on $\A$, the completeness of the cotorsion pairs $(\C\cap \W,\Fc)$ and $(\C, \Fc\cap \W)$ ensures that for any $X\in\A$ there are short exact sequences $0\to Y\to Q_c X\to X\to 0$ and $0\to X\to Q_f X\to Z\to 0$, with $Q_c X\in\C$, $Q_f X\in\Fc$, $Y\in\Fc\cap \W$ and $Z\in \C\cap\W$. This can be done functorially in our setup, but by slight abuse of terminology we will refer to \emph{any} two elements $Q_c X$ and $Q_f X$ fitting in short exact sequences as above as choices of \emph{cofibrant replacement} and \emph{fibrant replacement} for $X$, respectively. We will always denote such replacements by $Q_cX$ and $Q_fX$ as we did above.

As a final piece of notation, if a functor $\F:\A\to \B$ between the underlying categories of two abelian model categories preserves all weak equivalences, it then naturally induces a functor $\Ho(\F):\Ho(\A)\to\Ho(\B)$ given by $\Ho(\F)(X):=\F(X)$ for $X\in\Ho(\A)$. When $\F$ is Quillen we can identify $\Ho(\F)$ with the total derived functor of $\F$ under mild assumptions as follows.

\begin{lem}\label{preserve_reflect_we} Suppose that $\A$ and $\B$ are abelian model categories, and let $\F:\A\to\B$ be an exact functor between the underlying abelian categories.
\begin{enumerate}
\item If $\F$ preserves trivial objects then it preserves all weak equivalences.
\item Assume that $\F$ is a (left or right) Quillen functor. If $\F$ reflects trivial objects then it reflects all weak equivalences.
\end{enumerate}
Furthermore, with the assumption of (i) and if $\F$ is Quillen, then the total derived functor of $\F$ is naturally isomorphic to $\Ho(\F)$. 
\end{lem}

\begin{proof}
Suppose that $f$ is a morphism in $\A$. For (i), assume first that $f$ is either a trivial cofibration or trivial fibration. Then in particular it is either a monomorphism with trivial cokernel, or an epimorphism with trivial kernel. Under our assumptions then $\F f$ is still a monomorphism/epimorphism with trivial cokernel/kernel. It then follows from \cite[Lemma 5.8]{Hov2} that $\F f$ is a weak equivalence. The case of a general $f$ follows immediately by factorising $f$ as the composite of a trivial cofibration and a trivial fibration.

For (ii), we assume that $\F$ is left Quillen (the other case is dual). Factorise $f=gh$ where $h$ is a trivial cofibration and $g$ is a fibration. Then by the 2-out-of-3 property of weak equivalences it suffices to show that if $\F g$ is a weak equivalence then so is $g$. Since $\F$ is exact, we may apply \textit{loc.\ cit.} again to see that $\F g$ is a weak equivalence precisely when $\F(\ker(g))$ is trivial. It then follows that $g$ is a trivial fibration since $\F$ reflects trivial objects.

Finally, the total derived functor of $\F$ is defined to be $\Ho(\F\circ Q)$ where $Q$ is some fibrant or cofibrant replacement functor, depending on whether $\F$ is right or left Quillen. But by definition, any object is weakly equivalent to its replacement and $\F$ preserves such weak equivalences under the assumption of (i). Thus $\Ho(\F)$ is isomorphic the total derived functor in that case, as required.
\end{proof}

An abelian model category $\A$ is called \emph{hereditary} if the class $\C$ of cofibrant objects is closed under taking kernels of epimorphisms and if the class $\Fc$ of fibrant objects is closed under taking cokernels of monomorphisms. Such an abelian model category is automatically stable (cf.\ \cite[Corollary 1.1.15]{Beck14}), meaning that its associated homotopy category $\Ho(\A)$ is canonically equipped with the structure of a triangulated category. We quickly recall what this triangulated structure is.

\begin{itemize}
\item The shift functor is given by the suspension functor $\Sigma$ associated with the model structure. For a cofibrant object $X\in\C$, we may choose from the factorisations in the model structure a cofibration $j_X:X\to W$ where $W$ is trivially fibrant, and in fact $W\in \Gamma= \C\cap\Fc\cap \W$ since cofibrations are closed under compositions. Then we let $\Sigma(X)$ be defined as $\coker(j_X)$. Since $\A$ is hereditary, note that $\Sigma(X)\in\C\cap\Fc$ whenever $X\in\C\cap\Fc$. This construction doesn't depend on the choice of $j_X$ up to weak equivalence, and hence this defines an endofunctor of $\Ho(\A)\simeq \C\cap\Fc/\Gamma$ which can be shown to be an auto-equivalence. There is also a dually defined loop functor which is naturally isomorphic to $\Sigma^{-1}$. We note that the loop and suspension functors can be defined more generally for any (not necessarily stable) pointed model category, cf.\ \cite[Definition 3.1.2]{BR}, and are then adjoint functors.
\item The distinguished triangles are obtained from the so-called cofiber sequences as in \cite[Definition 3.4.1 \& Theorem 4.2.1]{BR}. In the context of hereditary abelian model categories, these are given as follows. Suppose that $0\to X\stackrel{f}{\to} Y\stackrel{g}{\to} Z\to 0$ is a short exact sequence in $\A$ with $X, Y, Z\in\C$. Fix $j_X:X\to W$ as above, so that $W\in\Fc\cap \W$ and $\Sigma X\cong \coker(j_X)$. Since $\Ext^1_\A(Z, W)=0$ we have that $j_X$ extends along $f$ to a map $Y\to W$. Postcomposing with the projection $W\to \Sigma X$, we obtain a map $l:Y\to \Sigma X$ such that $l\circ f=0$. Thus $l$ descends to a map $h:Z\to \Sigma X$. The distinguished triangles in $\Ho(\A)$ are then, up to isomorphism, the triangles of the form
\begin{equation}\label{d_triangles}
X\stackrel{f}{\to} Y\stackrel{g}{\to} Z \stackrel{h}{\to} \Sigma X
\end{equation}
as above.
\end{itemize}

\begin{rem}\label{fiber_remark}
There is a dual notion of fiber sequences which may be used to construct the distinguished triangles as well, cf.\ \cite[Definition 3.4.2, Corollaries 4.4.1 \& 4.4.2]{BR}.  In our context, those are obtained similarly to the above by using short exact sequences of fibrant objects and the loop functor. That these two constructions give the same triangulated structure on the homotopy category holds in the general context of stable model categories (cf.\ \textit{loc.\ cit.}) but for hereditary abelian model categories this can be also seen directly by applying \cite[Lemma 1.4.4]{Beck14} and its dual. 
\end{rem}

Recall that a ring $R$ is called {\it Gorenstein} if it is both left and right noetherian and if it has finite selfinjective dimension both as a left and as a right $R$-module. In this case the left and right selfinjective dimensions of $R$ coincide (cf.\ \cite{EJ11}, Proposition 9.1.8). If $n$ denotes their common value then $R$ is called an {\it $n$-Gorenstein ring}. We will frequently make use of the fact that if $R$ is an $n$-Gorenstein ring then an $R$-module $\m$ has finite projective dimension if and only if it has finite injective dimension. In this case both dimensions are bounded above by $n$ (cf.\ \cite{EJ11}, Theorem 9.1.10). The relevance of this notion for us comes from the following result of Ollivier-Schneider.

\begin{thm}[{\cite[Theorem 0.1]{OS14}}] The pro-$p$ Iwahori-Hecke algebra $\He$ is $n$-Gorenstein where $n\leq r_{\text{ss}}+r_Z$.
\end{thm}

The bound on the selfinjective dimension of $\He$ is actually sharp (cf.\ \cite[Remark 8.2]{Koz}). The theorem applies generally to the algebras $\He_M$ for any standard Levi subgroup $M$ of $G$. 

We now recall the construction, due to Hovey, of two abelian model structures on the category $\Mod(R)$ for any Gorenstein ring $R$. Let $\A$ be an abelian category with enough projectives. An object $A\in\A$ is called {\it Gorenstein projective} if there is an acyclic complex
\[
Y = \cdots\longrightarrow Y_1\stackrel{d_1}{\longrightarrow}Y_0\stackrel{d_0}{\longrightarrow} Y_{-1}\longrightarrow\cdots
\]
of projective objects of $\A$ for which $A=Z_0Y=\Ker(d_0)$ and which remains exact upon applying $\Hom_{\A}(\cdot,P)$ for any projective object $P$ of $\A$. When $\A=\Mod(R)$ for a Gorenstein ring $R$, then a module $\m$ is Gorenstein projective as an object of $\Mod(R)$ if and only if $\Ext^i_R(\m,\mathfrak{p})=0$ for all $i\geq 1$ and all modules $\mathfrak{p}$ of finite projective dimension (cf.\ \cite[Corollary 11.5.3]{EJ11}). Then $\m$ is Gorenstein projective if and only if there is an acyclic complex $Y$ of projectives with $\m=Z_0Y$, cf.\ \cite[Remark 3.1]{KD21}. There is also dually a notion of \emph{Gorenstein injective} object of $\A$, and for $\A=\Mod(R)$ these have dual properties to the above (cf.\ \cite[Definition 10.1.1 \& Corollary 11.2.2]{EJ11})). We denote the classes of Gorenstein projective and Gorenstein injective modules over a ring $R$ by $\mathrm{GProj}(R)$ and $\mathrm{GInj}(R)$, respectively.

The following fundamental theorem is due to Hovey.

\begin{thm}[{\cite[Theorem 8.6]{Hov2}}]\label{GP_R_Mod}
Let $R$ be a Gorenstein ring and let $\W$ denote the class of $R$-modules which have finite projective/injective dimension. Then there is a cofibrantly generated abelian model structure on $\Mod(R)$ given by the triple $(\mathrm{GProj}(R), \Mod(R), \W)$. There is dually a cofibrantly generated abelian model structure on $\Mod(R)$ given by the triple $(\Mod(R), \mathrm{GInj}(R), \W)$.
\end{thm}

We call these the {\it Gorenstein projective} and {\it Gorenstein injective} model structures on $\Mod(R)$, respectively. We will write $\Mod(R)^{\text{GP}}$, resp.\ $\Mod(R)^{\text{GI}}$, for the category $\Mod(R)$ viewed as an abelian model category with the Gorenstein projective, resp.\ Gorenstein injective, model structure.

\begin{rem}\label{we_same} The identity adjunction
$$
\id:\Mod(R)^{\text{GP}}\rightleftarrows \Mod(R)^{\text{GI}}:\id
$$
is Quillen and in fact a Quillen equivalence (cf.\ \cite[bottom of p.583]{Hov2}). The latter means that if $\m, \mathfrak{n}\in\Mod(R)$ are such that $\m$ is Gorenstein projective and $\mathfrak{n}$ is Gorenstein injective, then a map between them is a weak equivalence in one model structure if and only if it is a weak equivalence in the other. But this holds because, more generally, the weak equivalences in both model structures are the same. Indeed, that follows from \Cref{preserve_reflect_we} as the trivial objects are the same in either model structure.

Even without using the notion of Quillen equivalences, we therefore see that the two model structures have \emph{the same} homotopy category, namely the abstract localisation $\Ho(R):=\Mod(R)[(\text{w.e.})^{-1}]$ as defined in \cite[Definition 1.2.1]{HovBook} (where `w.e.' stands for `weak equivalences'). Our two model structures provide two categories, namely $\mathrm{GProj}(R)/\mathrm{Proj}(R)$ and $\mathrm{GInj}(R)/\mathrm{Inj}(R)$, which are equivalent to it and in particular which describe the $\Hom$ sets explicitly. Similarly, the two model structures give two different collections of representatives for the isomorphism classes of the same distinguished triangles in $\Ho(R)$, and similarly give two different (explicit) constructions for the same shift functor $\Sigma$ (see also \Cref{triangulated_remark} later). In this paper we will use both of these model structures, depending on the situation.
\end{rem}

We record here a particular application of \Cref{preserve_reflect_we} which will be useful in order to compare the Gorenstein projective model structures over various Gorenstein rings.

\begin{lem}\label{free_adjunction}
Suppose that $A$ and $B$ are Gorenstein rings, with $A$ isomorphic to a subring of $B$. Further, assume that $B$ is free as a left $A$-module. Write $F^B_A:\Mod(A)\to\Mod(B)$ for the functor $\m\mapsto \m\otimes_A B$ and $\Res^B_A$ for its right adjoint given by restriction of scalars. Then the adjunction $F^B_A\dashv \Res^B_A$ is Quillen with respect to the Gorenstein projective model structure and there are natural isomorphisms $\mathbf{L}F^B_A\cong \Ho(F^B_A)$ and $\mathbf{R}\Res^B_A\cong \Ho(\Res^B_A)$.
\end{lem}

\begin{proof}
Under our assumptions, both $F^B_A$ and $\Res^B_A$ are exact. We claim that they preserve trivial objects, i.e.\ modules of finite projective/injective dimension. Indeed, $F^B_A$ and $\Res^B_A$ preserve projective and injective modules, respectively, since they are an adjoint pair and exact. It therefore follows by exactness that $F^B_A$, resp.\ $\Res^B_A$, preserves finite projective resolutions, resp.\ finite injective resolutions, and thus we get our claim. From all of this we also obtain that $\Res^B_A$ preserves epimorphisms, i.e.\ fibrations, and epimorphisms with kernels of finite projective dimension, i.e.\ trivial fibrations, and hence is right Quillen as required. Finally, the isomorphism on the total derived functors is immediate from \Cref{preserve_reflect_we}.
\end{proof}

The following characterisation of the morphisms in the homotopy category of any abelian model category is due to Gillespie.

\begin{prop}[{\cite[Proposition 4.4]{Gil11}}]\label{factor_through}
Let $\A$ be an abelian model category with model structure $(\C, \Fc, \W)$. Suppose $X\in \C$ and $Y\in \Fc$. Then the abelian group $\Hom_{\Ho(\A)}(X,Y)$ is isomorphic to the quotient of $\Hom_{\A}(X,Y)$ by the equivalence relation given by $f\sim g$ if and only if $f-g$ factors through an object of $\Gamma=\C\cap\Fc\cap\W$, if and only if $f-g$ factors through an object of $\C\cap\W$,  if and only if $f-g$ factors through an object of $\Fc\cap\W$.
\end{prop}

For a Gorenstein ring $R$ and writing $\W$ for the class of $R$-modules of finite projective dimension, it is shown in \cite[Corollary 8.5]{Hov2} that the classes $\mathrm{GProj}(R)\cap\W$ and $\mathrm{GInj}(R)\cap\W$ are that of projective and injective $R$-modules respectively. We thus see from \Cref{factor_through} that for $\m$ a Gorenstein projective module and $\mathfrak{n}$ any module, two maps $f,g:\m\to \mathfrak{n}$ are equal in $\Ho(R)$ if and only if $f-g$ factors through a projective module. Similarly, if $\m$ is any module and if $\mathfrak{n}$ is Gorenstein injective, then two maps $f,g:\m\to \mathfrak{n}$ are equal in $\Ho(R)$ if and only if $f-g$ factors through an injective module.

To make our notation simpler, for any two $\m, \mathfrak{n}\in\Mod(R)$ we will denote the $k$-vector space $\Hom_{\Ho(R)}(\m, \mathfrak{n})$ by $[\m, \mathfrak{n}]_R$. More generally, if $R$ and $S$ are two Gorenstein rings with $R$ isomorphic to a subring of $S$ and if $\m, \mathfrak{n}\in\Mod(S)$, we will write $[\m, \mathfrak{n}]_R$ for short instead of $[\Res^S_R(\m), \Res^S_R(\mathfrak{n})]_R$.

\begin{rem}\label{factor_projective}
Suppose that $R$ is a Gorenstein ring and $f:\m\to \mathfrak{n}$ is a map in $\Mod(R)$. Fix a surjection $\theta:\mathfrak{p}\to\mathfrak{n}$ where $\mathfrak{p}$ is a projective module. Assume that $f$ factors as a composite $\m\to\mathfrak{q}\stackrel{g}{\to}\mathfrak{n}$ where $\mathfrak{q}$ is a projective module. Since $\mathfrak{q}$ is projective, it follows that $g$, and hence $f$, factor through $\theta$. Consequently, we see that the maps $f\in \Hom_R(\m, \mathfrak{n})$ which factor through a projective module are precisely those in the image of $\theta\circ(-):\Hom_R(\m, \mathfrak{p})\to\Hom_R(\m, \mathfrak{n})$. Dually, if $\iota:\m\to\mathfrak{e}$ is a monomorphism into an injective module $\mathfrak{e}$ then the maps $f\in \Hom_R(\m, \mathfrak{n})$ which factor through a injective module are precisely those in the image of $(-)\circ\iota:\Hom_R(\mathfrak{e}, \mathfrak{n})\to\Hom_R(\m, \mathfrak{n})$.
\end{rem}

By combining \Cref{factor_through} and \Cref{factor_projective}, we see that for $\m, \mathfrak{n}\in\Mod(R)$ with $\m$ Gorenstein projective and for $\theta:\mathfrak{p}\to\mathfrak{n}$ as above, $[\m, \mathfrak{n}]_R\cong \coker (\theta\circ(-))\cong \Ext^1_S(\m, \ker(\theta))$. This generalises as follows. We say that an abelian model structure $(\C, \Fc, \W)$ on an abelian category $\A$ is \emph{projective}, resp. \emph{injective}, if $\Fc=\A$,  resp.\ $\C=\A$. We note that any projective or injective abelian model structure on $\A$ is hereditary, cf.\ \cite[Corollary 1.1.12]{Beck14}.

\begin{prop}\label{Hom_Ho} Let $\A$ be an abelian model category with model structure $(\C, \Fc, \W)$, and assume that $\A$ is either injective or projective. Let $X\in \C$ and $Y\in\Fc$. Then we have isomorphisms
$$
\Hom_{\Ho(\A)}(\Sigma^{-k}X,Y)\cong \Hom_{\Ho(\A)}(X,\Sigma^kY)\cong \Ext^k_\A(X, Y)
$$
of abelian groups for all $k>0$.
\end{prop}

\begin{proof}
The arguments given in \cite[Theorem 9.3]{Hov2} for Gorenstein projective/injective model structures carry over mutatis mutandis to this more general situation.
\end{proof}

Since the Gorenstein projective, resp.\ Gorenstein injective, model structure on $\Mod(R)$ is projective, resp.\ injective, it is hereditary. In particular, it is stable which is also proved directly in \cite[Theorem 9.3]{Hov2}. Hence we see that the homotopy category $\Ho(R)$ has the structure of a triangulated category as described earlier. We note that since every object is cofibrant in the Gorenstein injective model structure the triangles \eqref{d_triangles} are obtained for \emph{any} short exact sequence in $\Mod(R)$.

\subsection{Left and right transfers} We now recall the notion of left/right transfer of a model structure along an adjunction. If $F:\mathcal{E}\rightleftarrows\D:U$ is an adjunction where $\mathcal{E}$ is a model category and $\D$ is a locally small, bicomplete category, there is then a natural candidate for a model structure on $\D$ making the adjunction Quillen. Namely, one defines fibrations (resp.\ weak equivalences) in $\D$ to be the morphisms $f$ such that $U(f)$ is a fibration (resp.\ weak equivalence). The cofibrations are then defined as those morphisms which satisfy the left lifting property with respect to the trivial fibrations. This, however, may not in general give a well-defined model structure on $\D$. When it does, we call it the \emph{right transfer} of the model structure on $\mathcal{E}$ along $U$, and we say that the right transfer exists. In \cite[Proposition 4.10]{KD21}, it was shown that the right transfer of Hovey's Gorenstein projective model structure on $\Mod(\He)$ along the functor $U$ of pro-$p$ Iwahori invariants exists. This gives a model structure on $\Rep_k^\infty(G)$, called the $I$-Gorenstein projective model structure.

Dually, when $\D$ is a model category and $\mathcal{E}$ is a locally small, bicomplete category, the \emph{left transfer} of the model structure on $\D$ along $F$, if it exists, is the model structure on $\mathcal{E}$ whose cofibrations (resp.\ weak equivalences) are the morphisms $g$ such that $F(g)$ is a cofibration (resp.\ weak equivalence), and the fibrations are defined as those morphisms which satisfy the right lifting property with respect to the trivial cofibrations.

\section{Parabolic induction, recollements and supersingular modules}

\subsection{Parabolic induction and recollement: one parabolic}

We now wish to study how the parabolic (co)induction functor behaves with respect to Hovey's Gorenstein projective and Gorenstein injective model structures on $\Mod(\He_M)$ and $\Mod(\He)$. Since we only know that the Gorenstein projective model structure right transfers to the category $\Rep_k^\infty(G)$, we a priori favour that model structure over the injective one. But as we will see, when it comes to parabolic induction it is more fruitful to consider the Gorenstein injective model structure as well. Since working specifically with Hecke algebras does not simplify the approach, we work in this subsection in the more general situation described below.

\begin{set}\label{setup1}
We let $R, S$ be two Gorenstein rings. Moreover, we assume the following:
\begin{itemize}
\item there is an adjoint triple of additive functors
\[
\xymatrix@C=0.5cm{
\Mod(S) \ar[rrr]^{\G} &&& \Mod(R) \ar @/_1.5pc/[lll]_{\Ll}  \ar
 @/^1.5pc/[lll]^{\R}
 },
\]
meaning that $\Ll$ (resp.\ $\R$) is the left (resp.\ right) adjoint of $\G$;
\item the functor $\G$ is fully faithful, and $\Ll$ is exact.
\end{itemize}
\end{set}

\begin{ex} The above setup applies to $S=\He_M$ ($M$ some standard Levi subgroup), $R=\He$ and $\G\in\{\Coind_{\He_M}^{\He}, \Ind_{\He_M}^{\He}\}$ along with its two adjoints. The left adjoint is exact in either case because $\Loc_M$ is a localisation functor and the equivalence in \Cref{twist} is also exact. We've also seen that parabolic (co)induction is fully faithful in \Cref{fully_faithful}.
\end{ex}

Note that the functor $\G$ is automatically exact as it is both a left and a right adjoint. But $\R$ is only known to be left exact, being a right adjoint. Part (i) of the Lemma below is a generalisation of \cite[Lemma 8.1]{Koz} to our setup:

\begin{lem}\label{proj_dim} Let $\m\in\Mod(S)$ and $\mathfrak{n}\in\Mod(R)$. Then:
\begin{enumerate}
\item $\pd_{S}(\m)<\infty$ if and only if $\pd_{R}(\G(\m))<\infty$;
\item if $\pd_{R}(\mathfrak{n})<\infty$ then $\pd_{S}(\Ll(\mathfrak{n}))<\infty$; and
\item $\m$ is Gorenstein injective if and only if $\G(\m)$ is Gorenstein injective.
\end{enumerate}
\end{lem}

\begin{proof}
(i)-(ii) Since $\Ll$ and $\G$ are exact adjoint functors, we have an isomorphism of abelian groups
\begin{equation}\label{Ext_Ind1}
\Ext^i_{S}(\Ll(\mathfrak{n}),\m)\cong\Ext^i_{R}(\mathfrak{n},\G(\m))
\end{equation}
for all $\m\in\Mod(S)$, all $\mathfrak{n}\in\Mod(R)$ and all $i\geq 0$. This gives that $\id_R(\G(\m))\leq \id_S(\m)$ and also immediately implies (ii). Furthermore, given any $\m'\in\Mod(S)$, \Cref{ff_counit} and \eqref{Ext_Ind1} together give isomorphisms
\begin{equation}\label{Ext_Ind2}
\Ext^i_{S}(\m,\m')\cong\Ext^i_{S}(\Ll(\G(\m)),\m')\cong \Ext^i_{R}(\G(\m),\G(\m')).
\end{equation}
We deduce immediately that $\pd_{S}(\m)\leq \pd_{R}(\G(\m))$. Since $R$ and $S$ are assumed to be Gorenstein, finite injective dimension coincides with finite projective dimension for modules over both $R$ and $S$, and (i) follows from the above two inequalities.

(iii) Recall that $\m$ is Gorenstein injective if and only if $\Ext^i_{S}(\mathfrak{e},\m)=0$ for all $\mathfrak{e}\in\Mod(S)$ of finite injective dimension (and similarly for $\G(\m)$ and $R$-modules of finite injective dimension). Therefore, (ii) and \eqref{Ext_Ind1} give that if $\m$ is Gorenstein injective then so is $\G(\m)$. Conversely, (i) and \eqref{Ext_Ind2} give that if $\G(\m)$ is Gorenstein injective then so is $\m$.
\end{proof}

We next show that the Gorenstein projective/injective model structures on $\Mod(S)$ are in fact determined by those on $\Mod(R)$. Note that since the notion of weak equivalences doesn't depend on which model structure we choose to work with (cf.\ \Cref{we_same}), there is no ambiguity in the statement of part (iii) of the Proposition below.

\begin{prop}\label{ind_tran} With $S$, $R$, $\G$ and $\Ll$ as in \Cref{setup1}, we have:
\begin{enumerate}
\item the Gorenstein projective model structure on $\Mod(S)$ is the right transfer of the Gorenstein projective model structure on $\Mod(R)$ along $\G$. In particular, $\G$ is a right Quillen functor for the Gorenstein projective model structures;
\item the Gorenstein injective model structure on $\Mod(S)$ is both the left transfer and the right transfer of the Gorenstein injective model structure on $\Mod(R)$ along $\G$. In particular, $\G$ is both a left and a right Quillen functor for the Gorenstein injective model structures; and
\item the functor $\Ll$ preserves weak equivalences.
\end{enumerate}
\end{prop}

\begin{proof}
We fix a morphism $f$ in $\Mod(S)$. Our first goal is to establish the following:
\begin{itemize}
\item[(A)] $f$ is a fibration if and only if $\G(f)$ is a fibration, in the Gorenstein projective model structures on $\Mod(S)$ and $\Mod(R)$; and
\item[(B)] $f$ is a fibration, resp.\ cofibration, if and only if $\G(f)$ is a fibration, resp.\ cofibration, in the Gorenstein injective model structures on $\Mod(S)$ and $\Mod(R)$.
\end{itemize}
To get (A), we simply observe that fibrations are epimorphisms and the statement holds because $\G$ is exact and faithful. To get (B), the case of cofibrations is proved completely dually to (A): the cofibrations are the monomorphisms, so that the claim holds by exactness and faithfulness of $\G$. The fibrations on the other hand are precisely the epimorphisms with Gorenstein injective kernels. Now the claim for that case holds by exactness and faithfulness of $\G$ and by \Cref{proj_dim}(iii).

To prove (i)-(ii), it now suffices to show that $f$ is a weak equivalence if and only if $\G(f)$ is a weak equivalence. First note that $\G$ preserves trivial objects, i.e.\  modules of finite projective dimension, by \Cref{proj_dim}(i). Since $\G$ is exact, it now follows from \Cref{preserve_reflect_we}(i) that $\G$ preserves weak equivalences. In particular, by (A)-(B) above we have that $\G$ is right Quillen for both the Gorenstein projective and injective model structures, and left Quillen for the Gorenstein injective model structure. Hence, we may apply \Cref{preserve_reflect_we}(ii) which, together with \Cref{proj_dim}(i) again, gives us that $\G$ reflects weak equivalences as required.

Finally, (iii) follows from \Cref{preserve_reflect_we}(i) and \Cref{proj_dim}(ii), given that (i) above implies that $\Ll$ is left Quillen for the Gorenstein projective model structure.
\end{proof}

\begin{cor}\label{derived_nice}
Continuing with $S$, $R$, $\G$ and $\Ll$ as in \Cref{setup1}, we have:
\begin{enumerate}
\item the two total right derived functors of $\G$ (with respect to the Gorenstein projective and Gorenstein injective model structures respectively) and the total left derived functor of $\G$ (with respect to the Gorenstein injective model structure) are all naturally isomorphic to $\Ho(\G)$. Similarly, the two total left derived functors of $\Ll$ are naturally isomorphic to $\Ho(\Ll)$;
\item the functor $\Ho(\G)$ is fully faithful; and
\item $\Ll$ reflects trivial objects, i.e.\ the converse of \Cref{proj_dim}(ii) holds, if and only if $(\Ll,\G)$ is a Quillen equivalence (for either/both model structures).
\end{enumerate}
\end{cor}

\begin{proof}
Part (i) is immediate from \Cref{preserve_reflect_we}. For the rest of this proof, we work with the Gorenstein injective model structure. This suffices to show (iii) because the derived functors of $\G$ (resp.\ $\Ll$) for both model structures can be identified with $\Ho(\G)$ (resp.\ $\Ho(\Ll)$), and moreover an adjoint pair is a Quillen equivalence precisely when the derived functors form an adjoint equivalence of categories (cf.\ \cite[Proposition 1.3.13]{HovBook}).

(ii) By \Cref{ff_counit} we need to show that the counit of the adjunction $(\Ho(\Ll), \Ho(\G))$ is an isomorphism. By the proof of \cite[Proposition 1.3.13]{HovBook}, this holds if and only if, for any $\m\in\Mod(S)$, the composite
$$
\Ll Q_c\G\m\stackrel{\Ll q_c}{\longrightarrow}\Ll\G\m\stackrel{\varepsilon_{\m}}{\longrightarrow}\m
$$
is a weak equivalence. Here $Q_c\G\m$ denotes a cofibrant replacement of $\G\m$ in $\Mod(R)$, which comes equipped with a trivial fibration $q_c:Q_c\G\m\to\G\m$. But by \Cref{ind_tran}(iii) we have that $\Ll q_c$ is a weak equivalence and $\varepsilon_{\m}$ is an isomorphism (cf.\ \Cref{ff_counit} again), so that our claim holds.

(iii) By the above $\varepsilon_{\m}\Ll q_c$ is always a weak equivalence. It therefore follows from \cite[Corollary 1.3.16]{HovBook} that $(\Ll,\G)$ is a Quillen equivalence if and only if $\Ll$ reflects weak equivalences between cofibrant objects. But in the Gorenstein injective model structure all objects are cofibrant, so the latter condition says that $\Ll$ reflects all weak equivalences. The result now follows from \Cref{preserve_reflect_we}(ii).
\end{proof}

Note that from \Cref{ind_tran}(iii) and exactness we get that $\Ll$ preserves both fibrations and trivial fibrations in the Gorenstein projective model structure. This is somewhat unusual for a \emph{left} Quillen functor. Also, we will see later that the Quillen equivalence as in \Cref{derived_nice}(iii) essentially never holds in the case of parabolic (co)induction from a standard Levi subgroup $M$ unless $M=G$, cf.\ \Cref{ind_non_ess}. For now we only illustrate that this is clear when $M$ is the maximal split torus.

\begin{ex}\label{torus_finite_global_dim} If $M=T$ is the maximal split torus, then $\He_T=k[\widetilde{\Lambda}]$ is the group algebra of $\widetilde{\Lambda}=T/(T\cap I)\cong \Z^{r_{ss}+r_Z}\times T(\mathbb{F}_q)$. Hence $\He_T$ is free of finite rank, coprime with $p$, over the Laurent polynomial ring $k[x_1^{\pm 1},\ldots, x_{r_{ss}+r_Z}^{\pm 1}]$. The latter ring has finite global dimension, and therefore so does $\He_T$ (cf.\ \cite[Theorem 7.5.6(iii)]{MCR01}). In particular, the homotopy category $\Ho(\Mod(\He_T))$ is zero. On the other hand, $\He$ typically has infinite global dimension. Indeed, when $G$ does not have root system of type $A_1\times\cdots\times A_1$ or when $r_{ss}>0$ and $q >2$, then $\He$ has infinite global dimension and in fact has a simple module $\m$ of infinite projective dimension (cf.\ \cite[7.1-7.3]{OS14}). So, under these hypotheses, $\Ho(\He)\neq 0$ and $\Ind_{\He_T}^{\He}$ is not a Quillen equivalence. And indeed, for this module $\m$ one automatically has $\pd_{\He_T}(\Loc_M(\m))<\infty$ even though $\pd_{\He}(\m)=\infty$.
\end{ex}

Suppose that $\m\in\Mod(R)$ and $\mathfrak{n} \in\Mod(S)$. We have an isomorphism $\Hom_{R}(\G(\mathfrak{n}), \m)\cong \Hom_{S}(\mathfrak{n}, \R(\m))$. Since $\G$ is exact we have that $\R$ preserves injectives, and hence we obtain a first quadrant Grothendieck spectral sequence
\begin{equation}\label{spec_seq}
E_2^{ij}=\Ext_{S}^i(\mathfrak{n}, R^j\R(\m))\Rightarrow \Ext_{R}^{i+j}(\G(\mathfrak{n}),\m).
\end{equation}
Note that \Cref{ind_tran}(ii) implies that $\R$ preserves fibrations in the Gorenstein injective model structures, i.e.\ surjections with Gorenstein injective kernels. Given that the functor $\R$ is not necessarily right exact, this may seem surprising but it is actually fairly straightforward to show this directly using the above spectral sequence.

\begin{lem}\label{right_adjoint_cohomology} Suppose that $R$ is an $n$-Gorenstein ring. Then:
\begin{enumerate}
\item the functor $\R$ has cohomological dimension bounded above by $n$;
\item the Gorenstein injective modules are all acyclic for $\R$; 
\item an $R$-module $\m$ is $\R$-acyclic if and only if there is an isomorphism of abelian groups
$$
\Ext_R^j(\G(\mathfrak{n}), \m)\cong \Ext^j_S(\mathfrak{n}, \R(\m))
$$
for all $j\geq 0$ and all $\mathfrak{n}\in\Mod(S)$; and
\item $\R$ preserves Gorenstein injectives.
\end{enumerate}
\end{lem}

\begin{proof}
Let $\m\in\Mod(R)$. By plugging $\mathfrak{n}=S$ in the spectral sequence \eqref{spec_seq}, the $E_2^{ij}$ term becomes zero for $i>0$, and we obtain an isomorphism
\begin{equation}\label{formula_derad}
\Ext_{R}^j(\G(S),\m)\cong R^j\R(\m)
\end{equation}
for all $j\geq 0$. But by \Cref{proj_dim}(i), $\G(S)$ has finite projective dimension and hence $\pd_{R}(\G(S))\leq n$ by assumption on $R$. Thus the left hand side of \eqref{formula_derad} vanishes for $i>n$, proving (i).

For (ii), assume now that $\m$ is Gorenstein injective. Then for any module of finite injective dimension $\mathfrak{e}\in \Mod(R)$ and any $j>0$, we have $\Ext_{R}^j(\mathfrak{e}, \m)=0$. Applying this to $\mathfrak{e}=\G(S)$, we immediately get from \eqref{formula_derad} that $R^j\R(\m)=0$ for all $j>0$ as claimed.

For (iii), the spectral sequence \eqref{spec_seq} immediately implies that the isomorphism of Ext groups holds when $\m$ is is $\R$-acyclic. Conversely, assume that $\Ext_R^j(\G(\mathfrak{n}), \m)\cong \Ext^j_S(\mathfrak{n}, \R(\m))$ for all $\mathfrak{n}\in\Mod(S)$ and all $j\geq 0$. Then combining this isomorphism with \eqref{formula_derad} we obtain
$$
R^j\R(\m)\cong \Ext^j_S(S, \R(\m))=0
$$
for all $j>0$, as required.

Finally, for (iv) we can apply (ii), (iii), and \Cref{proj_dim}(i) again to see that if $\m\in\Mod(R)$ is Gorenstein injective and if $\mathfrak{e}\in\Mod(S)$ has finite injective dimension then $\Ext_{S}^j(\mathfrak{e},\R(\m))\cong \Ext_R^j(\G(\mathfrak{e}), \m)=0$ for $j>0$, and thus $\R(\m)$ is Gorenstein injective.
\end{proof}

Next, we investigate properties of $\Mod(S)$ (resp.\ $\Ho(S)$) viewed as a full subcategory of $\Mod(R)$ (resp.\ $\Ho(R)$) via the functor $\G$ (resp.\ $\Ho(\G))$. We first note that $\Mod(S)$ is not in general closed under subquotients in $\Mod(R)$, and so in particular is not a Serre subcategory, as seen in the following example.

\begin{ex}
When $\G=\Ind_{\He_M}^{\He}$ or $\Coind_{\He_M}^{\He}$ for some proper Levi subgroup $M\neq G$, the essential image of $\G$ is not  closed under subobjects. Indeed, first let $V=\Ind_P^G(\mathbf{1}_M)\in\Rep_k^\infty(G)$. By standard properties of induction for smooth representations, the trivial representation $\mathbf{1}_G$ embeds as a subrepresentation of $V$. By left exactness of the functor $U$ and applying \cite[Theorem 1]{OllVign18}, we see that the trivial character $\mathbf{1}_{\He}$ of $\He$ embeds as a submodule of $\Ind_{\He_M}^{\He}(\mathbf{1}_{\He_M})$. However, there does not exist $\m\in\Mod(\He_M)$ such that $\Ind_{\He_M}^{\He}(\m)\cong\mathbf{1}_{\He}$. This is because, by \cite[Lemmas 3.5 \& 3.6]{Vign5}, we have an isomorphism $\Ind_{\He_M}^{\He}(\m)\cong \m^{\oplus n}$ of $k$-vector spaces, where $n=\abs{W_0}/\abs{W_{0,M}}>1$. So in particular $\dim_k \Ind_{\He_M}^{\He}(\m)>1$ whenever $\m\neq 0$.
\end{ex}


Hence we cannot in some cases of interest form an abelian quotient of $\Mod(R)$ by $\Mod(S)$ via $\G$. We now show how the situation is improved by working at the homotopy level, by making use of the triangulated structures of $\Ho(S)$ and $\Ho(R)$.

Let $\T$ be a triangulated category with shift functor $\Sigma:\T\to\T$. Recall that a subcategory $\Ss$ of $\T$ is called a \emph{triangulated subcategory} if $\Sigma$ restricts to an autoequivalence of $\Ss$, if $\Ss$ is strictly full (meaning any object of $\T$ isomorphic to an object of $\Ss$ lies in $\Ss$) and if whenever $X\to Y\to Z\to\Sigma X$ is a distinguished triangle in $\T$ with two of $X, Y, Z$ lying in $\Ss$, then so does the third one. A subcategory $\Ss$ is called \emph{thick} if it is triangulated and closed under direct summands. If $\T'$ is another triangulated category with shift functor $\Sigma'$, an additive functor $\F:\T\to \T'$ is called \emph{triangulated} if $\F\Sigma\cong \Sigma'\F$ and $\F$ sends distinguished triangles to distinguished triangles. Note that the essential image of a full triangulated functor is a triangulated subcategory. Moreover, if $\Ss$ is a triangulated subcategory then it inherits from $\T$ a triangulated structure and the inclusion $\Ss\to \T$ is a triangulated functor.

Suppose now that $\T$ contains all small coproducts. For example, the homotopy category of any Gorenstein ring has all small coproducts since Gorenstein projective modules are closed under arbitrary direct sums. Then recall that a triangulated subcategory $\Ss$ is called \emph{localising} if it is closed under arbitrary small coproducts. All such subcategories are automatically thick, cf.\ \cite[Proposition 1.6.8]{Neeman}.

In general there is a Verdier quotient $\T/\Ss$, for $\T$ a triangulated category and $\Ss$ a triangulated subcategory, which is itself a triangulated category equipped with a triangulated quotient functor $\pi:\T\to\T/\Ss$, and this construction satisfies a suitable universal property (cf.\ \cite[\S II.2.1-II.2.3]{Verdier}).

We will now see how our setup lifts to the homotopy level. For this we need the notion of a recollement of triangulated categories, which we now recall.

\begin{defn}\label{recollement_defn} A \emph{recollement} of triangulated categories is a diagram of triangulated functors
\[
\xymatrix@C=0.5cm{
\T' \ar[rrr]^{\mathsf{G}} &&& \T \ar[rrr]^{\mathsf{I}}  \ar @/_1.5pc/[lll]_{\mathsf{L}}  \ar
 @/^1.5pc/[lll]^{\mathsf{R}} &&& \T''
\ar @/_1.5pc/[lll]_{\mathsf{Q}} \ar
 @/^1.5pc/[lll]^{\mathsf{P}}
 }
\]
such that the following conditions are satisfied:
\begin{itemize}
\item $\mathsf{Q}$ (resp.\ $\mathsf{P}$) is the left (resp.\ right) adjoint of $\mathsf{I}$;
\item $\mathsf{L}$ (resp.\ $\mathsf{R}$) is the left (resp.\ right) adjoint of $\mathsf{G}$;
\item the functors $\mathsf{G}$, $\mathsf{Q}$ and $\mathsf{P}$ are fully faithful; and
\item the essential image of $\mathsf{G}$ is equal to the kernel $\Ker(\mathsf{I})$.
\end{itemize}
\end{defn}

We briefly describe some properties implied by this definition (cf.\ \cite[\S 4.9-4.13]{Kra10}). First, $\T'$ identifies with a thick subcategory of $\T$ via $\G$ such that $\T''\simeq \T/\T'$. Moreover, the adjunction morphisms give for any $X\in\T$ two functorial distinguished triangles
\begin{equation}\label{triangle_right_adjoints}
\mathsf{G}\mathsf{R}(X)\to X\to \mathsf{P}\mathsf{I}(X)\to\Sigma(\mathsf{G}\mathsf{R}(X))
\end{equation}
and
$$
\mathsf{Q}\mathsf{I}(X)\to X\to \mathsf{G}\mathsf{L}(X)\to\Sigma(\mathsf{Q}\mathsf{I}(X)).
$$
In particular, we see that the smallest triangulated subcategory of $\T$ containing $\T'$ and $\T''$ is $\T$ itself, where $\T''$ is identified as a subcategory of $\T$ via either one of the fully faithful functors $\mathsf{Q}$ and $\mathsf{P}$.

For any category $\C$ and any full subcategory $\D\subseteq \C$, we define two full subcategories $\D^\ddagger$ and $^\ddagger \D$ (as the symbol $\perp$ was already taken) of $\C$ by
$$
\D^\ddagger :=\{Y\in\C\mid \Hom_\C(X,Y)=0 \text{ for all } X\in\D\}
$$
and
$$
^\ddagger \D :=\{X\in\C\mid \Hom_\C(X,Y)=0 \text{ for all } Y\in\D\}.
$$
Then, in the above setting of a recollement, the category $\T''$ is equivalent to both $(\T')^\ddagger$ and $^\ddagger (\T')$ via the functors $\mathsf{P}$ and $\mathsf{Q}$ respectively, where we view $\T'$ as a full subcategory of $\T$ via the functor $\mathsf{G}$. Note that by the adjunctions $\Ll\dashv \G\dashv \R$, we have $(\T')^\ddagger=\Ker(\R)$ and $^\ddagger (\T')=\Ker(\Ll)$. In order to establish the existence of a recollement, we have the following:

\begin{prop}[{\cite[Propositions 4.9.1 \& 4.13.1]{Kra10}}]\label{Krause_criterion} Let $\T$ be a triangulated category and $\Ss$ a thick subcategory of $\T$.
\begin{enumerate}
\item The following are equivalent:
\begin{enumerate}
\item the inclusion $\Ss\xrightarrow{\subseteq}\T$ has a right adjoint;
\item the quotient functor $\T\to\T/\Ss$ has a a right adjoint; and
\item the inclusion $\Ss^\ddagger\xrightarrow{\subseteq}\T$ has a left adjoint and ${^\ddagger}(\Ss^\ddagger)=\Ss$.
\end{enumerate}
\item If the equivalent conditions of (i) and their dual hold, then the inclusion $\Ss\xrightarrow{\subseteq}\T$ is part of a recollement
\[
\xymatrix@C=0.5cm{
\Ss \ar[rrr]^{\subseteq} &&& \T \ar[rrr]  \ar @/_1.5pc/[lll]  \ar
 @/^1.5pc/[lll] &&& \T/\Ss
\ar @/_1.5pc/[lll] \ar
 @/^1.5pc/[lll]
 }
\]
of triangulated categories.
\end{enumerate}
\end{prop}

We now go back to \Cref{setup1}. We remind the reader that $\Ho(R)$ and $\Ho(S)$ as well as their triangulated structure do not depend on which model structure (Gorenstein projective or Gorenstein injective) we work with, cf.\ \Cref{we_same}. Applying the above, we now get:

\begin{prop}\label{recollement} The functors $\G$, $\Ll$ and $\R$ induce a recollement
\[
\xymatrix@C=0.5cm{
\Ho(S) \ar[rrr]^{\Ho(\G)} &&& \Ho(R) \ar[rrr]^{\mathsf{\pi}} \ar @/_1.5pc/[lll]_{\Ho(\Ll)}  \ar
 @/^1.5pc/[lll]^{\mathbf{R}\R} &&& \Ho(R)/\Ho(S)
\ar @/_1.5pc/[lll] \ar
 @/^1.5pc/[lll]
 }
\]
of triangulated categories, where we identify $\Ho(S)$ with $\Im(\Ho(\G))$ and $\pi:\Ho(R)\to\Ho(R)/\Im(\Ho(\G))$ is the quotient functor.
\end{prop}

\begin{proof}
We first view our homotopy categories as coming from the Gorenstein injective model structures. Then the existence of the adjoint triple $\Ho(\Ll)\dashv \Ho(\G)\dashv \mathbf{R}\R$ is immediate from \Cref{ind_tran}(ii) and \Cref{derived_nice}(i), and \Cref{derived_nice}(ii) says that $\Ho(\G)$ is fully faithful. We will be done by \Cref{Krause_criterion}(ii) if we show that these functors are triangulated and $\Im(\Ho(\G))$ is a thick subcategory.

For this, unravelling the definitions, recall that the suspension in $\Ho(S)$ of an arbitrary module $\m\in \Mod(S)$ (which is automatically cofibrant for the Gorenstein injective model structure) is given as the cokernel of a monomorphism to an injective module. Being an exact functor with exact left adjoint, the functor $\G$ preserves this construction so that $\Ho(\G)\circ\Sigma\cong \Sigma\circ \Ho(\G)$. Taking left and right adjoints of this identity gives us that all three functors commute with the shift functor. Furthermore, since $\G$ is exact and preserves the construction of $\Sigma$, we see that it in fact preserves the distinguished triangles as given in \eqref{d_triangles}. Thus $\Ho(\G)$ is triangulated. A completely dual argument using fiber sequences and the loop functor shows that $\mathbf{R}\R$ is also triangulated, recalling that $\R$ preserves short exact sequences of Gorenstein injectives by \Cref{right_adjoint_cohomology}. To show that $\Ho(\Ll)$ preserves distinguished triangles, we may argue exactly as above but using the Gorenstein projective model structure instead to explicitly describe the triangles. Indeed, $\Ll$ is exact and preserves the explicit construction of the shift, as the cokernel of a cofibration from a Gorenstein projective to a projective module, given by that model structure so that it preserves the construction in \eqref{d_triangles}.  Finally, $\Im(\Ho(\G))$ is in fact a localising subcategory and hence is thick, since $\Ho(\G)$ commutes with arbitrary direct sums as it is a left adjoint.
\end{proof}

\begin{rem}\label{triangulated_remark} The fact that $\Ho(\Ll)$, $\Ho(\G)$ and $\mathbf{R}\R$ are triangulated also follows immediately from general properties of stable model categories. Indeed, if $(\F,\G)$ is a Quillen adjunction between stable model categories, then the derived functors $\mathbf{L}\F$ and $\mathbf{R}\G$ commute with the suspension by \cite[Corollary 3.1.4]{BR}, noting that the loop functor is the inverse of suspension. Moreover, in any such model category the distinguished triangles in the homotopy category are given by cofiber sequences, and left/right derived functors preserve those (cf.\ \cite[Proposition 6.4.1 \& Theorem 7.1.11]{HovBook}).
\end{rem}

\subsection{Extension to several parabolics} We now extend the constructions from the previous subsection in the case where $R=\He$, $S=\He_M$ for some standard Levi subgroup $M$, and $\G=\Ind_{\He_M}^{\He}$ with left adjoint $\Ll=\LM$ and right adjoint $\R=\RM$. To simplify the notation, we will write $I^G_M$ for $\Ind_{\He_M}^{\He}$. We will denote by $\Ho(\He)_M$ the quotient triangulated category of $\Ho(\He)/\Ho(\He_M)$, and we have an adjoint triple
\[
\xymatrix@C=0.5cm{
\Ho(\He) \ar[rrr]^{\pi^G_M} &&& \Ho(\He)_M \ar @/_1.5pc/[lll]_{\ell^G_M}  \ar
 @/^1.5pc/[lll]^{r^G_M}
 }
\]
with $\ell^G_M$ and $r^G_M$ both fully faithful, as given by the recollement in \Cref{recollement}. The following result will be our main aim in this subsection.

\begin{thm}\label{many_Levi_rec} Let $\M=\{M_1,\ldots, M_n\}$ be a collection of standard Levi subgroups of $G$, and let $\Ss^G_\M$ be the smallest thick subcategory of $\Ho(\He)$ containing $\Im(\Ho(I_{M_1}^G)), \ldots, \Im(\Ho(I_{M_n}^G))$.
\begin{enumerate}
\item $\Ss^G_\M$ is a localising subcategory and the inclusion $i^G_\M:\Ss^G_\M\hookrightarrow\Ho(\He)$ has both a left and a right adjoint, which form part of a recollement
\[
\xymatrix@C=0.8cm{
\Ss^G_\M \ar[rrr]^{i^G_\M} &&& \Ho(\He) \ar[rrr]^{\pi^G_\M} \ar @/_1.5pc/[lll]_{\mathcal{L}^G_\M}  \ar
 @/^1.5pc/[lll]^{\mathcal{R}^G_\M} &&& \Ho(\He)_\M
\ar @/_1.5pc/[lll]_{\ell^G_\M} \ar @/^1.5pc/[lll]^{r^G_\M}
 }
\]
of triangulated categories, where $\Ho(\He)_\M=\Ho(\He)/\Ss^G_\M$.
\item Suppose $n>1$. Let $1\leq i\leq n$ and $\m\in \Ho(\He)$, and write $\M_i=\M\setminus\{M_i\}$. Then there is a distinguished triangle
$$
i^G_{\M_i}\mathcal{R}^G_{\M_i}(\m)\longrightarrow i^G_\M\mathcal{R}^G_\M(\m)\longrightarrow \m_i\longrightarrow \Sigma i^G_{\M_i}\mathcal{R}^G_{\M_i}(\m)
$$
where $\m_i= r^G_{\M_i}\pi^G_{\M_i}\Ho(I^G_{M_i})\mathbf{R}\mathcal{R}^G_{M_i}(\m)$, and similarly there is a distinguished triangle
$$
i^G_{\M_i}\mathcal{L}^G_{\M_i}(\m)\longrightarrow i^G_\M\mathcal{L}^G_\M(\m)\longrightarrow \m_i'\longrightarrow \Sigma i^G_{\M_i}\mathcal{L}^G_{\M_i}(\m)
$$
where $\m_i'=\ell^G_{\M_i}\pi^G_{\M_i}\Ho(I^G_{M_i})\Ho(\mathcal{L}^G_{M_i})(\m)$.
\item We have $\Im(\ell^G_\M)=\bigcap_{i=1}^n\Ker(\Ho(\mathcal{L}^G_{M_i}))$ and $\Im(r^G_\M)=\bigcap_{i=1}^n\Ker(\mathbf{R}\mathcal{R}^G_{M_i})$.
\end{enumerate}
\end{thm}

As a first step, we lift a result of Abe to the homotopy level.

\begin{lem}\label{ho_comp} Suppose that $M$ and $L$ are two standard Levi subgroups. Then we have:
\begin{enumerate}
\item $\Ho(\LM)\circ\Ho(I^G_L)\cong \Ho(I_{M\cap L}^M)\circ \Ho(\mathcal{L}_{M\cap L}^L)$; and
\item $\mathbf{R}\RM\circ\Ho(I^G_L)\cong \Ho(I_{M\cap L}^M)\circ \mathbf{R}\mathcal{R}_{M\cap L}^L$.
\end{enumerate}
\end{lem}

\begin{proof}
It was proved in \cite[Lemma 5.2]{Abe19_2} that $\RM\circ I^G_L\cong I_{M\cap L}^M\circ \mathcal{R}_{M\cap L}^L$. By taking left adjoints, we deduce that $\LM\circ I_L^G\cong I_{M\cap L}^M\circ \mathcal{L}_{M\cap L}^L$. Part (i) is now immediate, and part (ii) then follows by taking right adjoints in (i).
\end{proof}

To prove \Cref{many_Levi_rec} we will argue by induction on $n:=\abs{\M}$, the case $n=1$ being \Cref{recollement}. We will need to establish some preparatory results first, under the assumption that the statement of \Cref{many_Levi_rec}(i) holds for some fixed $n\geq 1$ and all split reductive groups. Assuming this from now on, we let $\M=\{M_1,\ldots, M_n\}$ be a collection of $n$ distinct standard Levi subgroups and we fix another standard Levi subgroup $L$. Then $L_i:=M_i\cap L$ is a standard Levi subgroup of $L$, and we may embed fully faithfully each $\Ho(\He_{L_i})$ inside $\Ho(\He_L)$ via $\Ho(I_{L_i}^L)$. Thus we may consider the smallest thick subcategory $\Ss_{\M\cap L}^L$ of $\Ho(\He_L)$ which contains these essential images, and by assumption we also have a recollement
\[
\xymatrix@C=0.5cm{
\Ss_{\M\cap L}^L \ar[rrr]^{i_{\M\cap L}^L} &&& \Ho(\He_L) \ar[rrr]^{\pi_{\M\cap L}^L} \ar @/_1.5pc/[lll]_{\mathcal{L}_{\M\cap L}^L}  \ar
 @/^1.5pc/[lll]^{\mathcal{R}_{\M\cap L}^L} &&& \Ho(\He_L)_{\M\cap L}
\ar @/_1.5pc/[lll]_{\ell^L_{\M\cap L}} \ar @/^1.5pc/[lll]^{r^L_{\M\cap L}}
 }.
\]
With this notation and assumptions established, we have the following:

\begin{lem}\label{triple_fun} \begin{enumerate}
\item The functors $\Ho(I^G_L)$, $\mathbf{R}\mathcal{R}_L^G$ and $\Ho(\mathcal{L}_L^G)$ induce a triple of triangulated functors
\[
\xymatrix@C=0.5cm{
\Ho(\He_L)_{\M\cap L} \ar[rrr]^{(I^G_L)_\M} &&& \Ho(\He)_\M \ar @/_1.5pc/[lll]_{(\mathcal{L}^G_L)_\M}  \ar
 @/^1.5pc/[lll]^{(\mathcal{R}_L^G)_\M}
 }
\]
such that
\begin{align*}
\pi_\M^G\circ \Ho(I^G_L)&=(I^G_L)_\M\circ\pi_{\M\cap L}^L,\\
\pi_{\M\cap L}^L\circ \Ho(\mathcal{L}_L^G)&=(\mathcal{L}^G_L)_\M\circ\pi_\M^G,\\
\text{and}\quad\quad\pi_{\M\cap L}^L\circ \mathbf{R}\mathcal{R}_L^G &=(\mathcal{R}^G_L)_\M\circ\pi_\M^G.
\end{align*}
\item The functor $(I^G_L)_\M$ is naturally isomorphic to both $\pi_\M^G\circ \Ho(I^G_L)\circ \ell^L_{\M\cap L}$ and $\pi_\M^G\circ \Ho(I^G_L)\circ r^L_{\M\cap L}$. Similarly, we have natural isomorphisms
$$
(\mathcal{L}^G_L)_\M\cong \pi_{\M\cap L}^L\circ \Ho(\mathcal{L}_L^G) \circ \ell_\M^G
$$
and
$$
(\mathcal{R}^G_L)_\M\cong \pi_{\M\cap L}^L\circ \mathbf{R}\mathcal{R}_L^G \circ r_\M^G.
$$
\end{enumerate}
\end{lem}

\begin{proof}
(i) By transitivity of parabolic induction (cf.\  \cite[Proposition 4.3]{Vign5}), for each $1\leq i\leq n$ we have
$$
I^G_L\circ I^L_{L_i}\cong I^G_{L_i}\cong I^G_{M_i}\circ I^{M_i}_{L_i}
$$
and so in particular $I^G_L\circ I^L_{L_i}$ factors through $I^G_{M_i}$. Hence the same holds for the corresponding homotopy functors.  As a result it follows that the composite $\pi_\M^G\circ \Ho(I^G_L): \Ho(\He_L)\to \Ho(\He)_\M$ kills the essential image of each $\Ho(I^L_{L_i})$ and hence kills $\Ss_{\M\cap L}^L$. By the universal property of quotients, this implies that there is an induced triangulated functor
$$
(I^G_L)_\M:\Ho(\He_L)_{\M\cap L}\to \Ho(\He)_\M
$$
with the property that $\pi_\M^G\circ \Ho(I^G_L)=(I^G_L)_\M\circ\pi_{\M\cap L}^L$ as required.

The construction of $(\mathcal{R}^G_L)_\M$ is very similar. By \Cref{ho_comp}(ii), the composite $\mathbf{R}\mathcal{R}_L^G\circ\Ho(I^G_{M_i})$ factors through $\Ho(I^L_{L_i})$ for each $1\leq i\leq n$. Hence it follows that the composite $\pi_{\M\cap L}^L\circ \mathbf{R}\mathcal{R}_L^G$ kills the essential image of $\Ho(I^G_{M_i})$ for each $i$, and hence it kills $\Ss^G_\M$. Again, we then have that $\pi_{\M\cap L}^L\circ \mathbf{R}\mathcal{R}_L^G$ factorises as $(\mathcal{R}^G_L)_\M\circ\pi_\M^G$ for some triangulated functor
$$
(\mathcal{R}^G_L)_\M:\Ho(\He)_\M\to \Ho(\He_L)_{\M\cap L}.
$$
The construction of $(\mathcal{L}^G_L)_\M$ is completely analogous, using \Cref{ho_comp}(i).

(ii) The two functors $r_\M^G$ and $\ell_\M^G$ are both fully faithful and hence we have natural isomorphisms
$$
\pi_\M^G\circ \ell_\M^G\cong \id_{\Ho(\He)_\M}\cong \pi_\M^G\circ r_\M^G
$$
by \Cref{ff_counit}. So we get the identity for $(\mathcal{R}^G_L)_\M$ by pre-composing the functor $\pi_{\M\cap L}^L\circ \mathbf{R}\mathcal{R}_L^G=(\mathcal{R}^G_L)_\M\circ\pi_\M^G$ by $r_\M^G$ and applying the above natural isomorphism. The identities for $(\mathcal{L}^G_L)_\M$ and $(I^G_L)_\M$ are obtained similarly.
\end{proof}

\begin{prop}\label{double_recollement} Let $\M'=\M\cup\{L\}$. The triple of functors from \Cref{triple_fun} is an adjoint triple, and it forms part of a recollement of the form
\begin{equation}\label{d_recollement}
\xymatrix@C=0.5cm{
\Ho(\He_L)_{\M\cap L} \ar[rrr]^{(I^G_L)_\M} &&& \Ho(\He)_\M \ar[rrr]^{\pi^\M_{\M'}} \ar @/_1.5pc/[lll]_{(\mathcal{L}^G_L)_\M}  \ar
 @/^1.5pc/[lll]^{(\mathcal{R}_L^G)_\M} &&& \T_{\M'}
\ar @/_1.5pc/[lll]_{\ell^\M_{\M'}} \ar @/^1.5pc/[lll]^{r^\M_{\M'}}
 }
\end{equation}
where $\T_{\M'}=\Ho(\He)_\M/\Ho(\He_L)_{\M\cap L}$.
\end{prop}

\begin{proof}
First, we show that the functors are adjoints and that $(I^G_L)_\M$ is fully faithful. This follows from the identities in \Cref{triple_fun}(ii). Indeed, we see that $(\mathcal{R}_L^G)_\M\cong \pi_{\M\cap L}^L\circ \mathbf{R}\mathcal{R}_L^G \circ r_\M^G$ is a composite of right adjoints so is a right adjoint, with left adjoint $\pi_\M^G\circ \Ho(I^G_L)\circ \ell^L_{\M\cap L}\cong (I^G_L)_\M$. Similarly, $(\mathcal{L}_L^G)_\M\cong \pi_{\M\cap L}^L\circ \Ho(\mathcal{L}_L^G) \circ \ell_\M^G$ is a composite of left adjoints, so is left adjoint to $\pi_\M^G\circ \Ho(I^G_L)\circ r^L_{\M\cap L}\cong (I^G_L)_\M$ as required. Finally, we saw in the proof of \Cref{triple_fun} that $\pi_{\M\cap L}^L\circ \mathbf{R}\mathcal{R}_L^G\circ i^G_\M=0$. After taking left adjoints this gives that
$$
\Im(\Ho(I^G_L)\circ \ell^L_{\M\cap L})\subseteq \Ker(\mathcal{L}^G_\M)=\Im(\ell^G_\M).
$$
As $(\pi_\M^G)|_{\Im(\ell^G_\M)}$ is an equivalence, we deduce that $(I^G_L)_\M\cong \pi_\M^G\circ \Ho(I^G_L)\circ \ell^L_{\M\cap L}$ is a composite of fully faithful functors and so is fully faithful.

Finally, we will be done by \Cref{Krause_criterion}(ii) if we show that the essential image of $(I^G_L)_\M$ is localising (and in particular thick). For this, note that since $\Ho(\He_M)$ and $\Ho(\He)$ have arbitrary direct sums so do their quotient categories $\Ho(\He_L)_{\M\cap L}$ and $\Ho(\He)_\M$ because the quotient functors are left adjoints. Then $(I^G_L)_\M$ preserves these direct sums since it is also a left adjoint.
\end{proof}

\begin{proof}[Proof of \Cref{many_Levi_rec}]
(i) As stated earlier, we work by induction on $n$. Assuming the statements hold for $n\geq 1$, let $M_1,\ldots, M_{n+1}$ be distinct standard Levi subgroups and write $\M=\{M_1, \ldots, M_n\}$ and $\M'=\M\cup\{M_{n+1}\}$ as before. By setting $L=M_{n+1}$ in \Cref{double_recollement}, we see that there is a recollement as in \eqref{d_recollement} and we write again $\T_{\M'}$ for the quotient $\Ho(\He)_\M/\Ho(\He_{M_{n+1}})_{\M\cap M_{n+1}}$.

We then define the functor $\pi^G_{\M'}$ to be the composite $\pi^\M_{\M'}\circ \pi^G_\M:\Ho(\He)\to \T_{\M'}$ and let $\Ss_{\M'}:=\Ker(\pi^G_{\M'})$. Note that $\pi^G_{\M'}$ has fully faithful left and right adjoints, given by $\ell^G_{\M'}:=\ell^G_\M\circ \ell^\M_{\M'}$ and $r^G_{\M'}:= r^G_\M\circ r^\M_{\M'}$ respectively. Therefore it follows from \Cref{Krause_criterion} that the inclusion $\Ss_{\M'}\to \Ho(\He)$ also has both a left and a right adjoint, and that these adjunctions form a recollement. Note also that $\Ss_{\M'}$ is automatically localising, because the functor $\pi^G_{\M'}$ commutes with arbitrary direct sums as it is a left adjoint. We are thus only left to check that $\Ss^G_{\M'}=\Ss_{\M'}$.

For this, we need to show that if $\Ss$ is an arbitrary thick subcategory of $\Ho(\He)$ containing the essential images of $\Ho(I_{M_1}^G),\ldots, \Ho(I_{M_{n+1}}^G)$, then $\Ss$ contains $\Ss_{\M'}$. It suffices to show that the quotient functor $\mathcal{Q}:\Ho(\He)\to \Ho(\He)/\Ss$ factors through $\T_{\M'}$. Certainly, $\Ss$ contains $\Ss^G_\M$ by induction and hence $\mathcal{Q}$ factors through some $\mathcal{Q}':\Ho(\He)_\M\to \Ho(\He)/\Ss$. But \Cref{triple_fun}(ii) then ensures that $\mathcal{Q}'$ kills the essential image of $(I^G_{M_{n+1}})_\M$, and hence it factors through $\T_{\M'}$ as claimed.

For (ii), we assume $n>1$ and fix $1\leq i\leq n$. We denote by $\eta^G_\M$ (resp.\ $\eta^G_{\M_i}$, resp.\ $\eta^{\M_i}_\M$) the unit of the adjunction $\pi^G_{\M}\dashv r^G_{\M}$ (resp.\ $\pi^G_{\M_i}\dashv r^G_{\M_i}$, resp.\ $\pi^{\M_i}_\M\dashv r^{\M_i}_\M$). Then by the recollement in (i) and \eqref{triangle_right_adjoints} we have two distinguished triangles
\begin{align*}
i^G_{\M_i}\mathcal{R}^G_{\M_i}(\m)\longrightarrow &\m \stackrel{\eta^G_{\M_i}}{\longrightarrow}r^G_{\M_i}\pi^G_{\M_i}(\m)\longrightarrow \Sigma i^G_{\M_i}\mathcal{R}^G_{\M_i}(\m)\quad\text{and}\\
i^G_{\M}\mathcal{R}^G_{\M}(\m)\longrightarrow &\m \stackrel{\eta^G_\M}{\longrightarrow}r^G_{\M}\pi^G_{\M}(\m)\longrightarrow \Sigma i^G_{\M}\mathcal{R}^G_{\M}(\m)
\end{align*}
for any $\m\in \Ho(\He)$. Recall that $\pi^G_{\M}=\pi^{\M_i}_\M\circ \pi^G_{\M_i}$ and $r^G_{\M}=r^G_{\M_i}\circ r^{\M_i}_\M$. Next, we consider the triangle \eqref{triangle_right_adjoints} for the object $\overline{\m}:=\pi^G_{\M_i}(\m)$ in the setting of the recollement \eqref{d_recollement} for the collection of Levi subgroups $\M_i$ and $L=M_i$. Applying the triangulated functor $r^G_{\M_i}$ to it, we get
$$
\mathfrak{n}\longrightarrow r^G_{\M_i}(\overline{\m})\stackrel{\varphi}{\longrightarrow}r^G_{\M}\pi^G_{\M}(\m)\longrightarrow \Sigma \mathfrak{n}
$$
where $\mathfrak{n}:=r^G_{\M_i}(I^G_{M_i})_{\M_i}(\mathcal{R}^G_{M_i})_{\M_i}(\overline{\m})$ and $\varphi=r^G_{\M_i}(\eta^{\M_i}_\M)_{\overline{\m}}$. Since we have $\eta^G_\M=\varphi\circ\eta^G_{\M_i}$, we may apply the octahedral axiom to obtain a distinguished triangle
\begin{equation}\label{triangle_R}
i^G_{\M_i}\mathcal{R}^G_{\M_i}(\m)\longrightarrow i^G_{\M}\mathcal{R}^G_{\M}(\m)\longrightarrow \mathfrak{n}\longrightarrow\Sigma i^G_{\M_i}\mathcal{R}^G_{\M_i}(\m),
\end{equation}
cf.\ \cite[0145]{Stacks}. Since $\mathfrak{n}\cong r^G_{\M_i}\pi^G_{\M_i} \Ho(I^G_{M_i})\mathbf{R}\mathcal{R}^G_{M_i}(\m)$, cf.\ \Cref{triple_fun}(i), this is the desired triangle. The construction for the left adjoints is completely dual to this. 

Finally, for (iii) we only prove the formula for the essential image of $r^G_\M$, the proof for $\ell^G_\M$ being completely analogous. We argue again by induction on $n$, the case $n=1$ being immediate from \Cref{recollement}. Assume the equality holds for some $n\geq 1$ and any collection of $n$ standard Levi subgroups. Again, let $M_1,\ldots, M_{n+1}$ be distinct standard Levi subgroups and write $\M=\{M_1, \ldots, M_n\}$ and $\M'=\M\cup\{M_{n+1}\}$. Note first that $\Im(r^G_{\M'})=\Ker(\mathcal{R}^G_{\M'})$ and $\Im(r^G_\M)=\Ker(\mathcal{R}^G_\M)$. So by the induction hypothesis, the inclusion $\bigcap_{i=1}^{n+1}\Ker(\mathbf{R}\mathcal{R}^G_{M_i})\subseteq \Im(r^G_{\M'})$ is automatic from the triangle \eqref{triangle_R} applied to $\M'$ and $i=n+1$. Assuming now that $\m\in\Ker(\mathcal{R}^G_{\M'})$, the same application of \eqref{triangle_R} then implies that $i^G_\M\mathcal{R}^G_\M(\m)$ lies in $\Im(r^G_\M)=\Ker(\mathcal{R}^G_\M)$. But since $i^G_\M$ is fully faithful we may apply \Cref{ff_counit} to obtain
$$
0=\mathcal{R}^G_\M i^G_\M\mathcal{R}^G_\M(\m)\cong \mathcal{R}^G_\M(\m)
$$
and the induction hypothesis gives $\m\in \bigcap_{j=1}^n\Ker(\mathbf{R}\mathcal{R}^G_{M_j})$. Replacing $\M$ by $\M'\setminus\{M_n\}$ and applying \eqref{triangle_R} to $\M'$ and $i=n$ instead, the above argument gives that $\m\in \bigcap_{j\neq n}\Ker(\mathbf{R}\mathcal{R}^G_{M_j})$ as well and hence we obtain $\m\in \bigcap_{j=1}^{n+1}\Ker(\mathbf{R}\mathcal{R}^G_{M_j})$ as claimed.
\end{proof}

\subsection{Supersingular modules and the derived adjoints}

We must check that the recollement obtained in \Cref{many_Levi_rec}(i) is non-trivial, i.e.\ that the quotient categories $\Ho(\He)_\M$ are non-zero. This will follow by considering supersingular modules. In this subsection, we assume that $k$ is algebraically closed. For the next result, we let
$$
c:=\min_{s_\alpha\in S} |\mu_\alpha|
$$
with $\mu_\alpha$ as in \eqref{quadratic_rel}, and note that $c=1$ or $2$.

\begin{prop}\label{exist_ss_inf_proj_dim} Suppose that either $G$ does not have root system of type $A_1\times\cdots\times A_1$ (possibly empty product) or that $r_{ss}>0$ and $q >2c$. Then there exists a simple supersingular $\He$-module of infinite projective dimension.
\end{prop}

\begin{proof}
If $G$ does not have root system of type $A_1\times\cdots\times A_1$ then any simple supersingular $\He$-module has infinite projective dimension by \cite[Theorem 7.7]{Koz}. So assume that $G$ has type $A_1\times\cdots\times A_1$ (non-empty product) and that $q >2c$. By \cite[Theorem 7.7]{Koz} again, our claim amounts to showing that there exists a $k$-valued character $\xi$ of the group $T(\mathbb{F}_q)$ such that $S_\xi\neq S$ (cf.\ \eqref{S_xi}). But since $p$ does not divide the order of $T(\mathbb{F}_q)$ and as $T(\mathbb{F}_q)$ is abelian, the intersection of the kernels of all the non-trivial $k$-valued characters of $T(\mathbb{F}_q)$ is trivial. Let $s=s_\alpha\in S$ be such that $|\mu_\alpha|=c$, and note that the group $\alpha^\vee(\mathbb{F}_q^\times)$ is non-trivial since $q >2c$. Then there must exist a character $\xi$ with $\xi|_{\alpha^\vee(\mathbb{F}_q^\times)}\neq 1$ and so $S_\xi\neq S$.
\end{proof}

\begin{prop}\label{ss_LM_RM_vanish} Suppose that $M$ is a proper Levi subgroup of $G$ and let $\m\in\Mod(\He)$.
\begin{enumerate}
\item If $\m$ is supersingular, then $\Ho(\LM)(\m)=0$ in $\Ho(\He_M)$.
\item If $\m$ is uniformly supersingular, then $\mathbf{R}\RM(\m)=0$ in $\Ho(\He_M)$.
\end{enumerate}
\end{prop}

\begin{proof}
These were proved in \cite[Proposition 4.27]{AHV18} (see also \cite[Proposition 5.18]{Abe19_2}) at the abelian level. Therefore (i) follows immediately. For (ii), let $\lambda_0^+$ be as in \Cref{lambda_pm} and let $\mathcal{O}^+=\widetilde{W}\cdot \lambda_0^+$. Write $z:=z_{\mathcal{O}^+}$ and fix $n\geq 1$ such that $\m\cdot z^n=0$.  First,  take an arbitrary $\m'\in\Mod(\He)$.  Since the element $z\in\He$ is central,  multiplication by $z^n$ is $\He$-linear on $\m'$ and hence induces an $\He_M$-linear endomorphism $f\mapsto f\cdot z^n$ of $\RM(\m')=\Hom_{\He_M^+}(\He_M,\m')$ (defined by $(f\cdot z^n)(h)=f(h)\cdot z^n$).  We claim that this is in fact an automorphism of $\RM(\m')$. 

To see this, we use the fact that $T_{\lambda_0^+}^{n+1}=z^nT_{\lambda_0^+}$ (cf.\ \cite[Lemma 5.17]{Abe19_2} and \cite[Lemma 2.18]{Vign5}).  Now let $f\in \RM(\m')$.  Suppose first that $f\cdot z^n=0$.  Then for all $h\in\He_M$,  we have
$$
0=f(h)\cdot z^nT_{\lambda_0^+}=f(h)\cdot T_{\lambda_0^+}^{n+1}=f(h\cdot (T^M_{\lambda_0^+})^{n+1})
$$
by definition of the map $j_M^+$ (cf.\ \eqref{embed_plusminus}).  Since $T^M_{\lambda_0^+}$ is invertible in $\He_M$,  this means that $f=0$.  To get surjectivity,  define $g:\He_M\to \m'$ by $g(h)=f(h\cdot (T^M_{\lambda_0^+})^{-n-1})\cdot T_{\lambda_0^+}$.  This is $\He_{M^+}$-linear because $T^M_{\lambda_0^+}$ is central in $\He_M$,  and we have
\begin{align*}
(g\cdot z^n)(h)&=f(h\cdot (T^M_{\lambda_0^+})^{-n-1})\cdot T_{\lambda_0^+}z^n\\
&=f(h\cdot (T^M_{\lambda_0^+})^{-n-1})\cdot T_{\lambda_0^+}^{n+1}\\
&=f(h)
\end{align*}
for all $h\in\He_M$.

We now specialise to the case where $\m'=Q_f\m$ is any Gorenstein injective replacement of $\m$, which by definition comes equipped with an $\He$-linear weak equivalence $i_{\m}:\m\to Q_f\m$.  Since we may view the commutative square
\[
\begin{tikzcd}
\m\arrow{r}{\cdot z^n} \arrow[swap]{d}{i_{\m}} & \m \arrow{d}{i_{\m}}\\
Q_f\m\arrow{r}{\cdot z^n} & Q_f\m
\end{tikzcd}
\]
as a square in $\Ho(\He)$,  we deduce that multiplication by $z^n$ induces the zero endomorphism of $Q_f\m$ in the homotopy category.  Thus the image of this endomorphism under the additive functor $\mathbf{R}\RM$ is also zero.  As $Q_f\m$ is Gorenstein injective, i.e.\ fibrant, we have an isomorphism 
$$
0=\mathbf{R}\RM(Q_f\m\xrightarrow{\cdot z^n}Q_f\m)\cong\RM(Q_f\m\xrightarrow{\cdot z^n}Q_f\m)
$$
in the arrow category of $\Ho(\He_M)$. As we have seen above, the right hand map is an isomorphism, from which we deduce that $\mathbf{R}\RM(\m)\cong\RM(Q_f\m)\cong 0$ in $\Ho(\He_M)$, as required.
\end{proof}

\begin{rem}\label{abelian_vanish}
Similarly to the above proof, take $\lambda_0^{\pm}$ as in \Cref{lambda_pm} and let $\mathcal{O}^{\pm}=\widetilde{W}\cdot \lambda_0^{\pm}$. Then part (i) of the above result holds more generally whenever every element of $\m$ is killed by a power of $z_{\mathcal{O}^-}$, and part (ii) holds whenever $z_{\mathcal{O}^+}$ acts nilpotently on $\m$. This follows from the above proof and because the abelian analogue of \Cref{ss_LM_RM_vanish} holds in that higher level of generality, cf.\ the proof of \cite[Proposition 4.27]{AHV18}.
\end{rem}

\begin{cor}\label{ind_non_ess}
Suppose that either $G$ does not have root system of type $A_1\times\cdots\times A_1$ (possibly empty product) or that $r_{ss}>0$ and $q >2$. Then for any collection $\M$ of proper standard Levi subgroups, the quotient category $\Ho(\He)_\M$ is non-zero. In particular, for any proper standard Levi subgroup $M$ of $G$, we have that $I^G_M$ is not a Quillen equivalence.
\end{cor}

\begin{proof}
By \Cref{exist_ss_inf_proj_dim}, \Cref{ss_LM_RM_vanish} and \Cref{many_Levi_rec}(iii) there is a simple module $\m$ which is non-zero in $\Ho(\He)$ and which lies in both $\Ker(\mathcal{L}^G_\M)$ and $\Ker(\mathcal{R}^G_\M)$. By the recollement in \Cref{many_Levi_rec}(i), $\Ho(\He)_\M$ is equivalent to both of these via the functors $\ell^G_\M$ and $r^G_\M$, so we are done.
\end{proof}

As another corollary,  we deduce that uniformly supersingular modules are Ext-orthogonal to parabolically induced modules in large enough degree without assuming any finiteness condition on these modules. This extends to any positive degree for finite length modules (cf.\ \Cref{ss_ext_orthogonal_finite}).

\begin{cor}\label{ss_ext_orthogonal} Suppose $\m\in\Mod(\He)$ is uniformly supersingular and let $\mathfrak{n}\in\Mod(\He_M)$ be any module. Then
$$
\Ext_{\He}^i(\Ind_{\He_M}^{\He}(\mathfrak{n}), \m)=0
$$
for all $i>r_{ss}+r_Z+1$.
\end{cor}

\begin{proof}
Let $Q_f\m$ be the fibrant replacement of $\m$.  It follows from \Cref{ss_LM_RM_vanish}(ii) that $\RM(Q_f\m)$  has finite injective dimension. Since $\RM(Q_f\m)$ is also Gorenstein injective by \Cref{right_adjoint_cohomology}(iv), we deduce that it is injective by \cite[Corollary 8.5]{Hov2}. Applying now \Cref{right_adjoint_cohomology}(ii)-(iii), we obtain that
$$
\Ext_{\He}^i(I^G_M(\mathfrak{n}), Q_f\m)\cong\Ext_{\He_M}^i(\mathfrak{n},\RM(Q_f\m))=0
$$
for any $i>0$.  Moreover, by definition there is a short exact sequence $0\to \m\to Q_f\m\to \mathfrak{e}\to 0$ with $\mathfrak{e}$ of finite injective dimension,  bounded above by $r_{ss}+r_Z$.  The long exact sequence on $\Ext$ and the vanishing above then gives us that
$$
\Ext_{\He}^{i-1}(I^G_M(\mathfrak{n}), \mathfrak{e})\cong \Ext_{\He}^i(I^G_M(\mathfrak{n}), \m)
$$
for all $i>1$. The bound on the injective dimension of $\mathfrak{e}$ gives the result.
\end{proof}

\begin{rem}
We note that the proofs of \Cref{ss_LM_RM_vanish} and \Cref{ss_ext_orthogonal} remain valid without the assumption that $k$ is algebraically closed.
\end{rem}

Next, we compute the image of all simple modules under the adjoints of parabolic induction at the homotopy level.  For that we need the corresponding formula at the level of the module categories, due to Abe, which we recall now. We refer the reader back to \eqref{defn_simple1}-\eqref{defn_simple2} for the construction of the simple $\He$-modules. 

\begin{thm}\label{Abe_formula} Let $(P, \mathfrak{n}, Q)$ be a standard supersingular triple for $\He$. Then for $P'=M'N'$ another standard parabolic subgroup of $G$, we have
$$
\mathcal{L}^G_{M'}(I_\He(P, \mathfrak{n}, Q))\cong\begin{cases}
I_{\He_{M'}}(P\cap M', \mathfrak{n}, Q\cap M') &\text{if $P\subseteq P'$ and $\Pi(\mathfrak{n})\subseteq \Pi_Q\cup\Pi_{P'}$,}\\
0 & \text{otherwise,}
\end{cases}
$$
and
$$
\mathcal{R}^G_{M'}(I_\He(P, \mathfrak{n}, Q))\cong\begin{cases}
I_{\He_{M'}}(P\cap M', \mathfrak{n}, Q\cap M') &\text{if $Q\subseteq P'$,}\\
0 & \text{otherwise.}
\end{cases}
$$
\end{thm}

\begin{proof}
The analogous statement was proved in \cite[Theorem 5.20]{Abe19_2} for the adjoints of parabolic coinduction and the simple modules presented in the form $CI_\He(P, \mathfrak{n}, Q)$. We just check that this still holds for the adjoints of parabolic induction and the simple modules in the form $I_\He(P, \mathfrak{n}, Q)$. Write $\sigma$ for the $\He_{M'}$-module $\mathcal{L}^G_{M'}(I_\He(P, \mathfrak{n}, Q))$. First, we have by \eqref{compatibility_simples} and \textit{loc.\ cit.} that
\begin{align*}
\sigma&\cong \iota_0^{w_0(M')}\Loc_{w_0(M')}(CI_\He(w_0(P), \iota_0^M\mathfrak{n}, w_0(Q)))\\
&\cong \begin{cases}
\iota_0^{w_0(M')}CI_{\He_{M'}}(w_0(P\cap M'), \iota_0^M\mathfrak{n}, w_0(Q\cap M')) &\text{if $P\subseteq P'$ and $\Pi(\mathfrak{n})\subseteq \Pi_Q\cup\Pi_{P'}$,}\\
0 & \text{otherwise.}
\end{cases}
\end{align*}
So we may assume that $P\subseteq P'$ and $\Pi(\mathfrak{n})\subseteq \Pi_Q\cup\Pi_{P'}$. Next, for any standard parabolic $Q\cap M'\subseteq Q'\subseteq (P\cap M')(\mathfrak{n})$ of $M'$ with Levi subgroup $L'$ we have $e_{\He_{w_0(L')}}(\iota_0^M\mathfrak{n})\cong \iota_0^{L'}e_{\He_{L'}}(\mathfrak{n})$, cf.\ \cite[Lemma 4.22]{AHV18}. From this and the above we get
\begin{align*}
\iota_0^{w_0(M')}\Coind_{\He_{w_0(L')}}^{\He_{w_0(M')}}\left(e_{\He_{w_0(L')}}(\iota_0^M\mathfrak{n})\right) &\cong \iota_0^{w_0(M')}\Coind_{\He_{w_0(L')}}^{\He_{w_0(M')}}\left(\iota_0^{L'}e_{\He_{L'}}(\mathfrak{n})\right)\\
&\cong \iota_0^{w_0(M')}\Ind_{\He_{w_{0,M'}(w_0(L'))}}^{\He_{w_0(M')}}\left(\iota_{w_0(M')}^{w_0(L')}\iota_0^{L'}e_{\He_{L'}}(\mathfrak{n})\right)\\
&\cong \Ind_{\He_{L'}}^{\He_{M'}}\left(\iota_{0, w_0(M')}^{w_0(L')}\iota_{w_0(M')}^{w_0(L')}\iota_0^{L'}e_{\He_{L'}}(\mathfrak{n})\right)\\
&\cong \Ind_{\He_{L'}}^{\He_{M'}}\left(\iota_0^{w_0(L')}\iota_0^{L'}e_{\He_{L'}}(\mathfrak{n})\right)\\
&\cong \Ind_{\He_{L'}}^{\He_{M'}}\left(e_{\He_{L'}}(\mathfrak{n})\right),
\end{align*}
where the second, third and fourth isomorphisms hold by \Cref{ind_coind_thm}, \cite[Lemma 4.9]{Vign5} and \eqref{transitivity_twists}, respectively. By exactness of $\iota_0^{w_0(M')}$ and from the definitions of both $I_{\He_{M'}}(P\cap M', \mathfrak{n}, Q\cap M')$ and $CI_{\He_{M'}}(w_0(P\cap M'), \iota_0^M\mathfrak{n}, w_0(Q\cap M'))$, it follows that we have an isomorphism
$$
I_{\He_{M'}}(P\cap M', \mathfrak{n}, Q\cap M')\cong \iota_0^{w_0(M')}CI_{\He_{M'}}(w_0(P\cap M'), \iota_0^M\mathfrak{n}, w_0(Q\cap M'))
$$
as required. The proof for the right adjoint is completely analogous.
\end{proof}

We now show that the analogue of the above holds at the homotopy level as well. For this, we need a couple of preliminary results.

\begin{prop}\label{finite_dim_acyclic} Let $\m$ be a finite length $\He$-module. Then $\m$ is $\RM$-acyclic for any standard Levi subgroup $M$ of $G$. In particular, $R^i\RM(I_\He(P, \mathfrak{n}, Q))=0$ for any $i>0$ and any standard supersingular triple $(P, \mathfrak{n}, Q)$.
\end{prop}

\begin{proof}
Fix a standard Levi subgroup $M$. It was shown in \cite[Proposition 3.6(2)]{Abe22} that for any $\He_M$-module $\mathfrak{n}$ there are isomorphisms
$$
\Ext^j_\He(\Coind_{\He_M}^{\He}(\mathfrak{n}),\m)\cong \Ext^j_{\He_M}(\mathfrak{n},(\RM)'(\m))
$$
for all $j\geq 0$, where $(\RM)'$ denotes the right adjoint of parabolic coinduction. By applying \Cref{right_adjoint_cohomology}(iii), we deduce that finite length modules are acyclic for $(\RM)'$. Since $(\RM)'\cong \iota_0^{w_0(M)}\circ \mathcal{R}^G_{w_0(M)}$, cf.\ \Cref{ind_coind_thm}(i), and $\iota_0^{w_0(M)}$ is an exact equivalence, we get that finite length modules are acyclic for $\mathcal{R}^G_{w_0(M)}$. As $M$ was arbitrary we are done.
\end{proof}

\begin{rem} It was pointed out to us by one of the referees that this acyclicity result may also be proved directly from the definition of $\RM$ as follows. There is by definition for any $\m\in\Mod(\He)$ a natural isomorphism $\RM(\m)\cong \varprojlim \m$, where the transition maps in the projective system are given by $m\mapsto m\cdot T_{\lambda_0^+}$ (cf.\ \cite[\S 5.1]{Abe19_2}). When $\m$ has finite length, i.e.\ is finite dimensional over $k$, one sees straightforwardly that this projective system satisfies the Mittag-Leffler condition, from which one deduces acyclicity.
\end{rem}

\begin{rem}\label{ss_ext_orthogonal_finite}
Suppose $\m\in\Mod(\He)$ is supersingular and of finite length, and let $\mathfrak{n}\in\Mod(\He_M)$ be any module. Then we have isomorphisms
$$
\Ext_{\He}^i(I_M^G(\mathfrak{n}), \m)\cong \Ext_{\He_M}^i(\mathfrak{n}, \RM(\m))= 0
$$
for all $i\geq 0$, by applying \Cref{finite_dim_acyclic}, \Cref{right_adjoint_cohomology}(iii) and using the fact that $\RM(\m)=0$ by \cite[Proposition 5.18]{Abe19_2}.  This strengthens \Cref{ss_ext_orthogonal} in the finite length case.
\end{rem}

\begin{lem}\label{acyclic_total_derived} Let $M$ be a standard Levi subgroup of $G$ and let $\m\in\Mod(\He)$ be $\RM$-acyclic. Then there is an isomorphism $\mathbf{R}\RM(\m)\cong \RM(\m)$ in $\Ho(\He_M)$.
\end{lem}

\begin{proof}
Taking a fibrant replacement of $\m$ in the Gorenstein injective model structure, we get a short exact sequence
$$
0\to \m\to Q_f\m\to \mathfrak{e}\to 0
$$
where $\mathfrak{e}$ has finite injective dimension. As both $\m$ and $Q_f\m$ are $\RM$-acyclic, cf. \Cref{right_adjoint_cohomology}(ii), so is $\mathfrak{e}$. By applying \Cref{right_adjoint_cohomology}(iii), we then get isomorphisms
$$
\Ext^j_\He(I^G_M(\mathfrak{n}),\mathfrak{e})\cong \Ext^j_{\He_M}(\mathfrak{n},\RM(\mathfrak{e}))
$$
for all $\mathfrak{n}\in\Mod(\He_M)$ and all $j\geq 0$. It then follows that $\RM(\mathfrak{e})$ also has finite injective dimension. By acyclicity of $\m$, the sequence
$$
0\to \RM(\m)\to \RM(Q_f\m)\to \RM(\mathfrak{e})\to 0
$$
is exact. The above therefore implies that the map $\RM(\m)\to \RM(Q_f\m)$ is a weak equivalence. As $\mathbf{R}\RM(\m)\cong \RM(Q_f\m)$ in $\Ho(\He_M)$, we are done.
\end{proof}

\begin{prop}\label{Abe_formula_derived} Let $(P, \mathfrak{n}, Q)$ be a standard supersingular triple for $\He$. Then for $P'=M'N'$ another standard parabolic subgroup of $G$, we have
$$
\Ho(\mathcal{L}^G_{M'})(I_\He(P, \mathfrak{n}, Q))\cong\begin{cases}
I_{\He_{M'}}(P\cap M', \mathfrak{n}, Q\cap M') &\text{if $P\subseteq P'$ and $\Pi(\mathfrak{n})\subseteq \Pi_Q\cup\Pi_{P'}$,}\\
0 & \text{otherwise,}
\end{cases}
$$
and
$$
\mathbf{R}\mathcal{R}^G_{M'}(I_\He(P, \mathfrak{n}, Q))\cong\begin{cases}
I_{\He_{M'}}(P\cap M', \mathfrak{n}, Q\cap M') &\text{if $Q\subseteq P'$,}\\
0 & \text{otherwise.}
\end{cases}
$$
\end{prop}

\begin{proof}
The statement for the left adjoints follows immediately from \Cref{Abe_formula}. The statement for the right adjoints follows additionally from \Cref{finite_dim_acyclic} and \Cref{acyclic_total_derived}.
\end{proof}

It was shown in \cite[Theorem 8.5]{Koz} that, when $p\nmid \abs{\Omega_{\text{tor}}}$, a simple module $I_\He(P,\mathfrak{n},Q)$ has finite projective dimension if and only if $\mathfrak{n}$ has finite projective dimension. We now give a new proof of the `only if' direction  without any assumptions on the prime $p$ and without using any Bruhat-Tits theory. This result is not new, in that the proof of the equivalence in \textit{loc.\ cit.} only uses their assumption on $p$ in the `if' direction, although they never state this explicitly. Using their restriction on $p$, we can also essentially reduce our aim of studying the isomorphism classes of simple modules in $\Ho(\He)$ to the supersingular case.

\begin{cor}\label{half_Koziol} Let $\m=I_\He(P,\mathfrak{n},Q)$ and $\m'=I_\He(P',\mathfrak{n}',Q')$ be two simple $\He$-modules and write $P=MN$ for the Levi decomposition of $P$.
\begin{enumerate}
\item Assume $\m$ has finite projective dimension. Then $\mathfrak{n}\in \Mod(\He_M)$ has finite projective dimension.
\item Assume that $p\nmid \abs{\Omega_{\text{tor}}}$ and that $\pd_{\He}(\m)=\pd_{\He}(\m')=\infty$. If $\m\cong \m'$ in $\Ho(\He)$ then $P=P'$, $Q=Q'$ and $\mathfrak{n}\cong\mathfrak{n}'$ in $\Ho(\He_M)$.
\end{enumerate}
\end{cor}

\begin{proof}
(i) Let $L$ denote the Levi subgroup of $Q$. Then we have $0=\mathbf{R}\mathcal{R}^G_L(I_\He(P, \mathfrak{n}, Q))\cong I_{\He_L}(P\cap L, \mathfrak{n}, L)$ by \Cref{Abe_formula_derived}, so that we may assume that $Q=G$. Applying the left adjoint instead, we now have
$$
\mathfrak{n}\cong I_{\He_M}(M, \mathfrak{n}, M)\cong\Ho(\LM)(I_\He(P, \mathfrak{n},G))=0
$$
in $\Ho(\He_M)$, as required.

(ii) First suppose that $Q\neq Q'$. Up to swapping them we may assume that $Q \nsubseteq Q'$. Write $L'$ for the Levi subgroup of $Q'$. Then applying $\mathbf{R}\mathcal{R}^G_{L'}$ to $\m\cong \m'$, we obtain $0\cong I_{\He_{L'}}(P'\cap L', \mathfrak{n}', L')$ by \Cref{Abe_formula_derived}, contradicting our assumption that $\m'$ has infinite projective dimension by \cite[Theorem 8.5]{Koz}. Thus $Q=Q'$, and applying $\mathbf{R}\mathcal{R}^G_{L'}$ again we may assume that $Q=Q'=G$ as above. Similarly now, if $P\neq P'$ then without loss of generality we have $P \nsubseteq P'$ and, writing $M'$ for the Levi subgroup of $P'$, we then obtain a contradiction from \Cref{Abe_formula_derived} by applying $\Ho(\mathcal{L}^G_{M'})$.
\end{proof}

Going back to our recollement, we write `all' instead of `$\M$' in the subscripts of all the functors and categories featuring in \Cref{many_Levi_rec}(i) when the collection $\M$ of standard Levi subgroups considered is that of \emph{all} proper standard Levi subgroups. We now characterise precisely which simple modules of infinite projective dimension lie in $\Ho(\He)_{\text{all}}$, viewed as a full subcategory of $\Ho(\He)$ via either $\ell^G_{\text{all}}$ or $r^G_{\text{all}}$.

\begin{cor}\label{simple_ker_adjoints} Suppose that $p\nmid \abs{\Omega_{\text{tor}}}$, and let $\mathfrak{s}=I_\He(P,\mathfrak{n},Q)\in \Mod(\He)$ be a simple module of infinite projective dimension. Then we have:
\begin{enumerate}
\item $\mathfrak{s}\in\Ker(\mathcal{L}^G_{\text{\emph{all}}})$ if and only if $\Pi=\Pi(\mathfrak{n})$ and $P=Q$; and
\item $\mathfrak{s}\in\Ker(\mathcal{R}^G_{\text{\emph{all}}})$ if and only if $Q=P(\mathfrak{n})=G$.
\end{enumerate}
In particular, $\mathfrak{s}\in\Ker(\mathcal{L}^G_{\text{\emph{all}}})\cap \Ker(\mathcal{R}^G_{\text{\emph{all}}})$ if and only if it is supersingular.
\end{cor}

\begin{proof}
First observe that for any standard parabolic subgroup $P'=M'N'$ of $G$, we have that $I_{\He_{M'}}(P, \mathfrak{n}, Q\cap P')$ and $I_{\He_{M'}}(P, \mathfrak{n}, Q)$ also have infinite projective dimension, whenever these modules are well-defined. Indeed, this follows from \cite[Theorem 8.5]{Koz} by our assumption on $p$ and because $\mathfrak{s}$ is assumed to have infinite projective dimension. The statement of (ii) then follows from \Cref{Abe_formula_derived} and \Cref{many_Levi_rec}(iii), and for (i) we get
\begin{equation}\label{left_vanish}
\mathfrak{s}\in \Ker(\mathcal{L}^G_{\text{all}})\iff \text{ for all }P\subseteq P'\subsetneqq G,\; \Pi(\mathfrak{n})\nsubseteq \Pi_Q\cup \Pi_{P'}.
\end{equation}
So we have to check that the right hand side is equivalent to $\Pi=\Pi(\mathfrak{n})$ and $P=Q$. One direction is clear. For the other, note that if $\Pi\neq \Pi(\mathfrak{n})$ then $P'=P(\mathfrak{n})$ makes the right hand side of \eqref{left_vanish} fail. Moreover, if $\Pi=\Pi(\mathfrak{n})$ and $P\neq Q$ then note that when $Q=G$ the right hand side of \eqref{left_vanish} fails for all choices of $P'$. So assume that $\Pi=\Pi(\mathfrak{n})$ and $P\subsetneqq Q\subsetneqq G$, in which case we let $P'$ be the standard parabolic corresponding to $\Pi_P\cup (\Pi\setminus \Pi_Q)\subsetneqq\Pi$. Once again $P'$ makes the right hand side of \eqref{left_vanish} fail. This completes the proof of (i).
\end{proof}

\begin{rem}\label{abelian_rem} The corresponding result at the abelian level, without any assumption on the projective dimension, is also true with the same proof, using \Cref{Abe_formula} instead of its homotopy variant. Hence a simple module, or any finite length module more generally, is supersingular if and only if it is killed by the functors $\mathcal{L}^G_{M}$ and $\mathcal{R}^G_{M}$ for every proper standard Levi subgroup $M$.
\end{rem}

Assuming that the root system of $G$ is irreducible, we can extend the above to all modules of finite length.

\begin{prop}\label{finite_dim_mod_ss} Suppose that $p\nmid \abs{\Omega_{\text{tor}}}$.  When the root system of $G$ is irreducible, the following are equivalent for an $\He$-module $\m$ of finite length:
\begin{enumerate}
\item $\m\in\Ker(\mathcal{L}^G_{\text{\emph{all}}})$;
\item $\m\in\Ker(\mathcal{R}^G_{\text{\emph{all}}})$; and
\item the inclusion $\m_{\text{ss}}\subseteq \m$ is a weak equivalence, where $\m_{\text{ss}}$ denotes the largest supersingular submodule of $\m$. 
\end{enumerate}
\end{prop}

\begin{proof}
From \Cref{ss_LM_RM_vanish}, both (iii)$\Rightarrow$(i) and (iii)$\Rightarrow$(ii) are clear. For the converse, fix an $\He$-module $\m$ of finite length and let $M_1, \ldots, M_n$ denote all the proper Levi subgroups of $G$. Note that $\m$ is finite dimensional over $k$, cf.\ \cite[Lemma 6.9]{OS14}. Let $\epsilon\in\{-, +\}$. For each $1\leq i\leq n$, we let $\lambda_i^{\epsilon}\in\widetilde{\Lambda}$ be chosen as in \Cref{lambda_pm} for the Levi subgroup $M_i$ and we write $\mathcal{O}^{\epsilon}_i=\widetilde{W}\cdot \lambda_i^{\epsilon}$ and $z_i^{\epsilon}:=z_{\mathcal{O}^{\epsilon}_i}$.  Then let $\m_i^{\epsilon}=\{x\in\m\mid (z^{\epsilon}_i)^lx=0 \text{ for some $l>0$}\}$ and $\m^{\epsilon}_{\leq i}:=\bigcap_{j\leq i} \m^{\epsilon}_j=(\m^{\epsilon}_{\leq i-1})^{\epsilon}_i$. Note that $z^{\epsilon}_i$ then acts invertibly on the quotient $\m/\m^{\epsilon}_i$ for any $i$, so that any simple subquotient of $\m/\m^{\epsilon}_i$ is not supersingular. Hence we see iteratively that any simple subquotient of $\m/\m^+_{\leq n}$ is not supersingular and similarly that any simple subquotient of $\m^+_{\leq n}/(\m^+_{\leq n})^-_{\leq n}$ is not supersingular. But we have that $\m_{\text{ss}}=(\m^+_{\leq n})^-_{\leq n}$ by \Cref{abelian_rem} and \Cref{abelian_vanish}. Thus we see that every simple subquotient of $\m_{\text{non-ss}}:=\m/ \m_{\text{ss}}$ is not supersingular.  We will show that either (i) or (ii) implies that $\m_{\text{non-ss}}$ has finite projective dimension, and thus that (iii) holds by \cite[Lemma 5.8]{Hov2}.

Suppose first that $\mathfrak{s}=I_\He(P,\mathfrak{n},Q)$ is a simple subquotient of $\m_{\text{non-ss}}$. If $P(\mathfrak{n})\neq G$ then $\mathfrak{s}$ is in the essential image of $I^G_{M(\mathfrak{n})}$. If instead $P(\mathfrak{n})= G$ then $P=B$ since $P\neq G$ and the root system is irreducible, and thus $\mathfrak{s}$ has finite projective dimension by \cite[Theorem 8.5]{Koz} and the fact that $\He_T$ has finite global dimension (cf. \Cref{torus_finite_global_dim}).  Either way, it follows that $\mathfrak{s}\in\Ss^G_{\text{all}}$ and hence that $\m_{\text{non-ss}}\in \Ss^G_{\text{all}}$.

Next, by applying the functors $\mathcal{L}^G_{\text{all}}$ and $\mathcal{R}^G_{\text{all}}$ to the distinguished triangle $\m_{\text{ss}}\to \m\to\m_{\text{non-ss}}\to \Sigma\m_{\text{ss}}$, we obtain from the above and \Cref{ss_LM_RM_vanish} that
\begin{equation}\label{isos1}
\mathcal{R}^G_{\text{all}}(\m)\cong \mathcal{R}^G_{\text{all}}(\m_{\text{non-ss}})\text{ and }\mathcal{L}^G_{\text{all}}(\m)\cong \mathcal{L}^G_{\text{all}}(\m_{\text{non-ss}}).
\end{equation}
But since $\m_{\text{non-ss}}\in \Ss^G_{\text{all}}$ we see from the recollement in \Cref{many_Levi_rec} that
\begin{equation}\label{isos2}
\m_{\text{non-ss}}\in \Ker(\mathcal{L}^G_{\text{all}})\iff \m_{\text{non-ss}}\in \Ker(\mathcal{R}^G_{\text{all}})\iff \m_{\text{non-ss}}=0,
\end{equation}
where the latter equality is in $\Ho(\He)$. Thus, combining \eqref{isos1} and \eqref{isos2}, we obtain that either one of (i) or (ii) implies that $\m_{\text{non-ss}}$ has finite projective dimension as claimed.
\end{proof}

When the hypotheses of \Cref{finite_dim_mod_ss} are satisfied we tentatively call the quotient category $\Ho(\He)_{\text{all}}$ the \emph{supersingular homotopy category} of $\He$. We will see later that this triangulated category is in fact the homotopy category of a model structure on $\Mod(\He)$.

\begin{rem} If we write $\T$ for $\Ker(\mathcal{L}^G_{\text{all}})$, then \Cref{finite_dim_mod_ss} tells us that the modules of finite length in $\T$ are, up to weak equivalence, the supersingular ones whenever the root system of $G$ is irreducible and $p\nmid \abs{\Omega_{\text{tor}}}$. Moreover we have $\T^\ddagger=\Ss^G_{\text{all}}$ and $\T={^\ddagger}\Ss^G_{\text{all}}$. This is similar to the torsion pair constructed at the abelian level in \cite{OS18} where the authors showed that when $\mathbb{G}=\text{SL}_2$, $p>2$ and $\mathfrak{F}\neq \mathbb{Q}_p$, then the supersingular modules are in a torsion pair with a certain class of modules which they call `non-supersingular' and which contains all parabolically induced modules, cf.\ Theorems 3.26 \& 3.27 in \textit{loc.\ cit}. Since our Hom-orthogonality boils down to Ext-orthogonality between suitable replacements at the abelian level, the two questions are not immediately related. But we hope to study elsewhere the relationship between the Ollivier-Schneider torsion pair and the adjoints of parabolic induction.
\end{rem}

\section{Bousfield localisations from recollements}

\subsection{Constructing model structures from a recollement}

It is natural to ask whether the various quotient categories $\Ho(\He)_\M$ that were constructed in the last section are actually the homotopy category of some model structure on $\Mod(\He)$. To answer this we work in as general a framework as we can.

\begin{set}\label{setup_2} We let $\A$ be a Grothendieck abelian category equipped with a cofibrantly generated, abelian model structure $(\C_\A, \Fc_\A, \W_\A)$. We will work in one of the following two situations:
\begin{enumerate}
\item the model structure on $\A$ is injective, i.e.\ $\C_\A=\A$; or
\item the model structure on $\A$ is projective, i.e.\ $\Fc_\A=\A$.
\end{enumerate}
Recall that the model structure on $\A$ is then automatically hereditary and thus stable, and we let $\Ss$ be a thick subcategory of $\Ho(\A)$. We now assume additionally that in either situation the inclusion $i:\Ss\stackrel{\subseteq}{\longrightarrow}\Ho(\A)$ is part of a recollement
\begin{equation}\label{generic_recollement}
\xymatrix@C=0.5cm{
\Ss\ar[rrr]^{i} &&& \Ho(\A) \ar[rrr]^{\mathsf{\pi_\Ss}} \ar @/_1.5pc/[lll]_{\Ll_{\Ss}}  \ar
 @/^1.5pc/[lll]^{\R_{\Ss}} &&& \Ho(\A)/\Ss
\ar @/_1.5pc/[lll]_{\ell_\Ss} \ar
 @/^1.5pc/[lll]^{r_\Ss}
 }
\end{equation}
where none of the functors are a priori assumed to be the total derived functor of some Quillen functors.
\end{set}

We recall that the class of trivially fibrant objects in an injective abelian model category equals the class $\mathrm{Inj}(\A)$ of injective objects, and dually the class of trivially cofibrant objects in a projective abelian model category equals the class $\mathrm{Proj}(\A)$ projective objects (cf. \cite[Corollary 1.1.9]{Beck14}).

Our aim is to localise the model structure on $\A$ at $\Ss$ in a suitable way, or more precisely to construct an abelian model structure on $\A$ whose homotopy category is the quotient $\Ho(\A)/\Ss$. Our main result is:

\begin{thm}\label{recollement_model} Let $\A$ be a Grothendieck abelian category equipped with a cofibrantly generated, abelian model structure which is either injective or projective. Suppose that $\Ss$ is a thick subcategory of $\Ho(\A)$ satisfying \Cref{setup_2}. Let $\W_\Ss$, resp.\ $\Fc_\Ss$, resp.\ $\C_\Ss$, denote the class of objects of $\A$ whose image in $\Ho(\A)$ lies in $\Ss$, resp.\ $\Ker(\R_{\Ss})$, resp.\ $\Ker(\Ll_{\Ss})$.
\begin{enumerate}
\item If the model structure on $\A$ is injective then $\A^i_\Ss=(\A, \Fc_\A\cap \Fc_\Ss, \W_\Ss)$ defines another cofibrantly generated, injective, abelian model structure on $\A$ such that there is a canonical equivalence $\Ho(\A^i_\Ss)\cong \Ho(\A)/\Ss$, the adjunction $\id:\A\rightleftarrows \A^i_\Ss:\id$ is Quillen and the induced total derived adjunction is naturally isomorphic to the adjoint pair $\pi_\Ss\dashv r_\Ss$.
\item Dually, if the model structure on $\A$ is projective then $\A^p_\Ss=(\C_\A\cap \C_\Ss, \A, \W_\Ss)$ defines another cofibrantly generated, projective, abelian model structure on $\A$ such that there is a canonical equivalence $\Ho(\A^p_\Ss)\cong \Ho(\A)/\Ss$, the adjunction $\id:\A^p_\Ss\rightleftarrows \A:\id$ is Quillen and the induced total derived adjunction is naturally isomorphic to the adjoint pair $\ell_\Ss\dashv \pi_\Ss$.
\end{enumerate}
\end{thm}

\begin{rem} We briefly explain how our results in this section complement the existing literature on abelian model categories and recollements.
\begin{enumerate}
    \item We note that Becker had a similar aim when he constructed the singular projective/injective model structures on categories of modules over a DG-algebra in \cite{Beck14}, see in particular Propositions 1.4.2 \& 1.4.6 in \textit{loc. cit.} The situation there was more specific as Becker assumed the existence of two model structures on $\A$ and localised one with respect to the others, although with weaker assumptions on these model structures and on $\A$ than the ones we have above. Meanwhile, we do not assume that $\Ss$ is the homotopy category of some other model structure on $\A$. In fact, the above result will allow us to apply Becker's constructions, see \Cref{apply_Becker}, and our construction in \Cref{other_proj_inj} generalises Proposition 2.2.1 in \textit{loc. cit.}, which was only concerned with modules over a DG-algebra.
    \item It came to our attention after this paper first appeared that Gillespie proved quite generally in \cite[Theorem A(3)]{Gil16_2} that recollements in homotopy categories of injective/projective abelian model categories do arise from model structures and Quillen adjunctions. This very nice result was also inspired by the aforementioned techniques of Becker, however he never discusses the cofibrant generation of these model structures and the fibrant/cofibrant objects in his cotorsion pairs are only defined as the orthogonal complements of the trivial objects. \Cref{recollement_model} above may be viewed as a refined version of his result in the setting of \Cref{setup_2}.
\end{enumerate}
\end{rem}

To prove the theorem, we need to check that the conditions in \Cref{Hovey} are satisfied. First we have:

\begin{lem}\label{thick} The classes of objects $\C_\Ss$, $\Fc_\Ss$ and $\W_\Ss$ defined in \Cref{recollement_model} are thick.
\end{lem}

\begin{proof}
Since distinguished triangles arise from short exact sequences, the result follows immediately from the fact that $\Ss$, $\Ker(\R_{\Ss})$, and $\Ker(\Ll_{\Ss})$ are thick subcategories of the triangulated category $\Ho(\A)$.
\end{proof}

Next, we characterise the trivially fibrant/cofibrant objects in the proposed new model structures.

\begin{lem}\label{trivially_fibrant_cofibrant}
Suppose that we are in \Cref{setup_2}. When $\A$ is injective then $\W_\Ss\cap\Fc_\Ss=\W_\A$, so that $\W_\Ss\cap\Fc_\A\cap\Fc_\Ss=\W_\A\cap \Fc_\A=\mathrm{Inj}(\A)$. Dually, when $\A$ is projective then $\W_\Ss\cap\C_\Ss=\W_\A$, so that $\W_\Ss\cap\C_\A\cap\C_\Ss=\W_\A\cap \C_\A=\mathrm{Proj}(\A)$.
\end{lem}

\begin{proof}
We only prove the injective case, the other one being dual. Note that $\W_\A\cap \Fc_\A=\mathrm{Inj}(\A)$ is immediate from the injectivity of the model structure on $\A$, so that the first equality immediately implies the second. For the first equality, we have $\W_\A\subseteq \W_\Ss\cap\Fc_\Ss$ by definition of the right hand side.  Conversely, assume that $X\in \W_\Ss\cap\Fc_\Ss$. Then $X\cong i(Y)$ in $\Ho(\A)$, for some $Y\in\Ss$, and $X\in\Ker(\R_{\Ss})$. But then we obtain from \Cref{ff_counit} that $Y=0$ in $\Ho(\A)$, and hence that $X\in \W_\A$ as required.
\end{proof}

Let $\A$ be a Grothendieck abelian category. Recall that a cotorsion pair $(\D, \E)$ on $\A$ is said to be \emph{cogenerated} by a set $\mathcal{G}\subseteq \D$ if $\E=\mathcal{G}^\perp$. Furthermore, by definition $\A$ has a generating set $\mathscr{S}$ (or equivalently a generator), which is a set of objects of $\A$ such that every object of $\A$ is a quotient of a direct sum of objects in $\mathscr{S}$. The cotorsion pair $(\D, \E)$ is called \emph{small} if it is cogenerated by a set and if $\D$ contains a generating set for $\A$. We note that small cotorsion pairs are automatically functorially complete and that an abelian model structure $(\C, \Fc, \W)$ on $\A$ is cofibrantly generated if and only if both $(\C, \Fc\cap\W)$ and $(\C\cap\W, \Fc)$ are small, cf.\ \cite[Corollary 6.6]{Hov2} and \cite[Proposition 1.2.7]{Beck14}.

We will also need a notion of generating set for a triangulated category $\T$. A set of objects $\mathcal{G}$ of $\T$ will be called a \emph{set of weak generators} if an object $X\in\T$ is zero if and only if $\Hom_\T(G, \Sigma^nX)=0$ for all $n\in\Z$ and all $G\in \mathcal{G}$. The relevance of this notion for us is the result \cite[Theorem 7.3.1]{HovBook}, which in particular implies that the homotopy category $\Ho(\A)$ of a hereditary, abelian, cofibrantly generated model category $\A$ has a set of weak generators given by the cokernels of the generating cofibrations. The proof of \textit{loc.\ cit.} requires some input from the homotopy theory of simplicial sets, so we deemed it worthwhile to provide a more elementary proof in the case where $\A$ is the module category of a Gorenstein ring. In fact we obtain a stronger result than weak generation.

\begin{prop}\label{generators_Ho} If $S$ is a Gorenstein ring, then a module $\m\in \Mod(S)$ has infinite projective dimension if and only if there exists a right ideal $J$ such that $[S/J, \m]_S\neq 0$. In particular, the cyclic modules $\{S/J \mid J \text{ a right ideal of } S\}$ form a set of weak generators for $\Ho(S)$.
\end{prop}

\begin{proof}
We view $\Ho(S)$ as the homotopy category of the Gorenstein injective model structure. One direction is clear, so suppose that $\m\in \Mod(S)$ has the property that $[S/J, \m]_S= 0$ for all right ideals $J$. Taking a fibrant replacement if necessary, we may assume $\m$ to be Gorenstein injective. We first claim that if $\mathfrak{n}\in \Mod(S)$ is finitely generated then $[\mathfrak{n}, \m]_S= 0$. This follows by induction on the number $l$ of generators of $\mathfrak{n}$. Indeed, $\mathfrak{n}$ fits into a short exact sequence $0\to \mathfrak{n}'\to \mathfrak{n}\to \mathfrak{n}''\to 0$ where $\mathfrak{n}'$ is generated by $l-1$ elements and $\mathfrak{n}''$ is cyclic, and we may view this exact sequence as a distinguished triangle in $\Ho(S)$ and apply the cohomological functor $[-, \m]_S$ to get the claim.

We now claim that $\m$ is an injective module. Indeed, if $J$ is any right ideal of $S$ and $f: J\to \m$ is any map, then we must show that $f$ extends to $S$. But the Noetherianity of $S$ implies that $J$ is finitely generated and thus, by the above, $f$ must be zero in $\Ho(S)$. This implies that $f$ factors through a map $g:J\to \mathfrak{e}$ for some injective module $\mathfrak{e}$ by \Cref{factor_through}. Since $\mathfrak{e}$ is injective it follows that $g$ extends to $S$, and thus so does $f$ as required.
\end{proof}

\begin{proof}[Proof of \Cref{recollement_model}]
We first note that \Cref{trivially_fibrant_cofibrant} implies that $(\A,\W_\Ss\cap\Fc_\A\cap\Fc_\Ss)=(\A, \mathrm{Inj}(\A))$ in case (i) and $(\W_\Ss\cap\C_\A\cap\C_\Ss, \A)=(\mathrm{Proj}(\A), \A)$ in case (ii) are small cotorsion pairs, cf.\ \cite[Example 1.2.6]{Beck14}. Since $\W_\Ss$ is thick by \Cref{thick}, we are only left to show that $(\W_\Ss, \Fc_\A\cap\Fc_\Ss)$ and $(\C_\A\cap\C_\Ss, \W_\Ss)$ are small cotorsion pairs in order to get our claimed model structures (cf.\ \Cref{Hovey} and the discussion preceding \Cref{generators_Ho}).

(i) We need to construct a set $\mathcal{G}$ of objects of $\A$ such that $\Fc_\A\cap\Fc_\Ss=\mathcal{G}^\perp$ and $\W_\Ss={^\perp}(\mathcal{G}^\perp)$. By assumption on $\A$, we have $\Fc_\A=\mathcal{G}_1^\perp$ for some set $\mathcal{G}_1$ of objects in $\W_\A$. By enlarging $\mathcal{G}_1$ if necessary, we may assume that it contains a generating set for $\A$, since $(\W_\A, \Fc_\A)$ is small. We now let $\mathscr{S}$ be a set of weak generators for $\Ho(\A)$ and, for each $S\in \mathscr{S}$ and each $n\in\Z$, we choose $X_{n,S}\in\A=\C_\A$ such that $X_{n,S}\cong \Sigma^{-n}i\Ll_{\Ss}(S)$ in $\Ho(\A)$ and let $\mathcal{G}_2=\{X_{n,S}\mid S\in\mathscr{S}\text{ and }n\in \Z\}$. Then we claim that $\mathcal{G}=\mathcal{G}_1\cup \mathcal{G}_2$ satisfies our requirements. To see this, first note that if $Y\in \mathcal{G}^\perp$ then in particular $Y\in \mathcal{G}_1^\perp=\Fc_\A$. So in order to obtain $\Fc_\A\cap\Fc_\Ss=\mathcal{G}^\perp$ it suffices to show that for $Y\in \Fc_\A$, we have $Y\in \mathcal{G}_2^\perp$ if and only if $Y\in \Fc_\Ss$. Indeed, by \Cref{Hom_Ho} we have
\begin{align*}
Y\in \mathcal{G}_2^\perp & \iff \Hom_{\Ho(\A)}(\Sigma^{-n}i\Ll_{\Ss}(S), \Sigma Y)=0 \text{ for all } n\in\Z\text{ and }S\in \mathscr{S}\\
& \iff \Hom_{\Ho(\A)}(S, \Sigma^{n+1}i\R_{\Ss} (Y))=0\text{ for all } n\in\Z\text{ and }S\in \mathscr{S}\\
& \iff i\R_{\Ss} (Y)=0 \iff Y\in \Fc_\Ss
\end{align*}
as required. Next, note that every object of $\Fc_\Ss$ is isomorphic in $\Ho(\A)$ to an object of $\Fc_\A\cap\Fc_\Ss$. Thus we see by applying \Cref{Hom_Ho} again that an object $X\in\A$ lies in $^\perp(\Fc_\A\cap\Fc_\Ss)$ if and only if the image of $X$ in $\Ho(\A)$ lies in $^\ddagger\Ker(\R_{\Ss})=\Ss$.

This shows that $(\W_\Ss, \Fc_\A\cap\Fc_\Ss)$ is a small cotorsion pair, and thus we have that $\A^i_\Ss$ defines an injective, cofibrantly generated, abelian model structure as required. By \Cref{trivially_fibrant_cofibrant} we have that the trivial fibrations in $\A$ and in $\A^i_\Ss$ are the same, and thus $\id:\A^i_\Ss\to \A$ is right Quillen. Furthermore, since the trivially fibrant objects in both model structures are the same we have that $\mathbf{R}\id:\Ho(\A^i_\Ss)\to \Ho(\A)$ identifies, after restricting to the cofibrant-fibrant objects, with the canonical embedding $(\Fc_\Ss\cap\Fc_\A)/\mathrm{Inj}(\A)\hookrightarrow \Fc_\A/\mathrm{Inj}(\A)$ with image $\Ker(\R_{\Ss})\simeq\Ho(\A)/\Ss$, where the latter equivalence is induced by $r_\Ss$. Thus, after these identifications, the left adjoint of $\mathbf{R}\id$ is naturally isomorphic to $\pi_\Ss$.

(ii) As above, we let $\mathscr{S}$ be a set of weak generators for $\Ho(\A)$. Then for each $S\in \mathscr{S}$ and each $n\in\Z$, we choose $X_{n,S}\in\C_\A$ such that $X_{n,S}\cong \Sigma^{-n}\ell_\Ss \pi_\Ss(S)$ in $\Ho(\A)$ and let $\mathcal{G}=\{X_{n,S}\mid S\in\mathscr{S}\text{ and }n\in \Z\}$. Since the model structure on $\A$ is projective, there is a projective generator $Z$ for $\A$ and we may apply \cite[Corollary 6.9]{Hov2} to see that $\mathcal{G}$ cogenerates a small cotorsion pair $(\D, \E)$, where $\D$ is the closure of $\mathcal{G}\cup\{Z\}$ under direct summands and certain colimits called transfinite compositions. As $\C_\A$ is closed under those (cf.\ Lemma 6.2 in \textit{loc.\ cit.}), it follows that $\D\subseteq \C_\A$. We first show that $\E=\mathcal{G}^\perp=\W_\Ss$. This is similar to the argument in (i). As every object of $\mathcal{G}$ is cofibrant, we may apply \Cref{Hom_Ho} to see that
\begin{align*}
Y\in \mathcal{G}^\perp & \iff \Hom_{\Ho(\A)}(S, \Sigma^{n+1} r_\Ss \pi_\Ss (Y))=0\text{ for all } n\in\Z\text{ and }S\in \mathscr{S}\\
& \iff r_\Ss \pi_\Ss(Y)=0 \iff Y\in \W_\Ss
\end{align*}
as required, where the last equivalence holds because $\Ker(\pi_\Ss)=\Ss$. Next, since we have $\D\subseteq \C_\A$ it suffices to show that a cofibrant $X\in\C_\A$ lies in $^\perp \W_\Ss$ if and only if it lies in $\C_\Ss$, in order to obtain that $\D=\C_\A\cap\C_\Ss$. This is now analogous to (i): we have by \Cref{Hom_Ho} that $X\in{^\perp} \W_\Ss$ if and only if the image of $X$ in $\Ho(\A)$ lies in $^\ddagger\Ss=\Ker(\Ll_{\Ss})$, i.e.\ if and only if $X\in \C_\Ss$, as required. This concludes the proof that the abelian model structure $\A^p_\Ss$ exists and is cofibrantly generated. The last part is completely dual to (i).
\end{proof}

The construction of these model structures now gives us the input necessary to apply the techniques of Becker in \cite{Beck14}. We then obtain the following result.

\begin{cor}\label{apply_Becker} We keep the same assumptions as in \Cref{recollement_model}.
\begin{enumerate}
\item If the model structure on $\A$ is injective, then $\A_i^\Ss=(\W_\Ss, \Fc_\A, \Fc_\Ss)$ defines a hereditary, cofibrantly generated, abelian model structure on $\A$ whose homotopy category is canonically equivalent to $\Ss$. Furthermore, the adjunction $\id:\A_i^\Ss\rightleftarrows \A:\id$ is Quillen and its total derived adjunction is naturally isomorphic to the adjoint pair $(i,\R_{\Ss})$.
\item If the model structure on $\A$ is projective, then $\A_p^\Ss=(\C_\A, \W_\Ss, \C_\Ss)$ defines a hereditary, cofibrantly generated, abelian model structure on $\A$ whose homotopy category is canonically equivalent to $\Ss$. Furthermore, the adjunction $\id:\A\rightleftarrows \A_p^\Ss:\id$ is Quillen and its total derived adjunction is naturally isomorphic to the adjoint pair $(\Ll_{\Ss},i)$.
\end{enumerate}
\end{cor}

\begin{proof}
Everything follows immediately from  \Cref{recollement_model} and \cite[Propositions 1.4.2 \& 1.4.6]{Beck14}, except the cofibrant generation for which we need to check that the associated cotorsion pairs are small. In the injective case we have $(\W_\Ss\cap\Fc_\Ss, \Fc_\A)=(\W_\A, \Fc_\A)$ and in the projective case we have $(\C_\A, \W_\Ss\cap\C_\Ss)=(\C_\A, \W_\A)$ , cf.\ \Cref{trivially_fibrant_cofibrant}, and so it follows from our assumptions that these are small. Furthermore, the proof of \Cref{recollement_model} shows that $(\W_\Ss, \Fc_\A\cap\Fc_\Ss)$, in the injective case, and $(\C_\A\cap\C_\Ss, \W_\Ss)$, in the projective case, are small as well.
\end{proof}

We would like to be able to recover the entire recollement in one go, while the above only gives us one half of it in either situation. Furthermore, the model structures constructed in \Cref{apply_Becker} are no longer injective/projective. We now see how to fix this by adding one extra hypothesis.

\begin{prop}\label{other_proj_inj} Suppose that $\A$ is a Grothendieck abelian category equipped with both an injective model structure $(\A, \Fc_\A, \W_\A)$ and a projective model structure $(\C_\A, \A, \W_\A)$ with the same classes of trivial objects, and assume that these are both cofibrantly generated. Let $\Ho(\A)$ denote their common homotopy category and assume $\Ss$ is a thick subcategory of $\Ho(\A)$ satisfying the requirements of \Cref{setup_2}. Let $\W_\Ss$, $\Fc_\Ss$ and $\C_\Ss$ be as in \Cref{recollement_model}.
\begin{enumerate}
\item ${^\ddagger}\A^i_\Ss=(\A, \W_\Ss\cap\Fc_\A, \C_\Ss)$ defines an injective, cofibrantly generated, abelian model structure on $\A$. Furthermore, the identity $\id:{^\ddagger}\A^i_\Ss\to \A_i^\Ss$ is a right Quillen equivalence, the adjunction $\id:\A\rightleftarrows {^\ddagger}\A^i_\Ss:\id$ is Quillen (for the injective model structure on $\A$) and its total derived adjunction is naturally isomorphic to the adjoint pair $(\Ll_{\Ss},i)$.
\item ${^\ddagger}\A^p_\Ss=(\W_\Ss\cap\C_\A, \A, \Fc_\Ss)$ defines a projective, cofibrantly generated, abelian model structure on $\A$. Furthermore, the identity $\id:{^\ddagger}\A^p_\Ss\to \A_p^\Ss$ is a left Quillen equivalence, the adjunction $\id:{^\ddagger}\A^p_\Ss\rightleftarrows \A:\id$ is Quillen (for the projective model structure on $\A$) and its total derived adjunction is naturally isomorphic to the adjoint pair $(i,\R_{\Ss})$.
\end{enumerate}
Finally, the identity $\id:{^\ddagger}\A^i_\Ss\to{^\ddagger}\A^p_\Ss$ is a right Quillen equivalence.
\end{prop}

\begin{proof}
(i) By \Cref{trivially_fibrant_cofibrant}, the pair $(\A, \W_\Ss\cap\Fc_\A\cap\C_\Ss)=(\A, \mathrm{Inj}(\A))$ is a small cotorsion pair. Furthermore, by assumption and by \Cref{apply_Becker}(i) we have that both $\W_\Ss$ and $\Fc_\A$ are of the form $\mathcal{G}^\perp$, hence so is their intersection. Thus we only need to show that $\C_\Ss={^\perp}(\W_\Ss\cap\Fc_\A)$ to get a small cotorsion pair. Since every object of $\W_\Ss$ is isomorphic in $\Ho(\A)$ to an object of $\W_\Ss\cap\Fc_\A$ and since those are fibrant in $\A$, we may apply \Cref{Hom_Ho} to deduce that an object lies in ${^\perp}(\W_\Ss\cap\Fc_\A)$ if and only if its image in $\Ho(\A)$ lies in $^\ddagger\Ss=\Ker(\Ll_{\Ss})$ as required.

The functors $\id:{^\ddagger}\A^i_\Ss\to \A_i^\Ss$ and $\id:{^\ddagger}\A^i_\Ss\to \A$ both clearly preserve fibrant and trivially fibrant objects so are right Quillen. In fact, since the trivially fibrant objects in ${^\ddagger}\A^i_\Ss$, $\A_i^\Ss$ and $\A$ are the same one gets, after restricting to cofibrant-fibrant objects, that $\mathbf{R}\id:\Ho({^\ddagger}\A^i_\Ss)\to \Ho(\A_i^\Ss)$ identifies with the identity functor on $\W_\Ss\cap\Fc_\A/\mathrm{Inj(\A)}$ and $\mathbf{R}\id:\Ho({^\ddagger}\A^i_\Ss)\to \Ho(\A)$ identifies with the canonical inclusion $\W_\Ss\cap\Fc_\A/\mathrm{Inj(\A)}\hookrightarrow \Fc_\A/\mathrm{Inj(\A)}$, as required.

(ii) Again, by \Cref{trivially_fibrant_cofibrant} the pair $(\W_\Ss\cap\C_\A\cap\Fc_\Ss, \A)=(\mathrm{Proj}(\A), \A)$ is a small cotorsion pair. Furthermore, by \Cref{apply_Becker}(i) there exists a set $\mathcal{G}\subseteq \W_\Ss$ such that $\mathcal{G}^\perp=\Fc_\A\cap\Fc_\Ss$. Now define a set $\mathcal{G}'\subseteq \W_\Ss\cap\C_\A$ by letting the objects of $\mathcal{G}'$ be a choice of cofibrant replacements in $\A$ for the objects of $\mathcal{G}$. Since these are cofibrant, we may apply \Cref{Hom_Ho} again to deduce that $(\mathcal{G}')^\perp=\Fc_\Ss$. Furthermore, the pair $({^\perp}\Fc_\Ss,\Fc_\Ss)$ is now a small cotorsion pair, such that ${^\perp}\Fc_\Ss$ is the closure of $\mathcal{G}'\cup\{Z\}$ under transfinite compositions and direct summands by \cite[Lemma 6.9]{Hov2}, where $Z$ is a choice of projective generator for $\A$. Since both $\W_\Ss$ and $\C_\A$ are closed under those by Lemma 6.2 in \textit{loc.\ cit.} and \Cref{apply_Becker}(i), we have ${^\perp}\Fc_\Ss\subseteq \W_\Ss\cap\C_\A$. To show equality it therefore suffices to show that an object of $\C_\A$ lies in ${^\perp}\Fc_\Ss$ if and only if it lies in $\W_\Ss$. This is proved via the usual argument using \Cref{Hom_Ho}. The statements on the Quillen functors are proved dually to (i).

Finally, $\id:{^\ddagger}\A^i_\Ss\to{^\ddagger}\A^p_\Ss$ will by definition be a right Quillen equivalence if, for $X\in \W_\Ss\cap\C_\A$ and $Y\in \W_\Ss\cap\Fc_\A$, a map $f:X\to Y$ is a weak equivalence in ${^\ddagger}\A^p_\Ss$ if and only if it is a weak equivalence in ${^\ddagger}\A^i_\Ss$. Note that the trivial cofibrations in ${^\ddagger}\A^p_\Ss$ are precisely the monomorphisms with projective cokernel, and so are still weak equivalences in ${^\ddagger}\A^i_\Ss$. Thus, by factorising $f$ in ${^\ddagger}\A^p_\Ss$ as a trivial cofibration followed by a fibration we may assume that $f$ is a fibration. Then it is a weak equivalence in ${^\ddagger}\A^p_\Ss$ if and only if its kernel lies in $\Fc_\Ss$. But since both $X$ and $Y$ lie in $\W_\Ss$ by assumption and $\W_\Ss$ is thick, this happens if and only if $\ker(f)\in \W_\Ss\cap \Fc_\Ss=\W_\A$, cf. \Cref{trivially_fibrant_cofibrant}. Similarly, $f$ is a weak equivalence in ${^\ddagger}\A^i_\Ss$ if and only if $\ker(f)\in \W_\Ss\cap\C_\Ss=\W_\A$.
\end{proof}

By applying Becker's construction \cite[Propositions 1.4.2 \& 1.4.6]{Beck14} to ${^\ddagger}\A^i_\Ss$ and ${^\ddagger}\A^p_\Ss$, we obtain the last piece of the recollement.

\begin{cor}\label{last_piece}
Keep the same assumptions as in \Cref{other_proj_inj}. Then the triple $(\A_i^\Ss)^\ddagger=(\C_\Ss, \Fc_\A, \W_\Ss)$ defines a hereditary, cofibrantly generated, abelian model structure on $\A$ whose homotopy category is canonically equivalent to $\Ho(\A)/\Ss$, and the adjunction $\id:(\A_i^\Ss)^\ddagger\rightleftarrows \A:\id$ is Quillen (for the injective model structure on $\A$) with total derived adjunction naturally isomorphic to $(\ell_\Ss, \pi_\Ss)$. Dually, the triple $(\A_p^\Ss)^\ddagger=(\C_\A, \Fc_\Ss, \W_\Ss)$ defines a hereditary, cofibrantly generated, abelian model structure on $\A$ whose homotopy category is canonically equivalent to $\Ho(\A)/\Ss$, and the adjunction $\id:\A\rightleftarrows (\A_p^\Ss)^\ddagger:\id$ is Quillen (for the projective model structure on $\A$) with total derived adjunction naturally isomorphic to $(\pi_\Ss, r_\Ss)$.
\end{cor}

Hence, sticking with the projective model structures, \Cref{recollement_model}(ii), \Cref{apply_Becker}(ii), \Cref{other_proj_inj}(ii) and \Cref{last_piece} together say that if we are in \Cref{setup_2} in such a way that both an injective and a projective model structure exist in $\A$ with the same class of trivial objects, then by using only the projective model structure we may construct four new abelian model structures on $\A$ such that the four adjunctions from the recollement \eqref{generic_recollement} are obtained as total derived functors of the various identity adjunctions between our new model structures and the projective model structure on $\A$.

In particular, by considering $\A=\Mod(\He)$ with its Gorenstein projective and injective model structures and by letting $\M=\{M_1, \ldots, M_n\}$ be some fixed collection of standard Levi subgroups of $G$, the above tells us that the recollement in \Cref{many_Levi_rec}(i) may be reconstructed from various model structures on $\Mod(\He)$ by only using the Gorenstein projective model structure. If we define
\begin{align*}
\W^G_\M&:=\left\{\m\in\Mod(\He) \mid \m\in\Ss^G_\M\right\},\\
\C^G_\M:&=\left\{\m\in\Mod(\He) \mid \m\in\Ker (\mathcal{L}^G_\M)\right\}\\
&=\left\{\m\in\Mod(\He) \mid \pd_{\He_{M_i}}(\mathcal{L}^G_{M_i}(\m))<\infty \text{ for all } 1\leq i\leq n\right\},\quad \text{and}\\
\Fc^G_\M:&=\left\{\m\in\Mod(\He) \mid \m\in\Ker (\mathcal{R}^G_\M)\right\}\\
&=\left\{\m\in\Mod(\He) \mid \m\in\bigcap_{1\leq i\leq n}\Ker (\mathbf{R}\mathcal{R}^G_{M_i})\right\},
\end{align*}
where the second equalities for $\C^G_\M$ and $\Fc^G_\M$ both follow from \Cref{many_Levi_rec}(iii), then in particular \Cref{recollement_model}(ii) and \Cref{other_proj_inj}(ii) say in this case that there are two cofibrantly generated, projective abelian model structures ${^\ddagger}\Mod(\He)^G_\M$ and $\Mod(\He)^G_\M$ with cofibrant objects $\C^G_\M\cap \mathrm{GProj}(\He)$ and $\W^G_\M\cap \mathrm{GProj}(\He)$, respectively, and trivial objects $\W^G_\M$ and $\Fc^G_\M$, respectively, so that $\Ho({^\ddagger}\Mod(\He)^G_\M)\simeq \Ho(\He)_\M$ and $\Ho(\Mod(\He)^G_\M)\simeq \Ss^G_\M$. 


Next, fix a standard parabolic $P=MN$ of $G$. Recall that there is an adjunction
$$
F_M:\Mod(\He_M)\rightleftarrows\Rep_k^\infty(M):U_M
$$
where $F_M(\m):=\XX_M\otimes_{\He_M}\m$ and $U_M(V):=\Hom_M(\XX_M,V)\cong V^{I_M}$ for any $\m\in\Mod(\He_M)$ and $V\in\Rep_k^\infty(M)$. When $M=G$, we drop the subscript and write $F$ and $U$ for the corresponding functors. On the smooth representations side, the parabolic induction functor $\Ind_P^G:\Rep_k^\infty(M)\to\Rep_k^\infty(G)$ is defined by first inflating smooth representations of $M\cong P/N$ to $P$ and then applying the smooth induction functor. This functor has both a left and right adjoint (cf.\ \cite[Proposition 4.3]{Vign16}), and its left adjoint can be described explicitly as the Jacquet functor $V\mapsto J_P^G(V):=V_N$ where $V_N$ denotes the $N$-coinvariants of $V$. This construction is compatible with the parabolic induction for $\He$-modules by a result of Ollivier-Vign\'eras.

\begin{thm}[{\cite[Theorem 1]{OllVign18}}]\label{com_diag} Let $M$ be a standard Levi subgroup of $G$. Then we have a natural isomorphism of functors $U\circ \Ind_P^G\cong \Ind_{\He_M}^{\He}\circ U_M$.
\end{thm}

We note in passing that the above immediately implies the analogue of \Cref{ind_tran}(i) on the smooth representations side.

\begin{cor}\label{ind_tran_G} The $I$-Gorenstein projective model structure on $\Rep_k^\infty(M)$ is the right transfer of the $I$-Gorenstein projective model structure on $\Rep_k^\infty(G)$ along $\Ind_P^G$.
\end{cor}

\begin{proof}
We need to show that for a morphism $f$ in $\Rep_k^\infty(M)$, $f$ is a fibration (resp.\ weak equivalence) if and only if $\Ind_P^G(f)$ is a fibration (resp.\ weak equivalence). But these notions are by definition completely characterised by passing to pro-$p$ Iwahori invariants. So the result follows immediately from \Cref{com_diag} and \Cref{ind_tran}(i).
\end{proof}

We finish by right transferring our new model structures on $\Mod(\He)$ to $\Rep_k^\infty(G)$. 

\begin{prop}\label{last_right_transfer} For any collection $\M$ of standard Levi subgroups of $G$, the right transfer of the model structures $\Mod(\He)^G_\M$ and ${^\ddagger}\Mod(\He)^G_\M$ to $\Rep_k^\infty(G)$ along $U$ both exist. 
\end{prop}

\begin{proof}
Since $\id:\Mod(\He)^{\text{GP}}\to \Mod(\He)^G_\M$ and $\id:\Mod(\He)^{\text{GP}}\to {^\ddagger}\Mod(\He)^G_\M$ are right Quillen, they in particular preserve path objects. Hence path objects in $\Mod(\He)^G_\M$ and ${^\ddagger}\Mod(\He)^G_\M$ can be computed exactly as in the Gorenstein projective model structure. Therefore, the arguments of \cite[Proposition 4.10]{KD21} apply mutatis mutandis to yield the result.
\end{proof}

It is unclear to us at this stage whether the two other model structures constructed in \Cref{apply_Becker}(ii) and \Cref{last_piece}, applied to $\Mod(\He)$, can be right transferred to $\Rep_k^\infty(G)$ as well. We note that these model structures are not projective and thus the standard techniques used in \cite{KD21} in order to establish the existence of right transfers now require one to construct a fibrant replacement functor as well as path objects for these right transfers. An in-depth study of this question and of the right transfers obtained in \Cref{last_right_transfer} and their relationship to the derived adjunction $(\mathbf{L}J_P^G,\mathbf{R}\Ind_P^G)$ will be done elsewhere.

\subsection{Bousfield localisation}

We finish this section by explaining how the projective model structures constructed in the previous subsection are Bousfield localisations of the projective model structure on $\A$. Our approach here is very similar to \cite[\S 1.5]{Beck14}, and in fact Proposition 1.5.3 in \textit{loc.\ cit}.\ shows that this is automatically true for the model structures constructed in \Cref{apply_Becker} and \Cref{last_piece}. We therefore focus on the model structures constructed in \Cref{recollement_model}(ii) and \Cref{other_proj_inj}(ii). First we recall the notion of Bousfield localisation from \cite{Hirsch}. We only state the definition in the context of pointed model categories to make the presentation simpler, but this notion can be defined for all model categories.

\begin{defn} Let $\mathscr{M}$ be a pointed model category and $\mathrm{S}$ be a class of morphisms in $\mathscr{M}$.
\begin{itemize}
\item An object $X\in \mathscr{M}$ is called an \emph{$\mathrm{S}$-colocal object} if it is cofibrant and for every morphism $f: A\to B$ in $\mathrm{S}$, the induced map $\Hom_{\Ho(\mathscr{M})}(\Sigma^n X, A)\to \Hom_{\Ho(\mathscr{M})}(\Sigma^n X, B)$ is a bijection for all $n\geq 0$.
\item A morphism $f:A \to B$ in $\mathscr{M}$ is called an \emph{$\mathrm{S}$-colocal equivalence}  if for every $\mathrm{S}$-colocal object $X$, the induced map $\Hom_{\Ho(\mathscr{M})}(\Sigma^n X, A)\to \Hom_{\Ho(\mathscr{M})}(\Sigma^n X, B)$ is a bijection for all $n\geq 0$.
\end{itemize} 
The notions of \emph{$\mathrm{S}$-local object} and \emph{$\mathrm{S}$-local equivalence} are defined dually to the above.
\end{defn}

We note that (co)local objects and equivalences are usually defined in terms of homotopy function complexes in the literature (cf.\ \cite[Definition 3.1.4]{Hirsch}). The above definition is equivalent to that because for pointed model categories the homotopy groups of these function complexes are precisely the $\Hom$ sets we considered above, see \cite[Lemma 6.1.2]{HovBook}. With this notion at hand, we have the following definition of Bousfield localisation.

\begin{defn}[{\cite[Definition 3.3.1]{Hirsch}}] Let $\mathscr{M}$ be a pointed model category and $\mathrm{S}$ be a class of morphisms in $\mathscr{M}$. The \emph{right Bousfield localisation} of $\mathscr{M}$ with respect to $\mathrm{S}$ is, if it exists, the model structure $\mathscr{R}_\mathrm{S}\mathscr{M}$ on the same underlying category as $\mathscr{M}$, with fibrations the same as in $\mathscr{M}$, weak equivalences the $\mathrm{S}$-colocal equivalences, and cofibrations the maps satisfying the left lifting property with respect to the trivial fibrations. The notion of \emph{left Bousfield localisation} of $\mathscr{M}$ with respect to $\mathrm{S}$ is defined dually.
\end{defn}

\begin{lem}\label{colocal_injective} Let $\A$ be a hereditary abelian model category, $X$ an object of $\A$, and $f$ a morphism in $\A$.
\begin{enumerate}
\item In the arrow category of $\Ho(\A)$, $f$ is isomorphic to a fibration between fibrant objects.
\item If $f$ is a fibration between fibrant objects, then $X$ is $\{f\}$-colocal if and only if
$$
\Hom_{\Ho(\A)}(\Sigma^n X, \ker(f))=0
$$
for all $n\geq 0$.
\end{enumerate}
\end{lem}

\begin{proof}
These statements are proved in steps (1)-(3) of the proof of \cite[Proposition 1.5.3]{Beck14} for injective model structures and the proof carries through mutatis mutandis in our more general setting.
\end{proof}

\begin{prop}\label{Bousfield_loc_abelian} Let $\A$ be an abelian category equipped with a projective abelian model structure $(\C, \A, \W)$, and let $\Ss$ be a thick subcategory of $\Ho(\A)$. Let $\W_\Ss$ denote the class of objects of $\A$ whose image in $\Ho(\A)$ lies in $\Ss$.
\begin{enumerate}
\item Suppose that the inclusion $\Ss\subset \Ho(\A)$ has a left adjoint $\Ll_{\Ss}$. Let $\C_\Ss$ denote the class of objects of $\A$ whose image in $\Ho(\A)$ lies in $\Ker(\Ll_{\Ss})$, and assume that $(\C_\Ss\cap\C, \A, \W_\Ss)$ defines an abelian model structure on $\A$. Then this model structure is the right Bousfield localisation of $\A$ with respect to $\mathrm{S}:=\{0\to X \mid X\in\W_\Ss\}$; and
\item Suppose that the inclusion $\Ss\subset \Ho(\A)$ has a right adjoint $\R_{\Ss}$. Let $\Fc_\Ss$ denote the class of objects of $\A$ whose image in $\Ho(\A)$ lies in $\Ker(\R_{\Ss})$, and assume that $(\C\cap\W_\Ss, \A, \Fc_\Ss)$ defines an abelian model structure on $\A$. Then this model structure is the right Bousfield localisation of $\A$ with respect to $\mathrm{S}':=\{0\to X \mid X\in\Fc_\Ss\}$.
\end{enumerate}
\end{prop}

\begin{proof}
We only prove (i), the proof of (ii) being entirely analogous. By the dual of \Cref{Krause_criterion}(i), we have $^\ddagger\Ss=\Ker (\Ll_{\Ss})$ and $\Ss=\Ker (\Ll_{\Ss})^\ddagger$. By definition of $\mathrm{S}$, the class of $\mathrm{S}$-colocal objects is that of those cofibrant objects $X$ whose image in $\Ho(\A)$ lies in $^\ddagger\Ss=\Ker (\Ll_{\Ss})$, i.e.\ the class of $\mathrm{S}$-colocal objects is $\C_\Ss\cap\C$. Then \Cref{colocal_injective} implies that an $\mathrm{S}$-colocal equivalence is a morphism which, in the arrow category of $\Ho(\A)$, is isomorphic to an epimorphism with kernel in $\W_\Ss$.

Now assume that $(\C_\Ss\cap\C, \A, \W_\Ss)$ defines an abelian model structure, which we denote by $\A_\Ss$. Then we must show that the weak equivalences in $\A_\Ss$ are the $\mathrm{S}$-colocal equivalences. For this, note first that $\id:\A\to\A_\Ss$ is right Quillen since $\W\subseteq \W_\Ss$, and we get a functor $\Ho(\id):\Ho(\A)\to \Ho(\A_\Ss)$ by \Cref{preserve_reflect_we}. Hence one gets that an $\mathrm{S}$-colocal equivalence is also isomorphic to an epimorphism with kernel in $\W_\Ss$ in the arrow category of $\Ho(\A_\Ss)$, and hence is a weak equivalence in $\A_\Ss$. For the converse, suppose that $f$ is a weak equivalence in $\A_\Ss$. Note first that $\C_\Ss\cap\C\cap \W_\Ss=\mathrm{Proj}(\A)=\C\cap\W$, so that the trivial cofibrations in $\A$ and in $\A_\Ss$ coincide. By factorising $f$ in $\A_\Ss$ as a composite of a trivial cofibration and a trivial fibration, we therefore see that $f$ is isomorphic in the arrow category of $\Ho(\A)$ to a trivial fibration in $\A_\Ss$, i.e.\ an epimorphism with kernel in $\W_\Ss$, as required.
\end{proof}

\begin{cor} The projective model structures constructed in \Cref{recollement_model} and \Cref{other_proj_inj} are right Bousfield localisations.
\end{cor}

Of course the above has completely dual analogues. In particular, the injective model structures constructed in \Cref{recollement_model} and \Cref{other_proj_inj} are left Bousfield localisations.

\section{Simple supersingular Hecke modules in the homotopy category}

For the remainder of this paper, we assume that $k$ is algebraically closed. Our main goal is to study the isomorphism classes of simple $\He$-modules in $\Ho(\He)$. In light of \Cref{half_Koziol}(ii), we focus our attention on the case of simple supersingular modules.

\subsection{The affine Hecke algebra and simple supersingular $\He$-modules}

We first recall some properties of the affine Hecke algebra $\Hea$, the classification of simple characters over it, and the classification of simple supersingular $\He$-modules.

Recall that we fixed a set $\{\hat{s}\mid s\in S\}\subseteq \widetilde{W}_\aff$ of lifts of the simple affine reflections defined as in \cite[\S 4.8]{OS14}. Then $\Hea$ is generated as a $k$-algebra by $\{T_t\mid t\in T(\mathbb{F}_q)\}$ and $\{T_{\hat{s}}\mid s\in S\}$. The ring structure of $\Hea$ is determined by the braid relations and the quadratic relations in \eqref{quadratic_rel}. In addition, each irreducible component $\Phi'$ of $\Phi$ gives rise to a subset $S'\subseteq S$ and a subalgebra of $\Hea$ generated by $k[T(\mathbb{F}_q)]$ and $\{T_{\hat{s}}\mid s\in S'\}$. We will refer to these subalgebras as the irreducible components of $\Hea$. We will also need a notion of irreducible components of the finite Hecke algebra $\He_F$ associated to a face $F\subseteq \overline{C}$. First, since $S$ is in bijection with the nodes of the affine Dynkin diagram of $G$, we may associate to the set $S_F$ the subgraph of the affine Dynkin diagram spanned by the corresponding nodes. Each connected component of that subgraph corresponds under the aforementioned bijection to a subset $S_{F'}$ of $S_F$, itself corresponding to a face $F'$ such that $F\subseteq \overline{F'}$ by \eqref{bijection_faces}. We call the associated subalgebras $\He_{F'}$ of $\He_F$ arising in this way the irreducible components of $\He_F$. 

Next, we recall the Ollivier-Schneider resolutions for $\He$-modules.  For each $0\leq i\leq d=r_{ss}$, we write $\mathscr{F}_i$ for the set of $i$-dimensional faces $F$ of $\mathscr{X}$ which are contained in $\overline{C}$. Note that by \cite[Lemma 1.2]{Koh} this is a set of representatives for the $G_\aff$-orbits of the $i$-dimensional faces in $\mathscr{X}$. We also fix for each $i$ a set $\mathscr{F}_{(i)}\subseteq \mathscr{F}_i$ of representatives for the $G$-orbits of the $i$-dimensional faces in $\mathscr{X}$. Then \cite[Theorem 3.12]{OS14} implies that for every $\m\in\Mod(\He)$ there is an exact sequence of $\He$-modules of the form
\begin{equation}\label{res_Oll_Sch}
0\longrightarrow \bigoplus_{F\in\mathscr{F}_{(d)}} \Res^\He_{\He^\dagger_F}(\m)\otimes_{\He^\dagger_F}(\epsilon_F)\He\longrightarrow\cdots\longrightarrow\bigoplus_{F\in\mathscr{F}_{(0)}} \Res^\He_{\He^\dagger_F}(\m)\otimes_{\He^\dagger_F}(\epsilon_F)\He\longrightarrow \m\longrightarrow 0,
\end{equation}
where we note that our Hecke algebras are opposite to the ones in \textit{loc.\ cit.} and hence our tensors products are reversed. Here, for a face $F\subseteq \overline{C}$, the twist $(\epsilon_F)\He$ by the orientation character $\epsilon_F$ is the one defined in \textit{loc.\ cit.}, \S 3.3.1. We now see that there is a similar resolution for $\Hea$-modules, cf.\ \cite[Remark 4.1]{Koz}.
\begin{prop}
There is a resolution of $(\Hea, \Hea)$-bimodules of the form
\begin{equation}\label{res_Haff2}
0\longrightarrow \bigoplus_{F\in\mathscr{F}_d} \Hea\otimes_{\He_F}\Hea\longrightarrow\cdots\longrightarrow\bigoplus_{F\in\mathscr{F}_0} \Hea\otimes_{\He_F}\Hea\longrightarrow \Hea\longrightarrow 0.
\end{equation}
This gives a free resolution of $\Hea$ as a left and as a right $\Hea$-module.
\end{prop}

\begin{proof}
The proof is mostly analogous to that of \cite[Theorem 3.12]{OS14}, we only sketch how it goes through for $\Hea$ instead of $\He$. First, recall that to any $V\in \Rep_k^\infty(G)$, we may functorially associate a certain $G$-equivariant coefficient system $\uu{V}$ as in \cite[Example 2.2]{Koh}. We refer the reader to \S 2.1 of \textit{loc.\ cit.}\ for the relevant definitions, we only point out that this construction also applies, with the same definitions, to smooth representations of $G_\aff$ and that there is a functorial oriented chain complex $\C^{or}_c(\mathscr{X}_{(\bullet)}, \uu{V})\to V$ for $V$ a smooth representation of one of $G$, $G_\aff$ (cf.\ equations (2.1)-(2.2) of \textit{loc.\ cit.}). If $\XX$ and $\XX_\aff$ are the representations from \S2.2, then note that extending functions by zero gives a $G_\aff$-equivariant split embedding $\XX_\aff\hookrightarrow\XX$ (with splitting given by restricting functions to $G_\aff$). Thus we see that the complex $0\to \C^{or}_c(\mathscr{X}_{(\bullet)},\uu{\XX_\aff})\to \XX_\aff\to 0$ is a $G_\aff$-equivariant retract of $0\to \C^{or}_c(\mathscr{X}_{(\bullet)},\uu{\XX})\to \XX\to 0$. Since the latter is exact on $I$-invariants by \cite[Theorem 3.4]{OS14} it follows that the same holds for the former. We are therefore left to show that the $\Hea$-bimodule $\C^{or}_c(\mathscr{X}_{(i)},\uu{\XX_\aff})^I$ is as given in the statement for each $0\leq i\leq d$.

This is done by arguing exactly as in \cite{OS14}. The proof that
$$
\C^{or}_c(\mathscr{X}_{(i)},\uu{\XX_\aff})\cong\bigoplus_{F\in\mathscr{F}_i}\ind_{P_F}^{G_\aff}(\XX_\aff^{I_F})
$$
is done exactly as in \S 3.3.2 of \textit{loc.\ cit.} We only point out that the orientation character $\epsilon_F$ doesn't appear here because each parahoric $P_F$ fixes $F$ pointwise and thus also fixes any chosen orientation on $F$, and that $\mathscr{F}_i$ is a set of representatives for the $G_\aff$-orbits of the $i$-dimensional faces in $\mathscr{X}$. The proof of Proposition 4.25 in \textit{loc.\ cit.\ }then adapts to give an isomorphism $\Hea \otimes_{\He_F}\XX_F\cong \XX_\aff^{I_F}$ of left $k[P_F]\otimes_k\Hea$-modules. Therefore, one obtains
$$
\ind_{P_F}^{G_\aff}(\XX_\aff^{I_F})^I\cong\ind_{P_F}^{G_\aff}(\Hea \otimes_{\He_F}\XX_F)^I\cong\Hea\otimes_{\He_F}\Hea
$$
exactly as in Proposition 3.10 of \textit{loc.\ cit}.
\end{proof}

The methods of \cite[\S 1]{OS14} then show that $\Hea$ is $n$-Gorenstein for some $n\leq d$. Thus the category $\Mod(\Hea)$ has a Gorenstein projective model structure and the homotopy category $\Ho(\Hea)$ is well-defined. Furthermore, we obtain by tensoring with \eqref{res_Haff2} an exact resolution
\begin{equation}\label{res_Haff3}
0\longrightarrow \bigoplus_{F\in\mathscr{F}_d} \Res^\Hea_{\He_F}(\m)\otimes_{\He_F}\Hea\longrightarrow\cdots\longrightarrow\bigoplus_{F\in\mathscr{F}_0} \Res^\Hea_{\He_F}(\m)\otimes_{\He_F}\Hea\longrightarrow \m\longrightarrow 0
\end{equation}
for any $\m\in\Mod(\Hea)$.

Recall from \cite[Proposition 2.2]{Vign3} that the characters of $\Hea$ are parametrized by pairs $(J, \xi)$ where $\xi$ is a character of $k[T(\mathbb{F}_q)]$ and $J\subseteq S_\xi$ where $S_\xi$ is as in \eqref{S_xi}. The character $\chi$ corresponding to such a pair is defined by
\begin{align*}
\chi(T_t)&=\xi(t)\quad\text{for all $t\in T(\mathbb{F}_q)$}\\
\chi(T_{\hat{s}})&=\begin{cases}
0 & \text{if $s\notin J$}\\
-1 & \text{if $s\in J$}
\end{cases}.
\end{align*}
We will sometimes write $\chi=(J, \xi)$ to mean that $\chi$ is defined as above. When $S_\xi=S$, the character $\chi$ will be called a \emph{twisted trivial character}, resp.\ \emph{twisted sign character}, if it additionally satisfies $J=\emptyset$, resp.\ $J=S$. There is a notion of supersingularity for characters of $\Hea$, and we have that a character $\chi$ is supersingular if and only if its restriction to each irreducible component of $\Hea$ does not equal the restriction of a twisted sign or twisted trivial character (cf.\ Theorem 6.15 in \textit{loc.\ cit}). We will use this result as a definition of supersingularity for characters of $\Hea$. By convention, if $G=T$ then every character as above is supersingular.

Next, the group $\widetilde{\Omega}$ acts on the characters of $\Hea$ as follows. Given a character $\chi$ of $\Hea$ and $\tilde{\omega}\in\widetilde{\Omega}$, we let $\chi^{\tilde{\omega}}(h):=\chi(T_{\tilde{\omega}}hT_{\tilde{\omega}^{-1}})$ for $h\in\Hea$. We denote by $\widetilde{\Omega}_\chi$ the stabiliser of $\chi$ under this action. This is a subgroup of finite index in $\widetilde{\Omega}$, containing $T(\mathbb{F}_q)$. We also let $\He_\chi$ denote the subalgebra of $\He$ generated by $\Hea$ and $\{T_{\tilde{\omega}}\mid \tilde{\omega}\in \widetilde{\Omega}_\chi\}$. It follows from the braid relations that $\He_\chi$ is free over $\Hea$ and that $\He$ is free over $\He_\chi$, both as left and right modules. For future use, we record here the following:

\begin{lem}\label{Hchi_Gor} For any character $\chi$ of $\Hea$, the algebra $\He_\chi$ is Gorenstein.
\end{lem}

\begin{proof}
Since $\He$ is free over $\He_\chi$, the functor $\He\otimes_{\He_\chi}(-)$ is exact and preserves projectives. Since $\He$ is $(r_{ss}+r_Z)$-Gorenstein it follows that for each $\mathfrak{n}\in\Mod(\He_\chi)$ and for any $i>r_{ss}+r_Z$ we have
$$
\Ext^i_{\He_\chi}(\mathfrak{n}, \Res^\He_{\He_\chi}(\He))\cong \Ext^i_\He(\mathfrak{n}\otimes_{\He_\chi} \He, \He)=0.
$$
Thus the injective dimension of $\Res^\He_{\He_\chi}(\He)$ is bounded above by $r_{ss}+r_Z$. Since $\He_\chi$ is a direct summand of $\Res^\He_{\He_\chi}(\He)$, it follows that it has finite selfinjective dimension as a right module over itself. An analogous argument shows the corresponding statement for left modules. 

It remains to show that $\He_\chi$ is (left and right) Noetherian. Write $\Omega_\chi$ for the image of $\widetilde{\Omega}_\chi$ in $\Omega$. This group decomposes as $\Omega_\chi=\Omega_\chi^{\text{free}}\times \Omega_\chi^{\text{tor}}$ where $\Omega_\chi^{\text{free}}$ is a free abelian group on some generators $\omega_1, \ldots, \omega_{r_Z}$ and $\Omega_\chi^{\text{tor}}$ is finite. If we let $\He_\chi^{\text{free}}$ denote the subalgebra of $\He_\chi$ generated by $\Hea$ and all $T_{\hat{\omega}_i^{\pm 1}}$, then the braid relations show that $\He_\chi^{\text{free}}$ is an iterated skew Laurent polynomial algebra over $\Hea$, and hence is Noetherian because $\Hea$ is Noetherian (cf.\ \cite[Theorem 5.1]{Vign3} and \cite[Theorem 1.4.5]{MCR01}). Since $\He_\chi$ is finite over $\He_\chi^{\text{free}}$, it must be Noetherian as well.
\end{proof}

We now suppose that $\chi=(J, \xi)$ is a supersingular character of $\Hea$. Given a finite dimensional, irreducible (right) representation $(V,\rho)$ of $\widetilde{\Omega}_\chi$ on which $T(\mathbb{F}_q)$ acts by the character $\xi$, we let $\chi\otimes V$ be the $\He_\chi$-module whose underlying vector space is $V$ and on which $\Hea$ acts by $\chi$ and $\widetilde{\Omega}_\chi$ acts via $\rho$. Then every simple supersingular $\He$-module is isomorphic to one of the form $(\chi\otimes V)\otimes_{\He_\chi}\He$, and furthermore two pairs $(\chi, V)$ and $(\chi', V')$ give rise to isomorphic simple supersingular $\He$-modules if and only if the two pairs are $\widetilde{\Omega}$-conjugate (cf.\ \cite[Theorem 6.18]{Vign3}). Here we say that $(\chi, V)$ and $(\chi', V')$ are $\widetilde{\Omega}$-conjugate if $\chi'=\chi^{\tilde{\omega}}$ for some $\tilde{\omega}\in\widetilde{\Omega}$ and $V'\cong V^{\tilde{\omega}}$, where $V^{\tilde{\omega}}$ is the representation of $\widetilde{\Omega}_{\chi^{\tilde{\omega}}}=\tilde{\omega}^{-1}\widetilde{\Omega}_\chi\tilde{\omega}$ with underlying vector space $V$ obtained by letting $v \cdot\tilde{\omega}^{-1}\tilde{\omega}'\tilde{\omega}:=v\cdot\tilde{\omega}'$ for all $v\in V$ and $\tilde{\omega}'\in\widetilde{\Omega}_\chi$.

\begin{rem}\label{normal_subgroups} Suppose that $\tilde{\omega}, \tilde{\omega}'\in \widetilde{\Omega}$. Since $\Omega$ is abelian, we have that $\tilde{\omega}\tilde{\omega}'\tilde{\omega}^{-1}=t\tilde{\omega}'$ for some $t\in T(\mathbb{F}_q)$. In particular, we deduce that every subgroup of $\widetilde{\Omega}$ which contains $T(\mathbb{F}_q)$ is automatically normal. This applies in particular to the stabiliser subgroups $\widetilde{\Omega}_\chi$ as above, so that we have $\widetilde{\Omega}_\chi=\widetilde{\Omega}_{\chi^{\tilde{\omega}}}$ and $\He_\chi=\He_{\chi^{\tilde{\omega}}}$ for all $\tilde{\omega}\in\widetilde{\Omega}$.
\end{rem}

We next recall a theorem of Koziol, which describes when simple supersingular modules are trivial in the Gorenstein projective model structure. The statement in \textit{loc.\ cit.} doesn't actually contain (ii) as below, but their proof actually shows the statement as written here.

\begin{thm}[{\cite[Theorem 7.7]{Koz}}]\label{Koziol_ss_thm} Suppose that $\m:=(\chi\otimes V)\otimes_{\He_\chi}\He$ is a simple supersingular $\He$-module as above, and let $\chi$ be parametrized by $(J, \xi)$. Then the following are equivalent:
\begin{enumerate}
\item $\pd_{\He}(\m)<\infty$;
\item $\pd_{\Hea}(\chi)<\infty$; and
\item $\Phi$ is of type $A_1\times\cdots \times A_1$ (possibly empty product) and $S_\xi=S$.
\end{enumerate}
\end{thm}

We now describe $\Res^\He_{\He_\chi}(\m)$, where $\m=(\chi\otimes V)\otimes_{\He_\chi}\He$ is a simple supersingular $\He$-module. Observe first that for any $\mathfrak{n}\in\Mod(\He_\chi)$ we obtain from the braid relations a direct sum decomposition
$$
\mathfrak{n}\otimes_{\He_\chi}\He= \bigoplus_{\tilde{\omega}\in\widetilde{\Omega}_\chi\backslash \widetilde{\Omega}} \mathfrak{n}\otimes T_{\tilde{\omega}}.
$$
This is in fact a decomposition as an $\He_\chi$-module. Indeed, for all $h\in\Hea$ and all $x\in \mathfrak{n}$ we have $(x\otimes T_{\tilde{\omega}})\cdot h=x(T_{\tilde{\omega}}hT_{\tilde{\omega}}^{-1})\otimes T_{\tilde{\omega}}$. Furthermore, the group $\widetilde{\Omega}_\chi=\widetilde{\Omega}_{\chi^{\tilde{\omega}}}=\tilde{\omega}^{-1}\widetilde{\Omega}_\chi\tilde{\omega}$ naturally acts on $\mathfrak{n}\otimes T_{\tilde{\omega}}$ (cf.\ \Cref{normal_subgroups}). Note that these actions do not depend, up to isomorphism, on the choice of coset representatives. When $\mathfrak{n}=\chi\otimes V$, the summand $\mathfrak{n}\otimes T_{\tilde{\omega}}$ naturally identifies with $\chi^{\tilde{\omega}}\otimes (V')^{\tilde{\omega}}$ as an $\He_\chi$-module. We therefore get a decomposition
\begin{equation}\label{n_decomp}
\Res^\He_{\He_\chi}(\m)\cong \bigoplus_{\tilde{\omega}\in\widetilde{\Omega}_\chi\backslash \widetilde{\Omega}} \chi^{\tilde{\omega}}\otimes V^{\tilde{\omega}}.
\end{equation}
We finish by recalling that to any character $\xi:T(\mathbb{F}_q)\to k^\times$ we may attach the idempotent
\begin{equation}\label{central_idem}
e_\xi=\abs{T(\mathbb{F}_q)}^{-1}\sum_{t\in T(\mathbb{F}_q)}\xi(t)T_{t^{-1}}\in k[T(\mathbb{F}_q)].
\end{equation}
Then we get that $e_\xi T_{\hat{s}}=T_{\hat{s}}e_{s\cdot \xi}$ for all $s\in S$, where $s\cdot \xi$ denotes the lift to $W$ of the $W_0$-action on characters of $T(\mathbb{F}_q)$ by conjugation of the argument. If we let $S_\xi$ be as in \eqref{S_xi}, then we have $s_\alpha\cdot \xi=\xi$ for all $s_\alpha\in S_\xi$ by \cite[Lemma 2.9]{OS19}, and it follows straightforwardly from \eqref{quadratic_rel} that
\begin{equation}\label{quad_exi}
(e_\xi T_{\hat{s}})^2=\begin{cases}-e_\xi T_{\hat{s}} & \text{if $s\in S_\xi$}\\
0 & \text{otherwise}
\end{cases}.
\end{equation}
For later use we record the following immediate consequence of the above:

\begin{lem}\label{exi_commute}
Suppose that $\xi:T(\mathbb{F}_q)\to k^\times$ is a character and let $F\subseteq \overline{C}$ be a face such that $S_F\subseteq S_\xi$. Then $e_\xi$ is central in $\He_F$ and $(e_\xi T_{\hat{s}})^2=-e_\xi T_{\hat{s}}$ for all $s\in S_F$.
\end{lem}

\subsection{Supersingular characters in $\Ho(\Hea)$}  Our first aim is to understand morphisms between supersingular characters in the homotopy category $\Ho(\Hea)$.  Fix a face $F\subseteq \overline{C}$. Recall that $\He_F$ is a Frobenius algebra and in particular selfinjective (cf.\ \cite[Example 4.1(ii)]{KD21}), and thus the Gorenstein projective model structure on $\Mod(\He_F)$ has all objects cofibrant and fibrant. In particular, we see from \Cref{factor_through} that
\begin{equation}\label{Hom_HoF}
[\m, \mathfrak{n}]_{\He_F}=\Hom_{\Ho(\He_F)}(\m, \mathfrak{n})\cong \Hom_{\He_F}(\m, \mathfrak{n})/{\sim}
\end{equation}
where two morphisms $f$, $g$ are equivalent under $\sim$ if and only if $f-g$ factors through a projective $\He_F$-module. We collect the facts we need about the relations between the Gorenstein projective model structures on the categories $\Mod(\He_F)$, $F\subseteq\overline{C}$, and $\Mod(\Hea)$ below:

\begin{lem}\label{facts_H_F}\begin{enumerate}
\item Let $F\subseteq \overline{C}$. If $\m, \mathfrak{n}\in\Mod(\He_F)$ are simple then
$$
[\m, \mathfrak{n}]_{\He_F}\cong\begin{cases} 0 & \text{if $\m$ is projective}\\ 
\Hom_{\He_F}(\m, \mathfrak{n}) & \text{otherwise}.
\end{cases}
$$
\item The functors $\Res^{\Hea}_{\He_F}$ and $\Res^{\He_{F'}}_{\He_F}$ (for $F'\subseteq \overline{F}\subseteq\overline{C}$) are right Quillen. They and their left adjoints preserve weak equivalences.
\item The `diagonal' functor
$$
\Delta:=(\Res^{\Hea}_{\He_F})_{F\subseteq \overline{C}}: \Mod(\Hea)\to \prod_{F\subseteq \overline{C}}\Mod(\He_F)
$$
is right Quillen. Moreover, the Gorenstein projective model structure on $\Mod(\Hea)$ is the right transfer of the product model structure on $\prod_{F\subseteq \overline{C}}\Mod(\He_F)$ (i.e.\ with all classes of morphisms defined termwise) along $\Delta$. In particular, $\pd_{\Hea}(\m)<\infty$ if and only if $\Res^{\Hea}_{\He_F}(\m)$ is projective for all $F\subseteq \overline{C}$.
\end{enumerate}
\end{lem}

\begin{proof}
For (i), we have that $[\m, \mathfrak{n}]_{\He_F}=0=\Hom_{\He_F}(\m, \mathfrak{n})$ if the simple modules $\m$ and $\mathfrak{n}$ are not isomorphic (cf.\ \eqref{Hom_HoF}). If on the other hand $\m\cong \mathfrak{n}$, then by Schur's lemma we have that $\Hom_{\He_F}(\m,\m)\cong k$ is one-dimensional and hence either its quotient $[\m, \m]_{\He_F}$ is zero or the homotopy relation is trivial. But if $[\m, \m]_{\He_F}=0$ then $\id:\m\to\m$ factors through a projective, from which we deduce that $\m$ is projective. We thus get our desired isomorphism.

Part (ii) is immediate from \Cref{free_adjunction}. For (iii), we obtain as a consequence of (ii) that $\Delta$ preserves fibrations and weak equivalences as those are defined termwise in the target category. Since $\Res^{\Hea}_{\He_F}$ is also faithful for each $F\subseteq\overline{C}$ and all objects are fibrant, we in fact get that $\Delta$ reflects fibrations. We are therefore left to show that $\Delta$ has a left adjoint and reflects weak equivalences. For the former, the left adjoint is given by $(\m_F)_{F\subseteq\overline{C}}\mapsto \bigoplus_{F\subseteq\overline{C}}\m_F\otimes_{\He_F}\Hea$. For the latter, by \Cref{preserve_reflect_we} it suffices to show that $\Delta$ reflects trivial objects, i.e.\ that if $\m\in\Mod(\Hea)$ is such that $\Res^{\Hea}_{\He_F}(\m)$ is projective for all $F\subseteq\overline{C}$ then $\pd_{\Hea}(\m)<\infty$. But this is the reverse implication shown in \cite[Lemma 4.2]{Koz}.
\end{proof}

\Cref{facts_H_F}(i) has the following immediate consequence which will be helpful later.

\begin{cor}\label{basic_htpy_lemma}
If $\m, \mathfrak{n}\in\Mod(\He_F)$ are semisimple such that each simple summand of $\m$ is non-projective, then $[\m, \mathfrak{n}]_{\He_F}=\Hom_{\He_F}(\m, \mathfrak{n})$, i.e.\ the homotopy relation on $\Hom_{\He_F}(\m, \mathfrak{n})$ is trivial.
\end{cor}

A priori, in order to determine the morphism sets between two supersingular characters $\chi\neq\chi'$ in $\Ho(\Hea)$ we would be required to compute some explicit Gorenstein projective/injective replacements of these characters. However, there are no known techniques to do this that we are aware of. We will instead compute the image of these morphisms under the functor $\Ho(\Delta)$. Our aim is to show the following:

\begin{prop}\label{characters_Haff_Ho} Let $\chi\neq\chi'$ be two supersingular characters of $\Hea$ of infinite projective dimension, given by the pairs $(J,\xi)$ and $(J', \xi')$ respectively. Let $\Phi=\sqcup_{i=1}^r \Phi_i$ be the decomposition of the root system $\Phi$ into irreducible components, ordered such that $\mathrm{rk}(\Phi_1)\geq \cdots\geq \mathrm{rk}(\Phi_r)$.
\begin{enumerate}
\item Suppose that $\Phi_1$ is of rank two and $\Phi_i$ is of rank 1 for all $i\geq 2$, $\xi=\xi'$, $S=S_{\xi}$ and $J\cap S_i=J'\cap S_i$ for all $i\geq 2$. Writing $S_1=\{s, s', s''\}$, assume further that, up to permuting the elements of $S_1$ and $\chi$ and $\chi'$, we have
\begin{align*}
(\chi(T_{\hat{s}}), \chi(T_{\hat{s'}}), \chi(T_{\hat{s''}}))&=(-1, -1, 0), \\
(\chi'(T_{\hat{s}}), \chi'(T_{\hat{s'}}), \chi'(T_{\hat{s''}}))&=(0, -1, 0)
\end{align*}
and $s', s''$ are adjacent in the affine Dynkin diagram of $\Phi$. Then $\Ho(\Delta)([\chi, \chi']_\Hea)$ has dimension 1 over $k$ and doesn't contain any isomorphisms.
\item In every other case $\Ho(\Delta)([\chi, \chi']_\Hea)=0$.
\end{enumerate}
\end{prop}

\begin{rem}\label{char_not_conj} Later we will consider simple supersingular modules for the whole pro-$p$ Iwahori-Hecke algebra $\He$, for which the action of $\widetilde{\Omega}$ on supersingular characters comes into play. Suppose that $\omega\in\Omega$ and fix a lift $\hat{\omega}\in\widetilde{\Omega}$. Recalling that the group $\Omega$ normalises $S$, if $\chi$ is any character of $\Hea$ corresponding to a pair $(J, \xi)$ then $\chi^{\hat{\omega}}$ corresponds to the pair $(J^{\hat{\omega}}, \xi^{\hat{\omega}})$ where $J^{\hat{\omega}}=\omega^{-1} J \omega$ and $\xi^{\hat{\omega}}(t)=\xi(\hat{\omega}t\hat{\omega}^{-1})$ ($t\in T(\mathbb{F}_q)$). It follows in particular that $\abs{J^{\hat{\omega}}} =\abs{J}$. We therefore see that the two supersingular characters of $\Hea$ from \Cref{characters_Haff_Ho}(i) are not conjugate under the $\widetilde{\Omega}$-action. 
\end{rem}

Before proving the Proposition, we first highlight that it immediately implies the following:

\begin{cor}\label{classification_affine} Suppose that $\chi\neq\chi'$ are two supersingular characters of $\Hea$ of infinite projective dimension. Then they are not isomorphic in $\Ho(\Hea)$. In particular, if $\mathbb{G}$ is semisimple and simply-connected then two supersingular $\He$-modules are isomorphic in $\Ho(\He)$ if and only if they are isomorphic in $\Mod(\He)$.
\end{cor}

\begin{proof}
Any isomorphism $f:\chi\to\chi'$ in $\Ho(\Hea)$ would induce an isomorphism $\Ho(\Delta(f))$ in $\prod_{F\subseteq\overline{C}}\Ho(\He_F)$. But it follows immediately from \Cref{characters_Haff_Ho} that this cannot occur. The last part follows immediately because $\He=\Hea$ if $\mathbb{G}$ is semisimple and simply-connected.
\end{proof}

We start preparing for the proof. Suppose that $\chi$ and $\chi'$ are as given in \Cref{characters_Haff_Ho}. Given any $f\in[\chi,\chi']_\Hea$ we obtain a collection $(f_F)\in\prod_{F\subseteq\overline{C}} \Hom_{\He_F}(\chi, \chi')$ by letting each $f_F$ be a choice of representative of $\Ho(\Delta)(f)_F=\Ho(\Res^{\Hea}_{\He_F})(f)$, and this satisfies the property that $\Res^{\He_{F'}}_{\He_F}(f_{F'})\sim f_{F}$ whenever $F'\subseteq\overline{F}$. Note that either $f_F\sim 0$ or otherwise the choice of $f_F$ must be unique, cf.\ \Cref{facts_H_F}(i). Our strategy will be to show that any such collection of morphisms must in fact be zero, up to homotopy, unless we are in the specific circumstances of part (i) of the Proposition. We therefore first need a method for determining whether some $f_F$ arising as above is homotopic to zero. To that end, we have the following elementary but key result. 

\begin{lem}\label{presheaf_morphism} Let $\chi\neq\chi'$ be two characters of $\Hea$, and suppose that we are given $(f_F)\in\prod_{F\subseteq\overline{C}} \Hom_{\He_F}(\chi, \chi')$ such that $\Res^{\He_{F'}}_{\He_F}(f_{F'})\sim f_{F}$ whenever $F'\subseteq\overline{F}$.
\begin{enumerate}
\item If $F'\subseteq\overline{F}$ and $f_{F'}\sim 0$ then $f_F\sim 0$.
\item If $F'\subseteq\overline{F}$ and $f_F\sim 0$ then either $f_{F'}\sim 0$ or $\Res^{\Hea}_{\He_F}(\chi)$ is projective.
\item Assume we have three faces $F_1, F_2, F_3$ with $F_i\subseteq \overline{F_3}$ for $i=1,2$ such that $\Res^{\Hea}_{\He_{F_3}}(\chi)$ is not projective. Then $f_{F_1}\sim 0$ if and only if $f_{F_2}\sim 0$.
\end{enumerate}
\end{lem}

\begin{proof}
Part (i) is immediate from \Cref{facts_H_F}(ii) because right Quillen functors preserve right homotopies. For part (ii), assume that $f_F\sim 0$ and $f_{F'}\not\sim 0$. Then by \Cref{facts_H_F}(i) we must have $\Res^{\Hea}_{\He_{F'}}(\chi)=\Res^{\Hea}_{\He_{F'}}(\chi')$ and $f_{F'}=\lambda\id_\chi$ for some $\lambda\in k^\times$. Then we have $\Res^{\He_{F'}}_{\He_F}(\lambda\id_\chi)\sim f_F\sim 0$ from which we deduce that $\Res^{\Hea}_{\He_F}(\chi)$ is projective. For (iii), the condition that $f_{F_1}\sim 0$ implies by (i) that $f_{F_3}\sim 0$. Since $\Res^{\Hea}_{\He_{F_3}}(\chi)$ is not projective it then follows from (ii) that $f_{F_2}\sim 0$ as required. We get the converse by swapping $F_1$ and $F_2$.
\end{proof}

In order to apply \Cref{presheaf_morphism}, we need some criterion to establish when the restriction of a character of $\Hea$ to some $\He_F$ is not projective. Note that projectivity always holds when $F=C$ since $\He_C=k[T(\mathbb{F}_q)]$ is semisimple. In general, we have the next result below which we gathered from the arguments of Koziol in \cite[\S 7]{Koz}. To prepare for its proof, we recall that the \emph{$0$-Hecke algebra} over $k$ associated to a finite Coxeter group $W_{\text{gen}}$ with set of Coxeter generators $S_{\text{gen}}$ is the $k$-algebra with generating set $\{H_s\mid s\in S_{\text{gen}}\}$ such that the braid relations hold and $H_s^2=-H_s$ for all $s\in S_{\text{gen}}$ (cf.\ \cite[Definition 1.1]{Nor79}). The simple modules for such algebras are all characters and are indexed by the subsets $S_{\text{gen}}$, where for $L\subseteq S_{\text{gen}}$ the associated character $\lambda_L$ is defined by $\lambda_L(H_s)=-1$ if $s\in L$ and $\lambda_L(H_s)=0$ otherwise (cf.\ \cite[\S 3]{Nor79}). When $L=\emptyset$, resp.\ $L=S_{\text{gen}}$, we call $\lambda_L$ the trivial, resp.\ sign, character.

\begin{lem}\label{Koziol_projective_face} Let $F\subseteq \overline{C}$ be a face, and let $\chi=(J, \xi)$ be a character of $\Hea$.
\begin{enumerate}
\item If $S_F\nsubseteq S_\xi$ then $\Res^\Hea_{\He_F}(\chi)$ is not projective.
\item Assuming instead that $S_F\subseteq S_\xi$, then the following are equivalent:
\begin{enumerate}
\item $\Res^\Hea_{\He_F}(\chi)$ is not projective;
\item there exists an irreducible component $\He_{F'}$ of $\He_F$ such that $\Res^\Hea_{\He_{F'}}(\chi)$ is not equal to a twisted sign or twisted trivial character; and
\item there exist $s, s'\in S_F$ corresponding to adjacent nodes in the affine Dynkin diagram such that $\chi(T_{\hat{s}})\neq \chi(T_{\hat{s'}})$.
\end{enumerate}
\end{enumerate}
\end{lem}

\begin{proof}
To prove non-projectivity, it suffices to find $F\subseteq\overline{F''}\subseteq\overline{C}$ such that $\Res^\Hea_{\He_{F''}}(\chi)$ is not projective, because $\He_F$ is free over $\He_{F''}$. For (i), we pick $F''$ to be given by $S_{F''}=\{s\}$, where $s\in S_F\setminus S_\xi$. That $\Res^\Hea_{\He_{F''}}(\chi)$ is not projective is then the content of the proof of \cite[Lemma 7.5]{Koz}. For (ii), we consider the idempotent $e_\xi$ (cf.\ \eqref{central_idem}). First note that $e_\xi$ is central in $\He_{F}$ and that the algebra $e_\xi\He_F$ is the $0$-Hecke algebra associated to the finite subgroup $W_F\leq W$ generated by $S_F$, cf.\ \Cref{exi_commute}. Also, from $\He_F=e_\xi\He_F\times (1-e_\xi)\He_F$ we see that $\Res^\Hea_{\He_{F}}(\chi)$ factors through the algebra $e_\xi\He_F$. Hence, to show that $\Res^\Hea_{\He_F}(\chi)$ is projective is equivalent to showing that $\Res^\Hea_{e_\xi\He_F}(\chi)$ is projective. By the braid relations, $e_\xi\He_F$ is the tensor product of the $0$-Hecke algebras $e_\xi \He_{F'}$ corresponding to the irreducible components of $\He_F$, and $\Res^\Hea_{e_\xi\He_F}(\chi)$ decomposes similarly as the tensor product of its restrictions to each component.

Now, suppose $\mathcal{C}$ is the 0-Hecke algebra of an irreducible Coxeter system $(W_{\text{gen}}, S_{\text{gen}})$. For a character $\lambda_L$ as above, we denote by $P_L$ its indecomposable projective cover. It is shown in \cite[Theorem 5.1 \& Theorem 5.2]{Nor79} that the trivial and sign characters of $\mathcal{C}$ are projective, and that for any $L\neq \emptyset$, $S_{\text{gen}}$ there is $L'\neq L$ such that $\lambda_L$ is a composition factor of $P_{L'}$. In particular, $\lambda_L$ is never projective for $L\neq \emptyset$, $S_{\text{gen}}$. Going back to our situation, if the restriction of $\chi$ to each irreducible component of $\He_F$ is a twisted sign or trivial character it then follows that $\Res^\Hea_{e_\xi\He_{F}}(\chi)$ is projective, since it is a tensor product of projectives. This shows (a)$\Rightarrow$(b). Moreover, (b)$\Rightarrow$(c) is clear by definition.

Finally, suppose that there exist $s, s'\in S_F$ adjacent in the affine Dynkin diagram with $\chi(T_{\hat{s}})\neq \chi(T_{\hat{s'}})$. Then define $F''$ by $S_{F''}=\{s, s'\}$, and once again note that $\Res^\Hea_{\He_{F''}}(\chi)$ factors through the algebra $e_\xi\He_{F''}$. The latter is isomorphic to a 0-Hecke algebra of type $A_2,$ $B_2$ or $G_2$, and hence it follows from \textit{loc.\ cit.}\ that $\Res^\Hea_{e_\xi\He_{F''}}(\chi)$ is not projective, since it is neither a sign nor a trivial character by assumption. This gives (c)$\Rightarrow$(a).
\end{proof}

Suppose that we have a collection of morphisms as in \Cref{presheaf_morphism} and fix a face $F\subseteq\overline{C}$. Our aim is to show that unless we are in the setting of \Cref{characters_Haff_Ho}(i) and $F$ is defined by $S_F=\{s', s''\}$ (with $s', s''$ as in the statement of the Proposition), then $f_F\sim 0$. We will do this in a sequence of reductions where we will impose more conditions on $\chi$, $\chi'$, $F$ and $\Phi$ as we go, and show that $f_F\sim 0$ whenever those conditions fail. For this, first note that if $\Res^{\Hea}_{\He_F}(\chi)$ is projective then $f_F\sim 0$. Similarly, if there is a face $F'\subseteq \overline{F}$ such that $\Res^{\Hea}_{\He_{F'}}(\chi)\neq\Res^{\Hea}_{\He_{F'}}(\chi')$ then $f_F\sim 0$ by \Cref{facts_H_F}(i) and \Cref{presheaf_morphism}(i). To compute $(f_F)_{F\subseteq \overline{C}}$, we are thus left to consider a face $F$ such that $\Res^{\Hea}_{\He_F}(\chi)$ is not projective and $\Res^{\Hea}_{\He_{F'}}(\chi)=\Res^{\Hea}_{\He_{F'}}(\chi')$ for all $F'\subseteq\overline{F}$. In particular, note that this last condition implies that $\xi=\xi'$ since $k[T(\mathbb{F}_q)]=\He_C\subseteq \He_F$.

Assume from now on that $\xi=\xi'$ and that $F$ is as above. Then the set 
$$
S(\chi, \chi'):=\{s\in S\mid \chi(T_{\hat{s}})\neq \chi'(T_{\hat{s}})\}
$$
must be non-empty since $\chi\neq\chi'$ and $\xi=\xi'$, and we denote its elements by $s_1,\ldots, s_n$. For each $1\leq j\leq n$, we then write $S_{i_j}$ for the irreducible component containing $s_j$. Observe further that our assumptions on $F$, $\chi$ and $\chi'$ imply that for each $j$ there is no face $F'\subseteq\overline{F}$ with $S_{F'}=S_F\cup\{s_j\}$. In other words, we must have $S_F\cap S_{i_j}=S_{i_j}\setminus\{s_j\}$ for every $1\leq j\leq n$ by \eqref{bijection_faces}. In particular, the corresponding components $S_{i_1}, \ldots, S_{i_n}$ are all distinct.

We now begin our proof of \Cref{characters_Haff_Ho} by first showing that $f_F\sim 0$ whenever $S_F\not\subseteq S_\xi$.

\begin{lem}\label{claim1} Suppose $\chi\neq\chi'$ are supersingular characters of $\Hea$ and that $\underline{f}:=(f_F)\in\prod_{F\subseteq\overline{C}} \Hom_{\He_F}(\chi, \chi')$ is as in \Cref{presheaf_morphism}. Assume $F\subseteq \overline{C}$ is a face such that $\Res^{\Hea}_{\He_F}(\chi)$ is not projective and $\Res^{\Hea}_{\He_{F'}}(\chi)=\Res^{\Hea}_{\He_{F'}}(\chi')$ whenever $F'\subseteq \overline{F}$. With the above notation, we have:
\begin{enumerate}
\item for each $1\leq j\leq n$, the component $\Phi_{i_j}$ is not of type $A_1$; and
\item if $S_F\not\subseteq S_\xi$ then $f_F\sim 0$.
\end{enumerate}
\end{lem}

\begin{proof}
(i) As explained above we have $S_F\cap S_{i_j}=S_{i_j}\setminus\{s_j\}$ for all $1\leq j\leq n$. Suppose $\Phi_{i_j}$ is of type $A_1$ for some $j$. Then $S_{i_j}$ consists of two simple affine reflections $s_j$ and $s'_j$, say, with the same associated coroot. Since $\chi(T_{\hat{s_j}})\neq \chi'(T_{\hat{s_j}})$ we must have $s_j\in S_\xi$ and so $S_{i_j}\subseteq S_\xi$ by definition of $S_\xi$. Without loss of generality, we assume $\chi(T_{\hat{s_j}})=-1$ and $\chi'(T_{\hat{s_j}})=0$. Then by supersingularity of $\chi$ and $\chi'$ we must have $\chi(T_{\hat{s'_j}})=0$ and $\chi'(T_{\hat{s'_j}})=-1$. But we also have that $s'_j\in S_F$ so that $\Res^{\Hea}_{\He_{F}}(\chi)\neq\Res^{\Hea}_{\He_{F}}(\chi')$, contradicting our assumptions.

(ii) Let $s'\in S_F\setminus S_\xi$ and note as above that we necessarily have $s'\neq s_j$ for any $1\leq j\leq n$. We can now define $F_1, F_2\subseteq \overline{F_3}$ by $F=F_1$, $S_{F_2}=\{s_1, s'\}$ and $S_{F_3}=\{s'\}$. Note that such a face $F_2$ exists because of (i). \Cref{Koziol_projective_face}(i) then tells us that $\Res^{\Hea}_{\He_{F_3}}(\chi)$ is not projective. As $\Res^{\Hea}_{\He_{F_2}}(\chi)\neq \Res^{\Hea}_{\He_{F_2}}(\chi')$, \Cref{presheaf_morphism}(iii) gives us that $f_F\sim 0$ as required.
\end{proof}

Next, in the case where instead $S_F\subseteq S_\xi$, we reduce quite thoroughly the possibilities for $\chi$ and $\Phi$ in order to avoid $f_F\sim 0$.

\begin{lem}\label{claim2} Suppose that $\chi$, $\chi'$, $\underline{f}$ and $F\subseteq \overline{C}$ are given as in \Cref{claim1}, and assume further that $S_F\subseteq S_\xi$.
\begin{enumerate}
\item There are $s', s''\in S_F$, adjacent in the affine Dynkin diagram of the root system $\Phi$, such that $\chi(T_{\hat{s'}})\neq \chi(T_{\hat{s''}})$.
\item Assume that one of the following holds:
\begin{enumerate}
\item there are $s', s''\in S_F$ as in (i) such that $s', s''\notin S_{i_j}$ for any $1\leq j\leq n$;
\item $\abs{S_{i_j}}>3$ for some $1\leq j\leq n$; or
\item $n=\abs{S(\chi, \chi')}>1$.
\end{enumerate}
Then $f_F\sim 0$.
\end{enumerate} 
\end{lem}

\begin{proof}
The existence of $s', s''\in S_F$ as in (i) follows from \Cref{Koziol_projective_face}(ii) because $\Res^{\Hea}_{\He_{F}}(\chi)$ is not projective by assumption. For (ii), we let $F_1, F_2\subseteq \overline{F_3}$ with $F=F_1$, $S_{F_2}=\{s_j, s', s''\}$ and $S_{F_3}=\{s', s''\}$, for a suitably chosen $1\leq j\leq n$ so that such an $F_2$ exists. In case (a), any choice of $j$ works. In case (b), we pick $j$ such that $\abs{S_{i_j}}>3$. In case (c), we pick $j$ such that $s', s''\notin S_{i_j}$. Either way, as $\Res^{\Hea}_{\He_{F_2}}(\chi)\neq\Res^{\Hea}_{\He_{F_2}}(\chi')$, it follows from \Cref{facts_H_F}(i) and \Cref{presheaf_morphism}(iii) that we only need to show that $\Res^{\Hea}_{\He_{F_3}}(\chi)$ is not projective, which itself holds by \Cref{Koziol_projective_face}(ii).
\end{proof}

We now strengthen further our assumptions on $\chi$, $\chi'$ and the root system. We write $\Phi=\sqcup_{i=1}^r\Phi_i$ as before and assume that $S_1=\{s, s', s''\}\subseteq S_\xi$, where $\chi(T_{\hat{s'}})\neq\chi(T_{\hat{s''}})$ and $s', s''$ are adjacent in the affine Dynkin diagram of $\Phi$. Furthermore, we assume that $s\in S_1$ is the unique simple affine reflection for which $\chi(T_{\hat{s}})\neq\chi'(T_{\hat{s}})$. By \Cref{claim1} and \Cref{claim2}(ii), we have $f_F\sim 0$ for every $F\subseteq\overline{C}$ unless these assumptions hold. With these strengthened assumptions, we now finally show that $f_F\sim 0$ for every $F\subseteq\overline{C}$ unless the hypotheses of \Cref{characters_Haff_Ho}(i) are satisfied. 

\begin{lem}\label{claim3} Let $\chi$, $\chi'$, $\underline{f}$ be given as in \Cref{claim2}, and suppose that $\Phi$, $\chi$ and $\chi'$ satisfy the assumptions above. Let $F\subseteq \overline{C}$ be any face with $S_F\cap S_1=\{s', s''\}$. 
\begin{enumerate}
\item If $S_\xi\neq S$ then $f_F\sim 0$.
\item If $r>1$ and $\Phi_i$ is not of type $A_1$ for some $i\geq 2$, then $f_F\sim 0$.
\end{enumerate} 
\end{lem}

\begin{proof}
For (i), pick $t\in S\setminus S_\xi$ and note that by assumption $t\notin S_1$. Define $F_1, F_2\subseteq \overline{F_3}$ by $F_1=F$,  $S_{F_2}=\{t,s', s''\}$ and $S_{F_3}=\{s', s''\}$. By \Cref{Koziol_projective_face}(ii), $\Res^{\Hea}_{\He_{F_3}}(\chi)$ is not projective so that by applying \Cref{presheaf_morphism}(iii) we are reduced to showing that $f_{F_2}\sim 0$. For this, define $F'_1,  F'_2\subseteq \overline{F'_3}$ by $F'_1=F_2$, $S_{F'_2}=\{t, s, s'\}$ and $S_{F'_3}=\{t, s'\}$. \Cref{Koziol_projective_face}(i) gives that $\Res^{\Hea}_{\He_{F'_3}}(\chi)$ is not projective. Since $\Res^{\Hea}_{\He_{F'_2}}(\chi)\neq\Res^{\Hea}_{\He_{F'_2}}(\chi')$,  \Cref{facts_H_F}(i) and another application of \Cref{presheaf_morphism}(iii) imply $f_{F_2}=f_{F'_1}\sim 0$ as required.

For (ii), we may therefore assume that $S_\xi=S$. Suppose that $\Phi_i$ is not of type $A_1$ for some $i\geq 2$. Since $\chi$ is supersingular we may find $t',t''\in S_i$, adjacent in the affine Dynkin diagram, such that $\chi(T_{\hat{t'}})\neq\chi(T_{\hat{t''}})$. Again appealing to \Cref{Koziol_projective_face}(ii), we have that $\Res^{\Hea}_{\He_{F_3}}(\chi)$ is not projective where $S_{F_3}=\{t', t''\}$. Define $F_1, F_2\subseteq \overline{F_3}$ by $S_{F_1}=\{s', s'', t', t''\}$ and $S_{F_2}=\{s, t', t''\}$. Once again \Cref{presheaf_morphism}(iii) gives that $f_{F_1}\sim 0$ since $\Res^{\Hea}_{\He_{F_2}}(\chi)\neq\Res^{\Hea}_{\He_{F_2}}(\chi')$. Hence $f_{\tilde{F}}\sim 0$ by \Cref{presheaf_morphism}(i), where $S_{\tilde{F}}=\{s', s''\}$. We already saw that $\Res^{\Hea}_{\He_{\tilde{F}}}(\chi)$ is not projective and thus \Cref{presheaf_morphism}(ii) implies that $f_F\sim 0$.
\end{proof}

Finally, in order to deal with part (i) of \Cref{characters_Haff_Ho} and for later use we will need the following two lemmas.

\begin{lem}\label{tensor_lemma}
Suppose that $R_1, R_2$ are two Gorenstein $k$-algebras such that the tensor product algebra $R=R_1\otimes_k R_2$ is also Gorenstein. Let $\m_1, \m_1'\in \Mod(R_1)$ and let $\m_2\in \Mod(R_2)$ with $\pd_{R_2}(\m_2)<\infty$. If $\m_1\cong \m_1'$ in $\Ho(R_1)$, then $\m_1\otimes_k\m_2\cong \m'_1\otimes_k\m_2$ in $\Ho(R)$.
\end{lem}

\begin{proof}
It suffices to show that the exact functor $-\otimes_k \m_2:\Mod(R_1)\to\Mod(R)$ preserves weak equivalences, since then $\Ho(-\otimes_k \m_2):\Ho(R_1)\to \Ho(R)$ is well-defined and preserves isomorphisms. By \Cref{preserve_reflect_we}(i), it suffices to show that this functor preserves trivial objects, i.e.\ that if $\mathfrak{e}\in\Mod(R_1)$ has finite projective dimension over $R_1$ then $\mathfrak{e}\otimes_k \m_2$ has finite projective dimension over $R$. But since $\pd_{R_2}(\m_2)<\infty$ by assumption, this is clear.
%
\end{proof}

\begin{lem}\label{Oll_Schn_rank2} Suppose that $\Phi$ is irreducible of rank 2 and let $\chi=(\xi, J)$ be a supersingular character with $S_\xi=S$.
\begin{enumerate}
\item The augmentation map $\bigoplus_{F\in\mathscr{F}_0} \chi\otimes_{\He_F}\Hea\to \chi$ in \eqref{res_Haff3} is a trivial fibration.
\item Let $\m\in\Mod(\He)$ be a simple supersingular module such that $\Res^\He_{\Hea}(\m)$ contains $\chi$. If $p\nmid|\Omega_{\text{tor}}|$ then the augmentation $\bigoplus_{F\in\mathscr{F}_{(0)}} \Res^\He_{\He^\dagger_F}(\m)\otimes_{\He^\dagger_F}(\epsilon_F)\He\to \m$ from \eqref{res_Oll_Sch} is a trivial fibration.
\end{enumerate}
\end{lem}

\begin{proof}
It suffices to show that all the terms of higher degrees in the relevant complexes have finite projective dimension. For (i), since $\Phi$ has rank 2 we have that $\Res^{\Hea}_{\He_{F}}(\chi)$ is projective for any face $F$ of positive dimension because $S=S_\xi$ and $|S_F|\leq 1$, cf.\ \Cref{Koziol_projective_face}(ii). For (ii), we similarly get that $\Res^{\Hea}_{\He_{F}}(\chi^\omega)$ is projective for any face $F$ of positive dimension and any $\omega\in\Omega$. Thus we deduce from \eqref{n_decomp} that $\Res^{\He}_{\He_{F}}(\m)$ is projective for all such $F$ as well. By our condition on $p$, this is equivalent to saying that $\Res^{\He}_{\He^\dagger_{F}}(\m)$ has finite projective dimension for all $F$ of positive dimension by \cite[Lemma 3.4]{Koz} and we are done since $(\epsilon_F)\He$ is free for all $F\subseteq \overline{C}$ (cf.\ \cite[Lemma 3.5]{Koz}).
\end{proof}

\begin{proof}[Proof of \Cref{characters_Haff_Ho}] 
Given a map map $f:\chi\to \chi'$ in $\Ho(\Hea)$, set $f_F$ to be any lift of $\Ho(\Delta)(f)_F=\Ho(\Res^{\Hea}_{\He_F})(f)$ to $\Hom_{\He_F}(\chi, \chi')$. \Cref{claim1}, \Cref{claim2} and \Cref{claim3} imply that if we're not in the precise situation of (i), then $f_F\sim 0$ for all $F\subseteq\overline{C}$. This shows (ii). In the setting of (i), we still have that $f_{F'}\sim 0$ for all faces $F'$ for which $S_{F'}\cap S_1\neq\{ s', s''\}$. If on the other hand $S_{F'}\cap S_1=\{ s', s''\}$ then $F'\subseteq\overline{F_1}$, where $S_{F_1}=\{s', s''\}$, so that $\Res^{\He_{F'}}_{\He_{F_1}}(f_{F'})\sim f_{F_1}$. As $\Res^{\Hea}_{\He_{F_1}}(\chi)=\Res^{\Hea}_{\He_{F_1}}(\chi')$ is not projective, cf.\ \Cref{Koziol_projective_face}(ii), it follows from \Cref{facts_H_F}(i) that $\Res^{\He_{F'}}_{\He_{F_1}}(f_{F'})=f_{F_1}$ and thus that $\Ho(\Delta)(f)$ is uniquely determined by $f_{F_1}$. This gives that $\dim_k\Ho(\Delta)([\chi, \chi']_\Hea)\leq 1$ by \Cref{facts_H_F}(i) again. Moreover, we see that $\Ho(\Delta)(f)$ is never an isomorphism. Indeed, swapping $s'$ and $s''$ if necessary, we may assume that $s$ and $s'$ correspond to adjacent nodes in the affine Dynkin diagram and we let $F''\subseteq \overline{C}$ be defined by $S_{F''}=\{s, s'\}$. Then $f_{F''}\sim 0$ is not a weak equivalence since \Cref{Koziol_projective_face}(ii) again shows that $\Res^{\Hea}_{\He_{F''}}(\chi')$ is not projective.

We are left to show that there exists a map $f:\chi\to \chi'$ in $\Ho(\Hea)$ such that $f_{F_1}=\id$. Suppose first that $r=1$ and so that $\Phi=\Phi_1$ is irreducible of rank 2. Then the augmentation
$$
\varepsilon:C_0:=\bigoplus_{F\in\mathscr{F}_0} \chi\otimes_{\He_F}\Hea\to \chi
$$
is a trivial fibration, cf.\ \Cref{Oll_Schn_rank2}. Note that $C_0$ is in fact a cofibrant replacement of $\chi$ by \Cref{facts_H_F}(ii) and that $F_1\in \mathscr{F}_0$. Now, the map $\id:\Res^{\Hea}_{\He_{F_1}}(\chi)\to \Res^{\Hea}_{\He_{F_1}}(\chi')$ has an adjoint $\Hea$-linear morphism $g:\chi\otimes_{\He_{F_1}}\Hea\to \chi'$. Extending $g$ by zero on the other summands of $C_0$, this gives a map $f:\chi\xrightarrow{\varepsilon^{-1}} C_0\xrightarrow{g} \chi'$ in $\Ho(\Hea)$ which does satisfy $f_{F_1}=\id$. Indeed, this follows because the map $\chi\to \chi\otimes_{\He_{F_1}}\Hea\subseteq C_0$, $x\mapsto x\otimes 1$, is an $\He_{F_1}$-linear splitting of $\varepsilon$ and thus is equal to its inverse in $\Ho(\He_{F_1})$, and therefore $f_{F_1}$ must be the adjoint morphism of $g$ which is the identity by construction.

When $r>1$, first note that $\chi$ and $\chi'$ factor through the algebra $\He^0_\aff:=e_\xi\Hea$ (where $e_\xi$ is central in $\Hea$ since $S=S_\xi$). The latter has generators $\{e_\xi T_{\hat{s}}\mid s\in S\}$ satisfying the braid relations and such that $(e_\xi T_{\hat{s}})^2=-e_\xi T_{\hat{s}}$ (cf. \eqref{quad_exi}). From the algebra decomposition $\Hea=\He^0_\aff\oplus (1-e_\xi)\Hea$ we get that $\He^0_\aff$ is also a Gorenstein algebra and that a Gorenstein projective module, resp.\ module of finite projective dimension, over $\He^0_\aff$ is still Gorenstein projective, resp.\ of finite projective dimension as an $\Hea$-module. Hence any cofibrant replacement in $\Mod(\He^0_\aff)$ is also a cofibrant replacement in $\Mod(\Hea)$. It thus suffices to construct a map $f:\chi\to\chi'$ in $\Ho(\He^0_\aff)$ whose restriction to $e_\xi\He_{F_1}$ is the identity.

By the braid relations, we have a decomposition $\He^0_\aff=\He_1\otimes_k \He_2$, where $\He_1$ is the subalgebra generated by $\{e_\xi T_{\hat{s}}\mid s\in S_1\}$ and $\He_2$ is the subalgebra generated by $\{e_\xi T_{\hat{s}}\mid s\in S\setminus S_1\}$. We correspondingly have decompositions $\Res^\Hea_{\He^0_\aff}(\chi)=\chi_1\otimes_k \chi_2$, $\Res^\Hea_{\He^0_\aff}(\chi')=\chi'_1\otimes_k \chi'_2$ where $\chi_i, \chi'_i$ is a character of $\He_i$ ($i=1, 2$) and $\chi_2=\chi'_2$ by assumption, and the sets $J, J'\subseteq S$ also decompose as $J=J_1\sqcup J_2$ and $J'=J'_1\sqcup J'_2$ with $J_2=J'_2\subseteq S\setminus S_1$ and $J_1, J'_1\subseteq S_1$. We note that $\He_1$ and $\He_2$ depend only on the root systems $\Phi_1$ and $A_1\times \cdots\times A_1$ ($r-1$ terms) respectively. We thus let $G_1$ and $G_2$ be the $\mathfrak{F}$-rational points of two split semisimple, simply-connected groups of type $\Phi_1$ and $A_1\times \cdots\times A_1$ respectively, with corresponding pro-$p$ Iwahori (affine) Hecke algebras $\He_\aff^1$ and $\He_\aff^2$. Then we get that $\He_i\cong e_{\mathbf{1}}\He_\aff^i$, say, and that $\chi_i$ and $\chi'_i$ come from two supersingular characters of $\He_\aff^i$ by taking the pairs $(J_i, \mathbf{1})$ and $(J'_i, \mathbf{1})$ ($i=1, 2$).

We now choose a cofibrant replacement $q_c:Q_c\chi_1\to \chi_1$ in $\Mod(\He_1)$. The argument for $r=1$ and \Cref{factor_through} then show that there is a map $g:Q_c\chi_1\to \chi'_1$ such that $g\circ q_c^{-1}$ goes to the identity on $\chi_1$ under restriction of scalars to $e_\xi \He_{F_1}$. But since $\chi_2=\chi'_2$ has finite projective dimension over $\He_2$ (cf.\ \Cref{Koziol_ss_thm}), it follows from \Cref{tensor_lemma} that $q=(q_c\otimes\id):Q_c\chi_1\otimes_k\chi_2\to \chi_1\otimes_k\chi_2=\chi$ is a weak equivalence in $\Mod(\He^0_\aff)$. Hence $g\otimes 1:Q_c\chi_1\otimes_k\chi_2\to \chi'_1\otimes_k\chi'_2=\chi'$ gives in $\Ho(\He^0_\aff)$ a map $f=(g\otimes 1)\circ q^{-1}$ such that $f_{F_1}=\id$, as required.
\end{proof}

\begin{rem}\label{faithful?} In case (i) of \Cref{characters_Haff_Ho}, when $r=1$, the above proof actually tells us that
\begin{align*}
[\chi, \chi']_\Hea &\cong \bigoplus_{F\in\mathscr{F}_0}[\chi\otimes_{\He_F}\Hea, \chi']_\Hea\\
&\cong \bigoplus_{F\in\mathscr{F}_0}[\chi, \chi']_{\He_F}\\
&\cong \Hom_{\He_{F_1}}(\Res^{\Hea}_{\He_{F_1}}(\chi), \Res^{\Hea}_{\He_{F_1}}(\chi'))
\end{align*}
where $F_1$ is as in the proof. The second isomorphism follows from the Quillen adjunction in \Cref{facts_H_F}(ii) and the third isomorphism holds by \Cref{facts_H_F}(i) because $\chi\neq\chi'$ for any other face in $\mathscr{F}_0$. Thus $\dim_k [\chi, \chi']_\Hea=1$ in this case. We suspect that $\Ho(\Delta)$ might be faithful in general, perhaps identifying bijectively a morphism set $[\m, \mathfrak{n}]_{\Hea}$ in $\Ho(\Hea)$ with the set of all $(f_F)\in\prod_{F\subseteq\overline{C}}[\m, \mathfrak{n}]_{\He_F}$ such that $\Ho(\Res_{\He_F}^{\He_{F'}})(f_{F'})=f_F$ whenever $F'\subseteq \overline{F}$. When the group $\mathbb{G}$ has semisimple rank 1, this is in fact true. Indeed, as the complex \eqref{res_Haff3} is a short exact sequence in that case and since $\He_C$ is always a semisimple algebra, we have that the augmentation again realises $\bigoplus_{F\in \mathscr{F}_0}\Hea\otimes_{\He_F}\m$ as a Gorenstein projective replacement of $\m$, for any $\m\in\Mod(\Hea)$. After chasing through the adjunctions one deduces, for any $\m, \mathfrak{n}\in\Mod(\Hea)$, that the isomorphism $[\m, \mathfrak{n}]_\Hea\cong \bigoplus_{F\in \mathscr{F}_0} [\m, \mathfrak{n}]_{\He_F}$ obtained from the augmentation agrees with the map induced by $\Ho(\Delta)$. Finally, the compatibility condition between the $f_F$'s with $F\in \mathscr{F}_0$ and $f_C$ is empty here because $\Ho(\He_C)=0$.
%
\end{rem}

We now study morphisms between simple supersingular $\He$-modules in $\Ho(\He)$. Recall the parametrisation of these modules by pairs $(\chi, V)$ as in the discussion preceding \Cref{normal_subgroups}.

\begin{prop}\label{same_character} Suppose that $\m=(\chi\otimes V)\otimes_{\He_\chi}\He$ and $\m'=(\chi'\otimes V')\otimes_{\He_\chi}\He$ are two simple supersingular $\He$-modules satisfying $\m\cong \m'$ in $\Ho(\He)$. Assume that $\pd_{\He}(\m)=\infty$ and that $\chi, \chi'$ and $\Phi$ are not as in \Cref{characters_Haff_Ho}(i). Then $\chi'=\chi^{\tilde{\omega}}$ for some $\tilde{\omega}\in\widetilde{\Omega}$.
\end{prop}

\begin{proof}
By \Cref{free_adjunction}, we have that restriction of scalars $\Res^{\He}_{\Hea}$ is a right Quillen functor which preserves all weak equivalences. If $\m\cong \m'$ in $\Ho(\He)$ we therefore have that $\Ho(\Res^{\He}_{\Hea})(\m)$ and $\Ho(\Res^{\He}_{\Hea})(\m')$ are also isomorphic in $\Ho(\Hea)$. By \eqref{n_decomp} we have
$$
\Res^{\He}_{\Hea}(\m)\cong\left(\bigoplus_{\tilde{\omega}\in\widetilde{\Omega}/\widetilde{\Omega}_\chi} \chi^{\tilde{\omega}}\right)^{\oplus \dim_k V}.
$$
Assume that $\chi'\neq\chi^{\tilde{\omega}}$ for all $\tilde{\omega}\in\widetilde{\Omega}$. Since we get from the above that
\begin{equation}\label{restriction_m}
[\m,\m']_\Hea\cong\left(\prod_{(\tilde{\omega}, \tilde{\omega}')\in\widetilde{\Omega}/\widetilde{\Omega}_\chi\times \widetilde{\Omega}/\widetilde{\Omega}_{\chi'}}[\chi^{\tilde{\omega}}, (\chi')^{\tilde{\omega}'}]_\Hea\right)^{\oplus (\dim_k V)(\dim_k V')},
\end{equation}
it then follows from \Cref{characters_Haff_Ho}(ii) that $[\m,\m']_\Hea$ has image zero under $\Ho(\Delta)$. If $f:\m\to\m'$ denotes an isomorphism in $\Ho(\He)$, then this implies that $\Ho(\Res^{\He}_{\He_F})(f)$ is both zero and an isomorphism for every $F\subseteq \overline{C}$. We deduce that $\Res^{\He}_{\He_F}(\m)$, and thus $\Res^{\Hea}_{\He_F}(\chi)$, is projective for every $F\subseteq \overline{C}$. Applying \Cref{facts_H_F}(iii) and \Cref{Koziol_ss_thm} it follows that $\pd_{\He}(\m)<\infty$, contradicting our assumptions.
%
\end{proof}

The conclusion of \Cref{same_character} is false when we are in the setting of \Cref{characters_Haff_Ho}(i), as is visible from the following examples.

\begin{ex} Suppose that $G=\mathrm{PGL}_3(\mathfrak{F})$ and assume that $p\neq 3$. Then $\Omega=\langle \omega\rangle\cong \Z/3\Z$ where $\omega$ is the image of $\begin{psmallmatrix}
0 & 0 & \varpi\\
1 & 0 & 0\\
0 & 1 & 0
\end{psmallmatrix}$ in $W$. Similarly, let $\hat{\omega}$ denote the image of the same matrix in $\widetilde{W}$. We have $S=\{s, s', s''\}$ and let $\chi\neq \chi'$ be two supersingular characters of $\Hea$ as in \Cref{characters_Haff_Ho}(i), i.e. corresponding to the pairs $(J, \xi)$ and $(J', \xi)$ with $S_\xi=S$, $J=\{s, s'\}$ and $J'=\{s'\}$. We will use the notation $\chi=(-1, -1,0)_\xi$ and $\chi'=(0,-1,0)_\xi$ for these characters, the entries of the tuple corresponding to the value of the character at $T_{\hat{s}}$, $T_{\hat{s'}}$ and $T_{\hat{s''}}$, taken in that order. Note that $S=S_\xi$ forces the character $\xi$ to be fixed by $\omega$ (cf.\ \Cref{det_rmk} later, viewing $\xi$ here as a character of the torus of $\mathrm{GL}_3(\mathbb{F}_q)$ which is trivial on scalar matrices). Also, the element $\omega$ acts on the set $S$ by permuting its elements cyclically and correspondingly we see that $\widetilde{\Omega}_\chi=\widetilde{\Omega}_{\chi'}=T(\mathbb{F}_q)$ and $\He_\chi=\He_{\chi'}=\Hea$. If we let $V$ and $V'$ be two irreducible finite dimensional representations of $\widetilde{\Omega}_\chi=\widetilde{\Omega}_{\chi'}$ on which $T(\mathbb{F}_q)$ acts by $\chi$ and $\chi'$ respectively, we then automatically have $V=\chi$ and $V'=\chi'$. Thus we obtain two nonisomorphic simple supersingular $\He$-modules by taking $\m=\chi\otimes_{\Hea}\He$ and $\m'=\chi'\otimes_{\Hea}\He$. Furthermore, those have infinite projective dimension (cf. \Cref{Koziol_ss_thm}). We claim that $\m\cong \m'$ in $\Ho(\He)$.

To see this, first note that \eqref{n_decomp} gives in this example that
\begin{equation}\label{restriction_chi}
\Res^\He_\Hea \m\cong \chi_0\oplus\chi_1\oplus \chi_2:=(-1,-1,0)_\xi\oplus(0, -1, -1)_\xi\oplus (-1, 0, -1)_\xi,
\end{equation}
and similarly that
\begin{equation}\label{restriction_chi'}
\Res^\He_\Hea \m'\cong \chi'_0\oplus\chi'_1\oplus \chi'_2:=(0, -1, 0)_\xi\oplus (0, 0, -1)_\xi\oplus(-1, 0, 0)_\xi.
\end{equation}
Next, observe that for every face $F\subseteq \overline{C}$ with $F\neq C$ the stabiliser $\Omega_F$ of $F$ in $\Omega$ is trivial. Thus, for such faces, $\He^\dagger_F=\He_F$ by \cite[Lemma 4.20(i)]{OS14}. Let $x$ denote the vertex of $\overline{C}$ with $S_x=\{s', s''\}$. By our assumption on $p$ and since all faces in $\overline{C}$ of a given dimension are $\Omega$-conjugate to each other, we obtain by applying \Cref{Oll_Schn_rank2}(ii) that the augmentation $\varepsilon:\Res^\He_{\He_x}(\m)\otimes_{\He_x} \He\to \m$ in the Ollivier-Schneider resolution is a trivial fibration. Here we dropped the twist by orientation $\epsilon_x$ character from the notation as it is trivial (since the stabiliser of $x$ automatically fixes it pointwise). Completely analogously, $\varepsilon:\Res^\He_{\He_x}(\m')\otimes_{\He_x} \He\to\m'$ is also a trivial fibration.

Now, by appealing to \Cref{Koziol_projective_face}(ii) we see that $\Res^\Hea_{\He_x}(\chi_1)$ and $\Res^\Hea_{\He_x}(\chi'_2)$ are projective, while $\Res^\Hea_{\He_x}(\chi_0)=\Res^\Hea_{\He_x}(\chi'_0)$ and $\Res^\Hea_{\He_x}(\chi_2)=\Res^\Hea_{\He_x}(\chi'_1)$ are not projective. By plugging all this into \eqref{restriction_chi} and \eqref{restriction_chi'}, we get a sequence of isomorphisms
$$
\Res^\He_{\He_x}(\m)\cong \Res^\Hea_{\He_x}(\chi_0)\oplus \Res^\Hea_{\He_x}(\chi_2)=\Res^\Hea_{\He_x}(\chi'_0)\oplus\Res^\Hea_{\He_x}(\chi'_1)\cong \Res^\He_{\He_x}(\m')
$$
in $\Ho(\He_x)$. Tensoring with $\He$, we obtain from the above the claimed isomorphism
$$
\m\cong \Res^\He_{\He_x}(\m)\otimes_{\He_x} \He\cong \Res^\He_{\He_x}(\m')\otimes_{\He_x} \He\cong\m'
$$
in $\Ho(\He)$.
\end{ex}

\begin{ex}\label{GL3_exception}
Suppose that $G=\mathrm{GL}_3(\mathfrak{F})$. Then $\Omega=\langle \omega\rangle\cong \Z$ where $\omega$ is once again the image of $\begin{psmallmatrix}
0 & 0 & \varpi\\
1 & 0 & 0\\
0 & 1 & 0
\end{psmallmatrix}$ in $W$. We keep the same notation as in the previous example, so we let $\hat{\omega}$ be the same lift to $\widetilde{\Omega}$ and we let $\chi=(-1, -1, 0)_\xi$ and $\chi'=(0,-1,0)_\xi$. Note that $S_\xi=S$ forces $\xi=\det^a$ for some $a\in\Z$ (cf.\ \Cref{det_rmk} later). We now have $\widetilde{\Omega}_\chi=\widetilde{\Omega}_{\chi'}=T(\mathbb{F}_q)\times\langle\hat{\omega}^3\rangle$. Once again we let $V$ and $V'$ be irreducible, finite dimensional representations of $\widetilde{\Omega}_\chi=\widetilde{\Omega}_{\chi'}$ on which $T(\mathbb{F}_q)$ acts by $\chi$ and $\chi'$ respectively. Note that they are one dimensional because $\hat{\omega}^3$ is central in $\widetilde{W}$, and thus are uniquely determined by the scalars $\lambda,\lambda'\in k^\times$ by which $\hat{\omega}^3$ acts on $V$ and $V'$ respectively. Assume that $\lambda=\lambda'$. We obtain two nonisomorphic, $3$-dimensional simple modules $\m=(\chi\otimes V)\otimes_{\He_\chi}\He$ and $\m'=(\chi'\otimes V')\otimes_{\He_\chi}\He$ of infinite projective dimension. We claim that $\m\cong \m'$ in $\Ho(\He)$.

The maps $\varepsilon_\m:\Res^\He_{\He_x}(\m)\otimes_{\He^\dagger_x} \He\to\m$ and $\varepsilon_{\m'}:\Res^\He_{\He_x}(\m')\otimes_{\He^\dagger_x} \He\to\m'$ from the Ollivier-Schneider resolution are trivial fibrations, cf.\ \Cref{Oll_Schn_rank2}(ii), so that it again suffices to show that $\Res^\He_{\He^\dagger_x}(\m)\cong \Res^\He_{\He^\dagger_x}(\m')$ in $\Ho(\He^\dagger_x)$. Note that $\Omega_x=\langle\omega^3\rangle$. Similarly to the previous example we have decompositions $\Res^\He_{\He^\dagger_x}(\m)\cong \m_0\oplus\m_1\oplus \m_2$ and $\Res^\He_{\He^\dagger_x} (\m')\cong \m'_0\oplus\m'_1\oplus \m'_2$. Here, all $\m_j$ and $\m'_j$ are one-dimensional, and $\He_x$ acts on $\m_0$, $\m_1$ and $\m_2$ via the characters $(-1, -1, 0)_\xi$, $(0, -1, -1)_\xi$ and $(-1, 0, -1)_\xi$ respectively, and on $\m'_0$, $\m'_1$ and $\m'_2$ via the characters $(0, -1, 0)_\xi$, $(0, 0, -1)_\xi$ and $(-1, 0, 0)_\xi$ respectively.  Furthermore, $\hat{\omega}^3$ acts as $\lambda\id$ on both $\m$ and $\m'$. Thus we see that $\m_0\cong\m'_0$ and $\m_2\cong\m'_1$, and \Cref{Koziol_projective_face}(ii) together with \cite[Lemma 3.4]{Koz} imply that $\m_1$ and $\m'_2$ have finite projective dimension over $\He^\dagger_x$. Putting everything together, this gives
$$
\Res^\He_{\He^\dagger_x}(\m)\cong \m_0\oplus\m_1\cong \m'_0\oplus\m'_2\cong \Res^\He_{\He^\dagger_x}(\m')
$$
in $\Ho(\He^\dagger_x)$, as required.
\end{ex}

\subsection{The $\widetilde{\Omega}$-action on morphisms}

We recall the action of subgroups of $\Omega$ on $\Hom$ spaces for various subalgebras of $\He$ as in \cite[\S 4.2]{Abe22}. Let $A, B\subseteq \He$ be two Gorenstein $k$-subalgebras such that $k[T(\mathbb{F}_q)]\subseteq A\subseteq B$ and $B$ is free over $A$, both as a left and as a right module. Furthermore, we also let $\widetilde{\Omega}_B:=\widetilde{\Omega}\cap B$ and denote by $\widetilde{\Omega}_A$ the normaliser of $A$ in $\widetilde{\Omega}_B$. One example to have in mind is $A=\Hea$, $\widetilde{\Omega}_A=\widetilde{\Omega}_B=\widetilde{\Omega}$ and $B=\He$, but we will see other situations where this applies as well. We will denote by $\Omega_A$ the image of $\widetilde{\Omega}_A$ under the canonical surjection $\widetilde{\Omega}\twoheadrightarrow \Omega$. We also denote by $\He_{A,\Omega_A}$ the $k$-subalgebra of $\He$ which is generated by $A$ and $\{T_{\tilde{\omega}}\}_{\tilde{\omega}\in\widetilde{\Omega}_A}$. To simplify notation, for any $\m\in\Mod(B)$ we will denote $\Res^B_A(\m)$ by $\m|_A$.

We fix $\tilde{\omega}\in\widetilde{\Omega}_B$ and let $A^{\tilde{\omega}}:=T_{\tilde{\omega}^{-1}}AT_{\tilde{\omega}}\subseteq B$. We have $A^{\tilde{\omega}}=A$ whenever $\tilde{\omega}\in\widetilde{\Omega}_A$, and in general the map $a\mapsto T_{\tilde{\omega}^{-1}}aT_{\tilde{\omega}}$ defines an isomorphism $A\cong A^{\tilde{\omega}}$ of $k$-algebras. Thus $A^{\tilde{\omega}}$ is also Gorenstein. This isomorphism induces an exact equivalence of categories $\mathcal{F}_{\tilde{\omega}}:\Mod(A)\xrightarrow{\simeq} \Mod(A^{\tilde{\omega}})$, where $\mathfrak{n}^{\tilde{\omega}}:=\mathcal{F}_{\tilde{\omega}}(\mathfrak{n})$ is defined to have the same underlying $k$-vector space as $\mathfrak{n}$ with $A^{\tilde{\omega}}$-action given by $x\cdot (T_{\tilde{\omega}^{-1}}aT_{\tilde{\omega}}):= xa$ for all $a\in A$ and all $x\in\mathfrak{n}^{\tilde{\omega}}$. Observe further that if $\m\in\Mod(B)$ then the map $x\mapsto xT_{\tilde{\omega}}$ defines a natural $A^{\tilde{\omega}}$-linear isomorphism $\nu_\m:(\m|_A)^{\tilde{\omega}}\xrightarrow{\cong} \m|_{A^{\tilde{\omega}}}$, i.e.\ there is a natural isomorphism $\nu:\mathcal{F}_{\tilde{\omega}}\circ\Res^B_A\to \Res^B_{A^{\tilde{\omega}}}$.  

Suppose now that $\m_1, \m_2\in\Mod(B)$ and let $f\in \Hom_A(\m_1|_A, \m_2|_A)$. By considering the image of $f$ under $\mathcal{F}_{\tilde{\omega}}$ and pulling back along the isomorphisms $\nu_{\m_i}$ ($i=1,2$), we obtain an $A^{\tilde{\omega}}$-linear map $f\cdot\tilde{\omega}:\m_1|_{A^{\tilde{\omega}}}\to \m_2|_{A^{\tilde{\omega}}}$ defined by $(f\cdot\tilde{\omega})(x)=f(xT_{\tilde{\omega}^{-1}})T_{\tilde{\omega}}$ for $x\in\m_1$. Thus we have constructed a map
$$
(-)\cdot\tilde{\omega}:\Hom_A(\m_1|_A, \m_2|_A)\to \Hom_{A^{\tilde{\omega}}}(\m_1|_{A^{\tilde{\omega}}}, \m_2|_{A^{\tilde{\omega}}})
$$
which, by functoriality, is compatible with composition: if $\m_3\in \Mod(B)$, and $f:\m_1|_A\to \m_2|_A$ and $g:\m_2|_A\to \m_3|_A$ are $A$-linear maps then we have $(g\circ f)\cdot\tilde{\omega}= (g\cdot\tilde{\omega})\circ (f\cdot\tilde{\omega})$. If we put $\tilde{\omega}\in \widetilde{\Omega}_A$ in this construction, we then obtain an action of $\widetilde{\Omega}_A$ on $\Hom_A(\m_1|_A, \m_2|_A)$. This action is trivial when restricted to $T(\mathbb{F}_q)$, and thus factors through an action of $\Omega_A$. Furthermore, we have $\Hom_{\He_{A,\Omega_A}}(\m_1|_{\He_{A,\Omega_A}}, \m_2|_{\He_{A,\Omega_A}})=\Hom_A(\m_1|_A, \m_2|_A)^{\Omega_A}$.

The above construction extends to morphisms in the homotopy category as follows. Since the functor $\mathcal{F}_{\tilde{\omega}}$ is exact and preserves free and projective modules, it preserves all weak equivalences by \Cref{preserve_reflect_we}(i). Since $B|_A$ and $B|_{A^{\tilde{\omega}}}\cong \mathcal{F}_{\tilde{\omega}}(B|_A)$ are free, we see in the same way that $\Res^B_A$ and $\Res^B_{A^{\tilde{\omega}}}$ also preserve all weak equivalences. Thus, these functors descend to homotopy categories and we get a diagram
\begin{equation}\label{omega_commutes_restriction}
\begin{tikzcd}
& \Ho(B) \arrow[swap]{dl}{\Ho(\Res^B_A)}\arrow{dr}{\Ho(\Res^B_{A^{\tilde{\omega}}})} & \\
\Ho(A) \arrow{rr}{\mathcal{F}_{\tilde{\omega}}} & & \Ho(A^{\tilde{\omega}})
\end{tikzcd}
\end{equation}
which commutes up to natural isomorphism (given by $\nu$). For $\m_1, \m_2\in\Mod(B)$ and $f\in [\m_1, \m_2]_A$, we now define $f\cdot\tilde{\omega}$ to be the pullback of $\mathcal{F}_{\tilde{\omega}}(f)$ along $\nu$. Again, this defines an action of $\Omega_A$ on $[\m_1, \m_2]_A$. From \eqref{omega_commutes_restriction}, we immediately have the following:

\begin{lem}\label{Ho_hom_invariants2} Suppose that $\m_1, \m_2\in\Mod(B)$ and let $\tilde{\omega}\in \widetilde{\Omega}_B$. Then the diagram
\begin{equation}
\begin{tikzcd}
 & {[\m_1, \m_2]_B} \arrow[swap]{dl}{\Ho(\Res^B_A)} \arrow{dr}{\Ho(\Res^B_{A^{\tilde{\omega}}})} & \\
 {[\m_1, \m_2]_A} \arrow{rr}{(-)\cdot\tilde{\omega}} & &  {[\m_1, \m_2]_{A^{\tilde{\omega}}}}
\end{tikzcd}
\end{equation}
commutes. In particular, we obtain that the image of the map
$$
\Ho(\Res^B_A):[\m_1, \m_2]_B\to [\m_1, \m_2]_A
$$
lies in $[\m_1, \m_2]_A^{\Omega_A}$.
\end{lem}

\begin{rem}
\begin{enumerate}
\item Assume that $\m_1$ is Gorenstein projective in $\Mod(B)$ and let $\theta:\mathfrak{p}\twoheadrightarrow \m_2$ be a surjection from a projective $B$-module. Then by \Cref{factor_through}, we have a commutative diagram
\begin{equation}\label{Ho_hom_invariants}
\begin{tikzcd}
\Hom_B(\m_1, \mathfrak{p}) \arrow{r} \arrow[hook]{d} &  \Hom_B(\m_1, \m_2) \arrow{r} \arrow[hook]{d} & {[\m_1, \m_2]_B}  \arrow{r} \arrow[dashed]{d} & 0\\
\Hom_A(\m_1|_A, \mathfrak{p}|_A)^{\Omega_A} \arrow{r} &  \Hom_A(\m_1|_A, \m_2|_A)^{\Omega_A} \arrow{r} & {[\m_1, \m_2]_A^{\Omega_A}}  \arrow{r} & 0
\end{tikzcd}
\end{equation}
where the top row is exact. By construction, the dashed map $[\m_1, \m_2]_B\to [\m_1, \m_2]_A^{\Omega_A}$ is given by $\Ho(\Res^B_A)$ and we see explicitly the last part of the Lemma here.
\item Take $A=\Hea$ and $B=\He$ so that $\Omega_A=\Omega$, and assume that $\mathbb{G}$ is semisimple and that $p$ does not divide $\abs{\Omega}$. Then the functor of taking $\Omega$-invariants is exact and the bottom row of \eqref{Ho_hom_invariants} is exact. Moreover, in that case the two leftmost vertical arrows in \eqref{Ho_hom_invariants} are isomorphisms and it then follows from the 5-lemma that $[\m_1, \m_2]_\He\cong [\m_1, \m_2]_\Hea^{\Omega}$ for all $\m_1, \m_2\in \Mod(\He)$ with $\m_1$ Gorenstein projective. In particular, the functor $\Ho(\Res^\He_\Hea)$ is faithful in that case. In general though, letting $\mathfrak{c}=\ker(\theta)$ (with $\theta$ as above) and recalling that $[\m_1, \m_2]_\He\cong \Ext^1_\He(\m_1, \mathfrak{c})$ and $[\m_1, \m_2]_\Hea^\Omega\cong \Ext^1_\Hea(\m_1, \mathfrak{c})^\Omega$ (cf.\ \Cref{Hom_Ho}), the exact sequence in \cite[\S 4.2, eq.(4.4)]{Abe22} shows that the kernel of the map $[\m_1, \m_2]_B\to [\m_1, \m_2]_A^{\Omega_A}$ is $H^1(\Omega, \Hom_\Hea(\m_1, \mathfrak{c}))$ and hence may fail to be zero when $\Omega$ has non-trivial group cohomology. In general, $\Ho(\Res^\He_\Hea)$ may thus fail to be faithful.
\end{enumerate}
\end{rem}

As an application, we highlight below a particular case of this construction which will be useful to us. In what follows, given $F\subseteq \overline{C}$ and $\omega\in \Omega$, we write $F\omega$ for the face satisfying $S_{F\omega}=\omega^{-1} S_F \omega$. In our earlier notation, we then have $\He_F^{\hat{\omega}}=\He_{F\omega}$ for any lift $\hat{\omega}\in\widetilde{\Omega}$.

\begin{cor}\label{action_Omega_faces} Suppose that $\m_1, \m_2\in \Mod(\He)$ and let $g\in [\m_1, \m_2]_\He$.  Let $\omega\in \Omega$ and fix a lift $\hat{\omega}$ in $\widetilde{\Omega}$. If we let $(g_F)_{F\subseteq \overline{C}}:=(\Ho(\Res^\He_{\He_F})(g))_{F\subseteq \overline{C}}$, then we have $g_{F\omega}=g_F\cdot\hat{\omega}$ for all $F\subseteq \overline{C}$.
\end{cor}

\begin{proof}
Apply \Cref{Ho_hom_invariants2} to $\hat{\omega}$, with $B=\He$ and $A=\He_F$, to get the result.
\end{proof}

We now apply the above construction to study isomorphism classes of simple supersingular modules in $\Ho(\He)$. We let $\chi$ be a supersingular character of $\Hea$ and consider two irreducible, finite dimensional representations $V$ and $V'$ of $\widetilde{\Omega}_\chi$ on which $T(\mathbb{F}_q)$ acts by $\chi$. Let $\m=(\chi\otimes V)\otimes_{\He_\chi}\He$ and $\m'=(\chi\otimes V')\otimes_{\He_\chi}\He$, and fix $g\in[\m, \m']_\He$. Recall that we have a decomposition $(\m')|_{\He_\chi}\cong \bigoplus_{\tilde{\omega}\in\widetilde{\Omega}_\chi\backslash \widetilde{\Omega}} \chi^{\tilde{\omega}}\otimes (V')^{\tilde{\omega}}$, and similarly for $\m|_{\He_\chi}$ (cf.\ \eqref{n_decomp}). Correspondingly, we have $\Ho(\Res^\He_{\He_\chi})(g)=(g_{\tilde{\omega}})_{\tilde{\omega}\in\widetilde{\Omega}_\chi\backslash \widetilde{\Omega}}$, where $g_{\tilde{\omega}}\in [\m, \chi^{\tilde{\omega}}\otimes (V')^{\tilde{\omega}}]_{\He_\chi}$.

For $F\subseteq \overline{C}$, we let $\widetilde{\Omega}_{F,\chi}:=\widetilde{\Omega}_F\cap\widetilde{\Omega}_\chi$ and denote by $\He_{F, \chi}$ the subalgebra of $\He_F^\dagger$ generated by $\He_F$ and $\{T_{\tilde{\omega}}\}_{\tilde{\omega}\in \widetilde{\Omega}_{F,\chi}}$. The key result we need is the following:

\begin{prop}\label{Omega_iso} Let $F\subseteq \overline{C}$ and suppose that $g$ is an isomorphism. Then the map $g_{\tilde{\omega}}$ above induces an $\He_{F, \chi}$-linear isomorphism $\chi^{\tilde{\omega}}\otimes V^{\tilde{\omega}}\cong \chi^{\tilde{\omega}}\otimes (V')^{\tilde{\omega}}$ for all $\tilde{\omega}\in\widetilde{\Omega}_\chi\backslash \widetilde{\Omega}$ such that $\chi^{\tilde{\omega}}|_{\He_F}$ is not projective. In particular, if $\chi|_{\He_F}$ is not projective then $V\cong V'$ as representations of $\widetilde{\Omega}_{F,\chi}$.
\end{prop}

\begin{proof}
For $\tilde{\omega}\in\widetilde{\Omega}_\chi\backslash \widetilde{\Omega}$, we let $\iota_{\tilde{\omega}}:\chi^{\tilde{\omega}}\otimes V^{\tilde{\omega}}\to \m$ denote the natural $\He_\chi$-linear inclusion. Furthermore, let $l$ denote the index of $\widetilde{\Omega}_{\chi}$ in $\widetilde{\Omega}$. By \eqref{n_decomp}, we may then regard elements of $[\m, \m']_{\He_F}\cong \bigoplus_{\tilde{\omega}, \tilde{\omega}'}[\chi^{\tilde{\omega}}\otimes V^{\tilde{\omega}}, \chi^{\tilde{\omega}'}\otimes (V')^{\tilde{\omega}'}]_{\He_F}$ as $l\times l$ matrices with entries in the Hom spaces $[\chi^{\tilde{\omega}}\otimes V^{\tilde{\omega}},\chi^{\tilde{\omega}'}\otimes (V')^{\tilde{\omega}'}]_{\He_F}$, and the image of any $g'\in[\m, \m']_\He$ under $\Ho(\Res^\He_{\He_F})$ is the matrix $(\Ho(\Res^{\He_\chi}_{\He_F})(g'_{\tilde{\omega}'}\circ\iota_{\tilde{\omega}}))_{\tilde{\omega},\tilde{\omega}'}$. Since the composite $g'_{\tilde{\omega}'}\circ\iota_{\tilde{\omega}}$ has image under $\Ho(\Res^{\He_\chi}_{\He_F})$ equal to zero for $\tilde{\omega}\neq\tilde{\omega}'$ by \Cref{characters_Haff_Ho}(ii) and \Cref{char_not_conj}, we deduce that $\Ho(\Res^\He_{\He_F})(g')$ is in fact a diagonal matrix. 

Applying the above to $g'=g$, we deduce that the diagonal matrix $\Ho(\Res^\He_{\He_F})(g)$ is an isomorphism and thus for all $\tilde{\omega}\in\widetilde{\Omega}_\chi\backslash \widetilde{\Omega}$ we have that $\Ho(\Res^{\He_\chi}_{\He_F})(g_{\tilde{\omega}}\circ\iota_{\tilde{\omega}})$ is an isomorphism in $\Ho(\He_F)$ as well. By applying \Cref{Ho_hom_invariants2} with $A=\He_F$, $B=\He_\chi$ and noting that $\widetilde{\Omega}_A=\widetilde{\Omega}_{F,\chi}$ in that case, we see that this isomorphism is in fact in $[\chi^{\tilde{\omega}}\otimes V^{\tilde{\omega}}, \chi^{\tilde{\omega}}\otimes (V')^{\tilde{\omega}}]^{\widetilde{\Omega}_{F,\chi}}_{\He_F}$.  Finally, by applying \Cref{basic_htpy_lemma} we see that this isomorphism is an actual $\widetilde{\Omega}_{F,\chi}$-equivariant isomorphism in the category $\Mod(\He_F)$ whenever $(\chi^{\tilde{\omega}})|_{\He_F}$ is not projective, as required.
\end{proof}

\begin{thm}\label{imperfect_theorem} Let $\m=(\chi\otimes V)\otimes_{\He_\chi}\He$ and $\m'=(\chi'\otimes V')\otimes_{\He_{\chi'}}\He$ be two simple supersingular $\He$-modules of infinite projective dimension over $\He$.  Write $\xi=\chi|_{T(\mathbb{F}_q)}$ and $\xi'=\chi'|_{T(\mathbb{F}_q)}$. Let $\Phi=\sqcup_{i=1}^r \Phi_i$ be the decomposition of the root system $\Phi$ into irreducible components, ordered such that $\mathrm{rk}(\Phi_1)\geq \cdots\geq \mathrm{rk}(\Phi_r)$, and correspondingly $S$ decomposes as $\sqcup_{i=1}^r S_i$. In the case where $\Phi_1$ is of rank two and $\Phi_i$ is of rank 1 for all $i\geq 2$, and if $\xi=\xi'$ and $S_\xi=S$, assume further that one of the following conditions holds:
\begin{itemize}
\item $\chi(T_{\hat{s}})\neq \chi'(T_{\hat{s}})$ for some $s\in S_i$ with $i\geq 2$; or
\item writing $S_1=\{s, s', s''\}$ and up to permuting the elements of $S_1$, we do not have both $\chi(T_{\hat{s}})\neq \chi'(T_{\hat{s}})$ and $\chi(T_{\hat{s'}})=\chi'(T_{\hat{s'}})\neq \chi(T_{\hat{s''}})=\chi'(T_{\hat{s''}})$.
\end{itemize}
Suppose that $\m\cong \m'$ in $\Ho(\He)$ and that there is a face $F\subseteq \overline{C}$ such that $\chi|_{\He_F}$ is not projective and $\widetilde{\Omega}_\chi\subseteq \widetilde{\Omega}_F$.  Then $\m\cong \m'$ in $\Mod(\He)$.
\end{thm}

\begin{proof}
Our assumptions say precisely that we are not in the setting of \Cref{characters_Haff_Ho}(i). We may thus apply \Cref{same_character} to deduce that $\chi$ and $\chi'$ are $\widetilde{\Omega}$-conjugate. By conjugating the pair $(\chi', V')$, we may therefore assume that $\chi=\chi'$. We can then apply \Cref{Omega_iso} to the face $F\subseteq \overline{C}$ given in the statement to deduce that $V\cong V'$ as representations of $\widetilde{\Omega}_{F,\chi}=\widetilde{\Omega}_\chi$. This shows that $\m\cong \m'$ in $\Mod(\He)$.
\end{proof}

In general, it may be that $\widetilde{\Omega}_{F,\chi}\neq \widetilde{\Omega}_\chi$ whenever $\chi|_{\He_F}$ is not projective so that the above Theorem does not apply (see e.g.\ \Cref{not_apply}). However, our techniques still apply in many situations. To illustrate this, we finish this paper by classifying isomorphism classes of simple supersingular $\He$-modules in $\Ho(\He)$ for $G$ an arbitrary finite product of general linear groups.

\subsection{Products of $\mathrm{GL}_n$'s} In this subsection we assume that $\mathbb{G}=\mathrm{GL}_{n_1}\times \cdots\times \mathrm{GL}_{n_r}\times \mathbb{T}'$, where $n_1\geq\ldots\geq n_r\geq 2$ and $\mathbb{T}'\cong \mathbb{G}_m^l$ is a split torus. Our motivation for considering this case is that all standard Levi subgroups of $\mathrm{GL}_n$ are of this form.

For $1\leq i\leq r$, write $T_i(\mathbb{F}_q)$ for the $\mathbb{F}_q$-rational points of the diagonal matrices in $\mathrm{GL}_{n_i}$. In any $T_i(\mathbb{F}_q)$ and for $t\in\mathbb{F}_q$, write $\diag(t)_j$ for the diagonal matrix with a $t$ in the $j$-th diagonal entry and 1's in every other diagonal entry. Since $\mathbb{G}$ is a product of split reductive groups, we correspondingly have $\Omega=\Omega_1\times\cdots\times\Omega_r\times\Omega_{T'}$ where $\Omega_{T'}\cong\Z^l$ is central in $W$ and, for each $1\leq i\leq r$, $\Omega_i= \langle\omega_i\rangle\cong\Z$ where $\omega_i$ is the image of the matrix $\begin{psmallmatrix}
0 & \cdots & \cdots & \varpi\\
1 & \ddots &  & \vdots \\
 & \ddots  & \ddots & \vdots\\
 & & 1 & 0
\end{psmallmatrix}\in\mathrm{GL}_{n_i}(F)$ in $W$. We write $\hat{\omega}_i$ for the image of that same matrix in $\widetilde{W}$. We also have a decomposition $S=S_1\sqcup \ldots\sqcup S_r$. We now let $\widetilde{\Omega}_{T'}\cong \Omega_{T'}\times T(\mathbb{F}_q)$ denote the pre-image of $\Omega_{T'}$ in $\widetilde{W}$, where here $\Omega_{T'}$ identifies with a central subgroup of $\widetilde{W}$. Note that $\hat{\omega}_i$ acts on $T_i(\mathbb{F}_q)$ by conjugation, with $\hat{\omega}_i^{-1}\diag(t_1,\ldots, t_{n_i})\hat{\omega}_i=\diag(t_2,\ldots, t_{n_i}, t_1)$, so that $\widetilde{\Omega}\cong \widetilde{\Omega}_{T'}\rtimes (\langle \hat{\omega}_1\rangle\times\cdots\times \langle \hat{\omega}_r\rangle)$.

We make a few preliminary remarks.  In order to study simple supersingular $\He$-modules, we need to consider the stabiliser $\widetilde{\Omega}_\chi$ of a supersingular character.

\begin{lem}\label{stab_chi} Let $\chi$ be a character of $\Hea$.  Then we have $\widetilde{\Omega}_\chi=\widetilde{\Omega}_{T'}\rtimes (\langle \hat{\omega}_1^{d_1}\rangle\times\cdots\times \langle \hat{\omega}_r^{d_r}\rangle)$ for some $d_1, \ldots, d_r\geq 1$ such that $d_i$ divides $n_i$ for all $1\leq i\leq r$. 
\end{lem}

\begin{proof}
For each $1\leq i\leq r$, we write $\He_{\aff, i}$ for the $k$-subalgebra of $\Hea$ generated by $\{T_{\hat{s}}\mid s\in S_i\}$ and $T_i(\mathbb{F}_q)$. Note that $\Hea\cong\He_{\aff, 1}\otimes_k \ldots\otimes_k\He_{\aff, r}$ as $k$-algebras, and that $\widetilde{\Omega}$ normalises each $\He_{\aff,i}$. Correspondingly, $\chi\cong \chi_1\otimes_k\cdots\otimes_k\chi_r$ where $\chi_i:=\chi|_{\He_{\aff, i}}$ ($1\leq i\leq r$).  We fix some $1\leq i\leq r$.  The action of $\widetilde{\Omega}$ on characters of $\He_{\aff,i}$ factors through $\Omega_i=\langle \omega_i\rangle\cong\Z$. Since $\hat{\omega}_i^{n_i}=\diag(\varpi,\ldots,\varpi)$ is central in $\widetilde{W}$, we deduce that the stabiliser $\widetilde{\Omega}_{\chi_i}$ equals $\widetilde{\Omega}_{T'}\rtimes\langle \hat{\omega}_i^{d_i}\rangle\times\prod_{j\neq i}\langle \hat{\omega}_j\rangle$ for some $d_i|n_i$. Since $\widetilde{\Omega}_{\chi}=\bigcap_{1\leq i\leq r}\widetilde{\Omega}_{\chi_i}$, we deduce the claimed result.
\end{proof}

We continue to consider a character $\chi$ of $\Hea$. If $V$ is any finite dimensional representation of $\widetilde{\Omega}_\chi$ on which $T(\mathbb{F}_q)$ acts by $\chi$, we then have that $\hat{\omega}_1^{d_1}, \ldots, \hat{\omega}_r^{d_r}$ (with $d_i$ as in \Cref{stab_chi}) and the generators of $\Omega_{T'}$ have a simultaneous eigenvector (since they commute) and this eigenvector spans a one dimensional subrepresentation. If $V$ is irreducible it must therefore be one dimensional. When $\chi$ is supersingular the corresponding simple supersingular $\He$-module $(\chi\otimes V)\otimes_{\He_\chi}\He$ then has $k$-dimension $d_1\cdots d_r$.

\begin{lem}\label{d_i_bigger_1}
Suppose that $\chi=(J, \xi)$ is a character of $\Hea$ and that $\widetilde{\Omega}_\chi=\widetilde{\Omega}_{T'}\rtimes (\langle \hat{\omega}_1^{d_1}\rangle\times\cdots\times \langle \hat{\omega}_r^{d_r}\rangle)$.
\begin{enumerate}
\item If there exist $s, s'\in S_i$ adjacent in the affine Dynkin diagram, for some $1\leq i\leq r$, such that $\chi(T_{\hat{s}})\neq \chi(T_{\hat{s'}})$ then $d_i>1$.
\item If $S\neq S_\xi$ then $d_i>1$ for every $i$ such that $S_i\nsubseteq S_\xi$.
\end{enumerate}
In particular, if $\chi$ is supersingular then $d_i>1$ for all $1\leq i\leq r$.
\end{lem}

\begin{proof}
For (i), since the action of $\omega_i$ permutes the nodes of the affine Dynkin diagram of $\mathrm{GL}_{n_i}$ cyclically, cf.\ \cite[\S 1.8]{IwMa65}, without loss of generality we may assume $s'=\omega_i s\omega_i^{-1}$. Thus we have that $\chi(T_{\hat{s}})\neq \chi^{\hat{\omega}_i}(T_{\hat{s}})$ and hence $\hat{\omega}_i\notin\widetilde{\Omega}_\chi$ as required.

For (ii), suppose that we have $s=s_\alpha\in S_i\setminus S_\xi$. Then there is some $1\leq j\leq n_i$ such that the subtorus $\alpha^\vee(\mathbb{F}_q^\times)$ of $T_i(\mathbb{F}_q)$ is equal to $\{\diag(t)_j\cdot \diag(t^{-1})_{j+1}\mid t\in \mathbb{F}_q\}$, where the subscripts are taken modulo $n_i$. The condition that $s\notin S_\xi$ then corresponds to saying that there exists $t\in \mathbb{F}_q$ such that $\xi(\diag(t)_j)\neq \xi(\diag(t)_{j+1})$. Since $\diag(t)_{j+1}=\hat{\omega}_i\diag(t)_j\hat{\omega}_i^{-1}$, this shows that $\xi^{\hat{\omega}_i}\neq \xi$ and thus $\hat{\omega}_i\notin\widetilde{\Omega}_\chi$ as required.

Finally, if $\chi$ is supersingular then it must be that for every $i$ such that $S_i\subseteq S_\xi$, there exist $s,s'\in S_i$ as in (i). The last claim thus follows immediately from (i) and (ii).
\end{proof}

\begin{rem}\label{det_rmk}
The arguments in the above proof show that $S_\xi=S$ if and only if for all $1\leq i\leq r$, all $1\leq j\leq n_i$ and all $t\in\mathbb{F}_q$, $\xi(\diag(t)_j)=\xi(\diag(t)_{j+1})$. In particular, we see that $\xi^{\tilde{\omega}}=\xi$ for every $\tilde{\omega}\in\widetilde{\Omega}$ whenever $S=S_\xi$. Since the image of $\xi$ in $k$ is contained in a subfield isomorphic to $\mathbb{F}_q$, we may view $\xi$ as an $\mathbb{F}_q$-valued character. Then we see from the above that $S=S_\xi$ if and only if, for all $1\leq i\leq r$, there is some $0\leq a_i\leq q-2$ such that $\xi|_{T_i(\mathbb{F}_q)}=\det^{a_i}$.
\end{rem}

Next, observe that for any face $F\subseteq\overline{C}$ with $F\neq C$ we have $\Omega_F\neq \Omega$. Indeed, we then have that $S_F\cap S_i$ is a non-empty proper subset of $S_i$ for some $i$ and so, picking some $s\in S_F\cap S_i$ and $s'\in S_i\setminus S_F$, the fact that the action of $\Omega_i$ on $S_i$ by conjugation is transitive implies that $\omega_i^{-d}s\omega_i^d=s'$ for some $d$ and thus that $\omega_i^d\notin \Omega_F$. 

We begin our study of isomorphism classes of simple supersingular $\He$-modules in $\Ho(\He)$ in this setting. We first show that when two simple supersingular modules have the same underlying $\Hea$-characters $\chi=\chi'$ (up to conjugation), then we can quickly reduce to the case where $S_\xi=S$.

\begin{lem}\label{case_S_not_Sxi}
Let $\chi$ be a supersingular character of $\Hea$ of infinite projective dimension, and let $\m=(\chi\otimes V)\otimes_{\He_\chi}\He$ and $\m'=(\chi\otimes V')\otimes_{\He_\chi}\He$ be two simple supersingular $\He$-modules. Suppose that $S_\xi\neq S$. Then $\m\cong \m'$ in $\Ho(\He)$ if and only if $\m\cong \m'$ in $\Mod(\He)$.
\end{lem}

\begin{proof}
We write $\widetilde{\Omega}_\chi=\widetilde{\Omega}_{T'}\rtimes (\langle \hat{\omega}_1^{d_1}\rangle\times\cdots\times \langle \hat{\omega}_r^{d_r}\rangle)$ and we pick $s\in S\setminus S_\xi$, say with $s\in S_i$. Note that $\Res^\Hea_{\He_{F}}(\chi)$ is not projective for any $F\subseteq \overline{C}$ such that $s\in S_F$ by \Cref{Koziol_projective_face}(i). In particular, this applies to the face $F$ given by $S_{F}=\{\omega_i^{-d_il}s\omega_i^{d_il}\mid 0\leq l<n_i/d_i\}$, which is well-defined by \Cref{d_i_bigger_1}(ii). Also note that this $F$ satisfies $\Omega_{F}=\langle\omega_i^{d_i}\rangle\times \Omega_{T'}\times \prod_{j\neq i}\Omega_j$ and thus we have $\widetilde{\Omega}_\chi\subseteq \widetilde{\Omega}_F$. We may therefore apply \Cref{imperfect_theorem} to obtain the result.
\end{proof}

\begin{rem}
When $\mathbb{G}$ is a reductive group whose root system is irreducible and not of type $A$, then the action of $\Omega$ on the affine Dynkin diagram can never be transitive (cf.\ \cite[\S 1.8]{IwMa65}). Hence, in those cases, we also have that two nonisomorphic simple supersingular $\He$-modules are still nonisomorphic in $\Ho(\He)$ whenever they contain a common supersingular character $\chi=(J, \xi)$ of $\Hea$ such that $S_\xi\neq S$. The argument is the same as the above proof: if we pick $s\in S\setminus S_\xi$ then the face $F$ such that $S_F$ is equal to the orbit of $s$ under $\Omega_\chi$, which is well-defined by the non-transitivity of the $\Omega$-action, allows us to conclude using \Cref{imperfect_theorem}.
\end{rem}

In the case where $S=S_\xi$, the following result will be key in order to apply \Cref{imperfect_theorem}.

\begin{lem}\label{case_d_neq_2} Let $\chi=(J,\xi)$ be a supersingular character of $\Hea$ such that $S=S_\xi$ and fix $1\leq i\leq r$. For all $2<d$ with $d$ dividing $n_i$, if such a $d$ exists, there is a face $F\subseteq \overline{C}$ such that $\Res^\Hea_{\He_F}(\chi)$ is not projective and $\Omega_F=\langle \omega_i^d\rangle\times \Omega_{T'}\times \prod_{j\neq i}\Omega_j$.
\end{lem}

\begin{proof}
The requirements that $\chi$ is supersingular and $S=S_\xi$ ensure that there exist $s_i, s_i'\in S_i$, adjacent in the affine Dynkin diagram, such that $\chi(T_{\hat{s_i}})\neq \chi(T_{\hat{s_i'}})$. If there is $d>2$ dividing $n_i$ then we automatically have $n_i>2$ and we may define $F\subseteq \overline{C}$ such that $S_F$ is equal to the orbit of $\{s_i, s_i'\}$ under the subgroup $\langle \omega_i^d\rangle$.  By construction we have $\Omega_F=\langle \omega_i^d\rangle\times \Omega_{T'}\times \prod_{j\neq i}\Omega_j$, and \Cref{Koziol_projective_face}(ii) gives that $\Res^\Hea_{\He_F}(\chi)$ is not projective as required.
\end{proof}

We will be unable to use the above with $d=d_j$ (some $j$) in the case that $d_i=2$ for all $1\leq i\leq r$. We deal with that case next, first introducing some notation. For each $1\leq i\leq r$ we let $[r]_i:=\{1\leq j\leq r\mid j\neq i\}$, and for any $\mathbf{j}\subseteq [r]_i$ we let $\hat{\omega}_{\mathbf{j}}:=\prod_{j\in \mathbf{j}}\hat{\omega}_j\in\widetilde{\Omega}$. We fix a generating set $\omega_{T',1}, \ldots, \omega_{T',l}$ for $\Omega_{T'}$ and write $\hat{\omega}_{T',1}, \ldots, \hat{\omega}_{T',l}$ for the corresponding lifts to $\widetilde{\Omega}_{T'}$.

\begin{prop}\label{case_d_equal_2} Suppose that for every $1\leq i\leq r$ there is some $m_i\geq 1$ such that $n_i=2m_i$, and let $\chi=(J,\xi)$ be a supersingular character of $\Hea$ with $\widetilde{\Omega}_\chi=\widetilde{\Omega}_{T'}\rtimes (\langle \hat{\omega}_1^2\rangle\times\cdots\times \langle \hat{\omega}_r^2\rangle)$, $S=S_\xi$ and $\pd_\Hea(\chi)=\infty$. Suppose that $V$ and $V'$ are two one-dimensional representations of $\widetilde{\Omega}_\chi$ on which $T(\mathbb{F}_q)$ acts via $\chi$, and let $\m=(\chi\otimes V)\otimes_{\He_\chi}\He$ and $\m'=(\chi\otimes V')\otimes_{\He_{\chi}}\He$. If $\m\cong \m'$ in $\Ho(\He)$ then $\m\cong \m'$ in $\Mod(\He)$.
\end{prop}

\begin{proof}
For each $1\leq i\leq r$, we write $\lambda_i$ and $\lambda'_i$ for the scalar by which $\hat{\omega}_i^2$ acts on $V$ and $V'$, respectively. Similarly, for each $1\leq j\leq l$, we write $\nu_j$ and $\nu_j'$ for the scalar by which $\hat{\omega}_{T',j}$ acts on $V$ and $V'$ respectively. From our assumption on $\chi$ and by \Cref{facts_H_F}(iii), there is a face $F\subseteq \overline{C}$ for which $\chi|_{\He_F}$ is not projective. Note that we have $\widetilde{\Omega}_{T'}\subseteq\widetilde{\Omega}_{F,\chi}$ and $\hat{\omega}_i^2\in \widetilde{\Omega}_{F,\chi}$ for each $i$ such that $n_i=2$. Applying \Cref{Omega_iso} we deduce that $\nu_j=\nu_j'$ for all $j$ and that $\lambda_i=\lambda'_i$ whenever $n_i=2$.

We are thus left to show that $\lambda_i=\lambda'_i$ for all $i$ for which $n_i>2$, i.e.\ $m_i>1$. Pick $s_0\in S_i$ and write $s_j=\omega_i^{-j}s_0\omega_i^j$ for $1\leq j\leq n_i-1$.  By our assumptions on $\chi$, without loss of generality we may assume that $J\cap S_i=\{s_{2j}\mid 0\leq j\leq m_i-1\}$.  Similarly $\chi^{\hat{\omega}_i}=(J^{\omega_i},\xi)$ then satisfies $J^{\omega_i}\cap S_i=\{s_{2j+1}\mid 0\leq j\leq m_i-1\}$. Note that any $\widetilde{\Omega}$-conjugate of $\chi$ is non-projective when restricted to any $\He_F$ such that $S_F\cap S_i$ contains two adjacent simple affine reflections in the affine Dynkin diagram, cf.\ \Cref{Koziol_projective_face}(ii).  Fix an isomorphism $f:\m\to \m'$ in $\Ho(\He)$, write $V=kv$ and $V'=kv'$, and for each $F\subseteq\overline{C}$ let $f_F:\m\to \m'$ be a lift in $\Mod(\He_F)$ of $\Ho(\Res^\He_{\He_F})(f)$.

Pick a prime factor $\ell$ of $m_i$. Similarly to the proof of \Cref{case_d_neq_2}, we let $F$ be the face given by taking $S_F$ to be the $\langle \omega_i^{2\ell}\rangle$-orbit of $\{s_0, s_1\}$ so that $\chi|_{\He_{F}}$ and $(\chi^{\hat{\omega}_i})|_{\He_{F}}$ are not projective. We also have that $\chi|_{\He_{F\omega_i^{-1}}}$ and $(\chi^{\hat{\omega}})|_{\He_{F\omega_i^{-1}}}$ are not projective.

Next, since $s_0\in S_F\cap S_{F\omega_i^{-1}}$ it follows that $\abs{S_F\cup S_{F\omega_i^{-1}}} <2\cdot\abs{S_F}=4m_i/\ell\leq n_i$ and hence there is a face $F'\subseteq \overline{F}$ with $S_{F'}=S_F\cup S_{F\omega_i^{-1}}$. Again both $\chi$ and $\chi^{\hat{\omega}_i}$ have non-projective restrictions to $\He_{F'}$. Also note that for all $\tilde{\omega}\in\widetilde{\Omega}$ we have that $\chi^{\tilde{\omega}}|_{\He_{F'}}$ is equal to either $\chi|_{\He_{F'}}$ or $(\chi^{\hat{\omega}_i})|_{\He_{F'}}$. By the proof of \Cref{Omega_iso} and \Cref{basic_htpy_lemma} we deduce that, with respect to the bases $\bigcup_{\mathbf{j}\subseteq [r]_i}\{v\otimes T_{\hat{\omega}_{\mathbf{j}}}, v\otimes T_{\hat{\omega}_{\mathbf{j}}}T_{\hat{\omega}_i}\}$ and $\bigcup_{\mathbf{j}\subseteq [r]_i}\{v'\otimes T_{\hat{\omega}_{\mathbf{j}}}, v'\otimes T_{\hat{\omega}_{\mathbf{j}}}T_{\hat{\omega}_i}\}$ of $\m$ and $\m'$ respectively, $f_{F'}$ is a block diagonal matrix $A$ whose blocks are $2\times 2$-matrices, indexed by the subsets $\mathbf{j}\subseteq [r]_i$, equal to $\begin{pmatrix}
    \mu_\mathbf{j} & 0\\
    0 & \mu'_\mathbf{j}
\end{pmatrix}$ for some scalars $\mu_\mathbf{j}, \mu'_\mathbf{j}\in k^\times$. Similarly, $f_F$ and $f_{F\omega_i^{-1}}$ are also given by block diagonal matrices of the same form as above, and \Cref{basic_htpy_lemma} ensures that both of these are actually equal to $A$ since $f_F\sim (f_{F'})|_{\He_F}$ and $f_{F\omega_i^{-1}}\sim (f_{F'})|_{\He_{F\omega_i^{-1}}}$.

Now, with respect to the above bases, the action of $T_{\hat{\omega}_i}$ on $\m$ and $\m'$ is given by two block diagonal matrices $B$ and $B'$ respectively, whose $2\times 2$-blocks are indexed by the subsets $\mathbf{j}\subseteq [r]_i$ as above but with each block equal to $\begin{pmatrix}
    0 & 1\\
    \lambda_i & 0
\end{pmatrix}$ and $\begin{pmatrix}
    0 & 1\\
    \lambda_i' & 0
\end{pmatrix}$ respectively (where we view these matrices as acting on the right). Using that $f_F=f_{F\omega_i^{-1}}\cdot\omega$, cf.\ \Cref{action_Omega_faces} and \Cref{basic_htpy_lemma}, we deduce that $A=B^{-1}AB'$. Looking at each $2\times2$-block, this tells us that
$$
\begin{pmatrix}
    \mu_\mathbf{j} & 0\\
    0 & \mu_\mathbf{j}'
\end{pmatrix}=\begin{pmatrix}
    0 & \lambda_i^{-1}\\
    1 & 0
\end{pmatrix}\begin{pmatrix}
    \mu_\mathbf{j} & 0\\
    0 & \mu_\mathbf{j}'
\end{pmatrix}\begin{pmatrix}
    0 & 1\\
    \lambda_i' & 0
\end{pmatrix}=\begin{pmatrix}
    \mu_\mathbf{j}'\lambda_i'\lambda_i^{-1} & 0\\
    0 & \mu_\mathbf{j}
\end{pmatrix}
$$
for all $\mathbf{j}\subseteq [r]_i$. We deduce that $\lambda_i=\lambda_i'$ as required.
\end{proof}

\begin{rem}\label{not_apply}
Assuming that $(n_1, \ldots, n_r)\neq(2,\ldots, 2)$, note that \Cref{case_d_equal_2} is the only situation where we could not appeal to \Cref{imperfect_theorem}. Indeed, by the description of $\chi$ as given in the proof we obtain from \Cref{Koziol_projective_face}(ii) that a face $F\subseteq\overline{C}$ satisfies that $\chi|_{\He_F}$ is not projective if and only if $S_F$ contains two adjacent nodes in the affine Dynkin diagram. If those belong to $S_i$, say, then $F\omega_i^2\neq F$ and we deduce that $\widetilde{\Omega}_\chi\nsubseteq\widetilde{\Omega}_F$.
\end{rem}

Using all of the above, we can finally establish a classification theorem for isomorphism classes in $\Ho(\He)$ of simple supersingular $\He$-modules.

\begin{thm}\label{last_thm}
Suppose that $\mathbb{G}=\mathrm{GL}_{n_1}\times \cdots\times \mathrm{GL}_{n_r}\times \mathbb{T}'$, where $n_1\geq\ldots\geq n_r\geq 2$ and $\mathbb{T}'\cong \mathbb{G}_m^l$ is a split torus. Let $\m=(\chi\otimes V)\otimes_{\He_\chi}\He$ and $\m'=(\chi'\otimes V')\otimes_{\He_{\chi'}}\He$ be two simple supersingular $\He$-modules of infinite projective dimension over $\He$, where $\chi=(J, \xi)$ and $\chi'=(J', \xi')$.
\begin{enumerate}
\item Suppose that $(n_1, n_2, \ldots, n_r)\neq(3,2,\ldots, 2)$. Then $\m\cong \m'$ in $\Ho(\He)$ if and only if $\m\cong \m'$ in $\Mod(\He)$.
\item Suppose that $(n_1, n_2, \ldots, n_r)=(3,2,\ldots, 2)$ and $\m\not\cong \m'$ in $\Mod(\He)$. Then $\m\cong \m'$ in $\Ho(\He)$ if and only if $\xi=\xi'$, $S=S_\xi$, $J'\subseteq J$ such that $\abs{J\cap S_1}=2$ and $\abs{J'\cap S_1}=1$ (up to permuting $\m$ and $\m'$), and $V\cong V'$ as representations of $\widetilde{\Omega}_\chi=\widetilde{\Omega}_{\chi'}$.
\end{enumerate}
\end{thm}

\begin{proof}
We write $\widetilde{\Omega}_\chi=\widetilde{\Omega}_{T'}\rtimes(\langle \tilde{\omega}_1^{d_1}\rangle\times \cdots \times \langle \tilde{\omega}_r^{d_r}\rangle)$ and $\widetilde{\Omega}_{\chi'}=\widetilde{\Omega}_{T'}\rtimes(\langle \tilde{\omega}_1^{d'_1}\rangle\times \cdots \times \langle \tilde{\omega}_r^{d'_r}\rangle)$. For $1\leq i\leq r$, let $\lambda_i, \lambda_i'\in k^\times$ for the scalars by which $\tilde{\omega}_i^{d_i}$ and $\tilde{\omega}_i^{d_i'}$ act on $V$ and $V'$ respectively. Similarly, for each $1\leq j\leq l$, we write $\nu_j$ and $\nu_j'$ for the scalar by which $\hat{\omega}_{T',j}$ acts on $V$ and $V'$ respectively. Note that $\pd_\Hea(\chi)=\pd_\Hea(\chi')=\infty$ by \Cref{Koziol_ss_thm}.

We first claim that the following holds:
\begin{equation}\label{last_claim}
\text{If $\chi$ and $\chi'$ are $\widetilde{\Omega}$-conjugate, then $\m\cong \m'$ in $\Ho(\He)\iff\m\cong \m'$ in $\Mod(\He)$.}
\end{equation}
One direction being clear, we suppose that $\m\cong \m'$ in $\Ho(\He)$ and we note that, replacing the pair $(\chi', V')$ by some conjugate if necessary, we may in fact assume that $\chi=\chi'=(J, \xi)$, say.

If $S\neq S_\xi$ then \Cref{case_S_not_Sxi} implies our claim, so that we only need to consider the case $S=S_\xi$. If $d_i>2$ for some $1\leq i\leq r$ then \Cref{case_d_neq_2} and \Cref{imperfect_theorem} together imply that $\m\cong \m'$ in $\Mod(\He)$. If $d_i=2$ for all $1\leq i\leq r$, we get $\m\cong \m'$ in $\Mod(\He)$ from \Cref{case_d_equal_2}. This concludes the proof of \eqref{last_claim}.

When $(n_1, n_2, \ldots, n_r)\neq(3,2,\ldots, 2)$, \Cref{same_character} implies that $\chi$ and $\chi'$ must be $\widetilde{\Omega}$-conjugate for $\m$ and $\m'$ to be isomorphic in $\Ho(\He)$. Thus (i) follows immediately from \eqref{last_claim}. We are therefore left to consider the case where $(n_1, n_2, \ldots, n_r)=(3,2,\ldots, 2)$ and $\chi, \chi'$ are not $\widetilde{\Omega}$-conjugate. Then \Cref{same_character} again implies that, up to permuting $\m$ and $\m'$, we must have $\chi$ and $\chi'$ as in the statement in order for $\m$ and $\m'$ to be isomorphic in $\Ho(\He)$. Note that we then have $\widetilde{\Omega}_\chi=\widetilde{\Omega}_{\chi'}=\widetilde{\Omega}_{T'}\rtimes (\langle \hat{\omega}_1^3\rangle\times \langle \hat{\omega}_2^2\rangle\times\cdots \times \langle \hat{\omega}_r^2\rangle)$.

We must finally show that under these assumptions on $\chi$ and $\chi'$, $\m\cong \m'$ in $\Ho(\He)$ if and only if $\lambda_i=\lambda_i'$ and $\nu_j=\nu_j'$ for all $1\leq i\leq r$ and all $1\leq j\leq l$. Assume first that these equalities all hold. We argue similarly to the proof of \Cref{characters_Haff_Ho}. By the product decomposition of $\mathbb{G}$, we have an algebra isomorphism $\He\cong \He_1\otimes_k\He_{\geq 2}$ where $\He_1$ is the pro-$p$ Iwahori-Hecke algebra of $\mathrm{GL}_3(\mathfrak{F})$ and $\He_{\geq 2}$ is the pro-$p$ Iwahori-Hecke algebra of $\mathrm{GL}_2(\mathfrak{F})\times\cdots\times \mathrm{GL}_2(\mathfrak{F})\times T'$ (with $r-1$ $\mathrm{GL}_2$-factors). There is a similar decomposition $\Hea=\He_{\aff, 1}\otimes_k\He_{\aff, \geq 2}$ and correspondingly we have $\chi=\chi_1\otimes_k \chi_{\geq 2}$. Also, we have a decomposition $\widetilde{\Omega}_\chi=\widetilde{\Omega}_{\chi,1}\times \widetilde{\Omega}_{\chi,\geq 2}$, and we may thus write $V=V_1\otimes_k V_{\geq 2}$ with $V_1=kv_1$ and $V_{\geq 2}=kv_{\geq 2}$ one-dimensional representations of $\widetilde{\Omega}_{\chi,1}$ and $\widetilde{\Omega}_{\chi,\geq 2}$ respectively. 

Write $v=v_1\otimes v_{\geq 2}\in V$. By considering the basis $\bigcup_{\mathbf{j}\subseteq [r]_1}\{v\otimes T_{\hat{\omega}_{\mathbf{j}}}, v\otimes T_{\hat{\omega}_{\mathbf{j}}}T_{\hat{\omega}_1}, v\otimes T_{\hat{\omega}_{\mathbf{j}}}T_{\hat{\omega}^2_1}\}$ of $\m$, we see by inspection that $\m\cong \m_1\otimes_k \m_{\geq 2}$ where $\m_1\cong(\chi_1\otimes V_1)\otimes_{\He_{\chi, 1}}\He_1$ and $\m_{\geq 2}\cong(\chi_{\geq 2}\otimes V_{\geq 2})\otimes_{\He_{\chi, \geq 2}}\He_{\geq 2}$. We note that these are simple supersingular modules of $\He_1$ and $\He_{\geq 2}$ respectively, and that $\pd_{\He_{\geq 2}}(\m_{\geq 2})<\infty$ by \Cref{Koziol_ss_thm}. Completely analogously, we have a decomposition $\m'\cong \m'_1\otimes_k \m'_{\geq 2}$ where the factors are simple supersingular modules over $\He_1$ and $\He_{\geq 2}$ respectively. Furthermore, by our assumptions on $\chi$ and $\chi'$ we have $\m_{\geq 2}\cong \m'_{\geq 2}$. We then deduce from \Cref{GL3_exception} and \Cref{tensor_lemma} that $\m\cong \m'$ in $\Ho(\He)$ as required.

Conversely, suppose that $\m\cong \m'$ in $\Ho(\He)$. Write $\underline{\lambda}=(\lambda_1, \ldots, \lambda_r)$,  $\underline{\lambda'}=(\lambda'_1, \ldots, \lambda'_r)$, $\underline{\nu}=(\nu_1, \ldots, \nu_l)$ and $\underline{\nu'}=(\nu'_1, \ldots, \nu'_l)$. The modules $\m$ and $\m'$ are determined uniquely, up to isomorphism, in $\Mod(\He)$ by the triples $(\chi, \underline{\lambda}, \underline{\nu})$ and $(\chi', \underline{\lambda'}, \underline{\nu'})$, and we write $\m=(\chi, \underline{\lambda}, \underline{\nu})$ and $\m'=(\chi', \underline{\lambda'}, \underline{\nu'})$ for short. Using this notation, we then have by the converse direction above that 
$$
(\chi, \underline{\lambda}, \underline{\nu})\cong (\chi', \underline{\lambda'}, \underline{\nu'})\cong (\chi, \underline{\lambda'}, \underline{\nu'})
$$
in $\Ho(\He)$. By \eqref{last_claim} this implies $\underline{\lambda}=\underline{\lambda'}$ and $\underline{\nu}=\underline{\nu'}$, as required.
\end{proof}

\bibliographystyle{abbrv}
\bibliography{bibforme}

\end{document}